\documentclass[11pt]{article}
\usepackage{amsmath,latexsym,amssymb,amsfonts,amsbsy, amsthm}
\usepackage{graphicx}
\usepackage{color,xcolor}
\usepackage{ulem}

\def\inte#1{
	\displaystyle\mathop{#1\kern0pt}^\circ }

\let\grad\nabla

\def\virgp{\raise 2pt\hbox{,}}
\def\cdotpv{\raise 2pt\hbox{;}}

\def\eqdefa{\buildrel\hbox{\footnotesize def}\over =}

\def\C{\mathop{\bf C\kern 0pt}\nolimits}
\def\DD{\mathop{\bf D\kern 0pt}\nolimits}
\def\K{\mathop{\bf K\kern 0pt}\nolimits}
\def\N{\mathop{\bf N\kern 0pt}\nolimits}
\def\Q{\mathop{\bf Q\kern 0pt}\nolimits}
\def\R{\mathop{\bf R\kern 0pt}\nolimits}
\def\SS{\mathop{\bf S\kern 0pt}\nolimits}
\def\ZZ{\mathop{\bf Z\kern 0pt}\nolimits}
\def\TT{\mathop{\bf T\kern 0pt}\nolimits}

\newcommand{\beq}{\begin{equation}}
	\newcommand{\eeq}{\end{equation}}
\newcommand{\ben}{\begin{eqnarray}}
	\newcommand{\een}{\end{eqnarray}}
\newcommand{\beno}{\begin{eqnarray*}}
	\newcommand{\eeno}{\end{eqnarray*}}

\newtheorem{thm}{Theorem}[section]
\newtheorem{lem}{Lemma}[section]
\newtheorem{rmk}{Remark}[section]

\newtheorem{prop}{Proposition}[section]
\renewcommand{\theequation}{\thesection.\arabic{equation}}

\begin{document}
	
	\title{Construction of low regularity strong solutions to the viscous surface wave equations}
	
	\author{ Guilong Gui \footnote{School of Mathematics and Computational Science, Xiangtan University, Xiangtan 411105, China. Email: {\tt glgui@amss.ac.cn}.}\quad
		Yancan Li \footnote{School of Mathematics and Computational Science, Xiangtan University, Xiangtan 411105, China. Email: {\tt yancan\_li@126.com}.}
	}
	
	\date{}
	\maketitle
	
	\begin{abstract}
		We construct in the paper the low-regularity strong solutions to the viscous surface wave equations in anisotropic Sobolev spaces. Here we use the Lagrangian structure of the system to homogenize the free boundary conditions, and establish a new iteration scheme on a known equilibrium domain to get the low-regularity strong solutions, in which no nonlinear compatibility conditions on the initial data are required.
		
		\vskip 0.3cm
		
	\end{abstract}

	\vskip 0.2cm
	
	
	\renewcommand{\theequation}{\thesection.\arabic{equation}}
	\setcounter{equation}{0}

	\section{Introduction}

	\subsection{Formulation in Eulerian Coordinates}
	We consider in this paper the local existence of time-dependent flows of an
	viscous incompressible fluid in a moving domain $\Omega(t)$
	with an upper free surface $\Sigma_{F}(t)$ and a fixed bottom $\Sigma_B$
	\begin{equation}\label{VFS-eqns-1}
		\begin{cases}
			&\partial_t u + (u\cdot \grad) u + \grad p- \nu\Delta\,u=-g\,e_1 \quad\mbox{in} \quad \Omega(t),\\
			&\grad \cdot\, u=0 \quad \mbox{in} \quad \Omega(t),\\
			& (p\,\mathbb{I}-\nu\mathbb{D}(u))n(t)=p_{\mbox{\tiny atm}}n(t)\quad \mbox{on} \quad \Sigma_{F}(t),\\
			&  \mathcal{V}(\Sigma_{F}(t))=u\cdot n(t) \quad \mbox{on} \quad \Sigma_{F}(t),\\
			&u|_{\Sigma_B}=0,
		\end{cases}
	\end{equation}
	where we denote $n(t)$ the outward-pointing unit normal on $\Sigma_F(t)$, $\mathbb{I}$ the $3\times3$ identity matrix, $(\mathbb{D}u)_{ij}=\partial_iu^j+\partial_ju^i$ twice the symmetric gradient of the velocity $u=(u^1, u^2, u^3)$, $u^i=u^i(t, x_1, x_2, x_3)$, $i, j=1, 2, 3$, the constant $g > 0$ stands for the strength of gravity, $e_1=(1,0,0)^T$, and $\nu > 0$ is the constant coefficient of viscosity. We denote $ \mathcal{V}(\Sigma_{F}(t))$ the outer-normal velocity of the free surface $\Sigma_{F}(t)$. The tensor $(p\,\mathbb{I} -\nu\mathbb{D}(u))$ is known as the viscous stress tensor. Equation $\eqref{VFS-eqns-1}_1$ is the conservation of momentum, where gravity is the only external force, which points in the negative $x_1$ direction (as the vertical direction); the second equation in \eqref{VFS-eqns-1} means the fluid is incompressible; Equation $\eqref{VFS-eqns-1}_3$ means the fluid satisfies the kinetic boundary condition on the free boundary $\Sigma_F(t)$, where $p_{atm}$ stands for the atmospheric pressure, assumed to be constant. the kinematic boundary condition $\eqref{VFS-eqns-1}_4$ states that the free boundary
	$\Sigma_F(t)$ is moving with speed equal to the normal component of the fluid velocity; $\eqref{VFS-eqns-1}_5$ implies that the fluid is no-slip, no-penetrated on the fixed bottom boundary. Here the effect of surface tension is neglected on the free surface.
		
	For convenience, it is natural to
	subtract the hydrostatic pressure from $p$ in the usual way by adjusting the actual pressure $p$ according to $\widetilde{p}=p+g\,x_1-p_{atm}$, and still denote the new pressure $\widetilde{p}$ by $p$ for simplicity, so that after substitution the gravity term in $\eqref{VFS-eqns-1}_1$ and the atmospheric pressure term in $\eqref{VFS-eqns-1}_3$  are eliminated. A gravity term appears in $\eqref{VFS-eqns-1}_3$.

	The problem can be equivalently stated as follows. Given an initial domain $\Omega_0 \subset \mathbb{R}^3$ bounded by a bottom surface $\Sigma_{B}$, and a top surface $\Sigma_F(0)$, as well as an initial velocity field $u_0$, where the upper boundary does not touch the bottom, we wish to find for each $t\in [0, T]$ a
	domain $\Omega(t)$, a velocity field $u(t, \cdot)$ and pressure $p(t, \cdot)$ on $\Omega(t)$, and a transformation
	$\bar{\eta}(t, \cdot):\,\Omega_0\rightarrow \mathbb{R}^3$ so that
	\begin{equation}\label{VFS-eqns-12}
		\begin{cases}
			&\Omega(t)=\bar{\eta}(t, \Omega_0), \quad\bar{\eta}(t, \Sigma_B)=\Sigma_B,\\
			&\partial_t\bar{\eta}=u\circ\bar{\eta},\\
			&\partial_t u + (u\cdot \grad) u + \grad p- \nu\Delta\,u=0 \quad \text{in} \quad \Omega(t),\\
			&\grad \cdot\, u=0 \quad \mbox{in} \quad \Omega(t),\\
			& \left((p-g\,x_1)\,\mathbb{I}-\nu\mathbb{D}(u)\right)n(t)=0\quad \mbox{on} \quad \Sigma_F(t),\\
			&u|_{\Sigma_B}=0,\\
			&u|_{t=0}=u_0,\quad \bar{\eta}|_{t=0}=x
		\end{cases}
	\end{equation}
	
	The conditions on the initial domain $\Omega_0$ are as follows: Let the equilibrium domain $\Omega \subset \mathbb{R}^3$ be the horizontal infinite slab
	\begin{equation}\label{def-domain-1}
		\begin{split}
			&\Omega=\{x=(x_1, x_h)|-\underline{b}<x_1<0,\quad x_h\in\mathbb{R}^2\}
		\end{split}
	\end{equation}
	with the bottom $\Sigma_b=\{x_1=-\underline{b}\}$ and the top surface  $\Sigma_0=\{x_1=0\}$, where the positive constant $\underline{b}$ will be the depth of the fluid at infinity. We assume that $\Omega_0$ is the image of $\Omega$
	under a diffeomorphism $\overline{\sigma}: \Omega \rightarrow  \Omega_0$, where $\overline{\sigma}(\Sigma_b)=\Sigma_B$, $\overline{\sigma}(\Sigma_0)=\Sigma_F(0)$, $\overline{\sigma}$ is of the form $\overline{\sigma}(x)=x+\xi_0(x)$, and $\xi_0$ satisfies $\xi_0^1 \in H^s(\Sigma_0)$, and $\nabla\xi_0,\,\nabla_h^{s-1} \xi_0\in H^1(\Omega)$ with $s >2$,  $\xi_0^1$ stands for the vertical component of the vector $\xi_0$, $\nabla=(\partial_1, \partial_2, \partial_3)^T$, and $\nabla_h=(\partial_2, \partial_3)^T$.

	\subsection{Known results}

Many mathematicians have contributions in the free boundary problems of the incompressible Navier-Stokes equations. By using Lagrangian coordinates transformation, Solonnikov \cite{Solonnikov-1977} proved the local well-posedness of the viscous surface problem in H\"{o}lder spaces in a bounded domain whose entire boundary is a free surface. In a horizontal infinite domain, Beale \cite{Beale-1981} obtained the local well-posedness of the free boundary Navier-Stokes equations without surface tension in $L^2$-based space-time Sobolev spaces, which was extended to $L^p$-based space-time Sobolev spaces by Abels \cite{Abels-2005}. Furthermore, Beale \cite{Beale-1984} introduced the flattening transformation to get the global well-posedness of the viscous surface problem with surface tension in the space-time Sobolev spaces, and then the decay properties of solutions to incompressible viscous surface waves with surface tension where studied by Beale-Nishida \cite{Beale-Nishida-1985} and Hataya \cite{Hataya2011}.  Sylvester \cite{Sylvester-1990} and Tani-Tanaka \cite{Tani-Tanaka-1995} also discussed the well-posedness of the incompressible viscous surface problem without surface tension.  Bae \cite{Bae-2011} proved global well-posedness with surface tension using energy methods rather than a Beale–Solonnikov framework. For a bounded mass of fluid with surface tension, local well-posedness was proved by Coutand and Shkoller \cite{CS-2003}. Recently, by using the flattening transformation, Guo and Tice \cite{Guo-Tice-1, Guo-Tice-2} constructed the local-in-time solutions of the viscous surface wave equations by using the geometric structure in the Eulerian coordinates, and then got its global well-posedness and the algebraic decay rate of the solutions by introducing the two-tier energy method. Wu \cite{Wu2014} extended their local well-posedness result from small data to general data. Wang, Tice and Kim \cite{WTK-2014} also considered the global well-posedness and decay for two layers fluid. Ren, Xiang and Zhang \cite{Ren-X-Zhang-2019} proved the low-regularity local well-posedness of the system in Sobolev spaces in Eulerian coordinates. Wang \cite{Wang-2020} studied the anisotropic decay of the global solution by using the same method. More recently, the first author of this paper developed a mathematical approach to establish global well-posedness of the incompressible viscous surface waves based on the Lagrangian framework \cite{Gui2020}, where no nonlinear compatibility conditions on the initial data were required.
 
Nevertheless, almost all well-posedness results in previous works were established for the initial data which has
high regularity or some compatibility conditions on the initial data are needed. In the present case, a natural and important question is whether a corresponding well-posedness result can be obtained with low regularity and without compatibility conditions of the  
accelerated velocity on the initial data.

	\subsection{Formulation of the system in Lagrangian Coordinates}
	
	Let us now, in more detail, introduce the Lagrangian coordinates in which the free boundary becomes fixed.

	Let $\eta\eqdefa \bar{\eta} \circ \bar{\sigma}$ be a position of the fluid particle $x$ in the equilibrium domain $\Omega$ at time $t$ so that
	\begin{equation}\label{def-flowmap-1}
		\begin{cases}
			&\frac{d}{dt}\eta(t, x)=u(t, \eta(t, x)), \quad t>0, \, x\in \Omega,\\
			&\eta|_{t=0}=x+\xi_0(x), \quad x\in \Omega,
		\end{cases}
	\end{equation}
	then the displacement $\xi(t, x)\eqdefa \eta(t, x)-x$ satisfies
	\begin{equation}\label{def-flowmap-2}
		\begin{cases}
			&\frac{d}{dt}\xi(t, x)=u(t, x+\xi(t, x)),\\
			&\xi|_{t=0}=\xi_0.
		\end{cases}
	\end{equation}
	We define Lagrangian quantities the velocity $v$ and the pressure $q$ in fluid as (where $x=(x_1, x_2, x_3)^T\in \Omega$): $v(t, x)\eqdefa u(t, \eta(t, x))$, $ q(t, x)\eqdefa p(t, \eta(t, x))$.
	Denote the Jacobian of the flow map $\eta$ by $J \eqdefa \mbox{det}(D\eta)$.
	Define $\mathcal{A} \eqdefa  (D\eta)^{-T}$, then according to definitions of the flow map $\eta$ and the displacement $\xi$, we may get the identities:
	\begin{equation}\label{flow-map-identity-1}
		\mathcal{A}_{i}^k \partial_{k} \eta^j=\mathcal{A}_{k}^j \partial_{i} \eta^k=\delta_i^j, \quad \partial_k(J\mathcal{A}_{i}^k)=0,\quad\partial_{i} \eta^j=\delta_i^j+\partial_{i} \xi^j, \quad \mathcal{A}_{i}^j=\delta_i^j-\mathcal{A}_{i}^k \partial_k\xi^j.
	\end{equation}
	Set $a_{ij}\eqdefa\,J\,\mathcal{A}_i^j$. Simple computation implies that
	$ J=1+\grad \cdot\xi+\mathcal{B}_{00}+\mathcal{B}_{000}$, where $\mathcal{B}_{00}:=\partial_1\xi^1 \nabla_h\cdot \xi^h-\partial_1\xi^h\cdot\nabla_h\xi^1+\nabla_h^{\perp}\xi^2\cdot\nabla_h\xi^3$, $\mathcal{B}_{000}:=\partial_1\xi^1 \nabla_h^{\perp}\,\xi^2\cdot \nabla_h\,\xi^3+\partial_1\xi^2\nabla_h^{\perp}\,\xi^3\cdot \nabla_h\,\xi^1+\partial_1\xi^3 \nabla_h^{\perp}\,\xi^1\cdot \nabla_h\,\xi^2$
	with $\xi^h\eqdefa (\xi^2, \xi^3)^T$, $\nabla_h\eqdefa (\partial_2,\, \partial_3)^T$, $\nabla_h^{\perp}\eqdefa (-\partial_3,\, \partial_2)^T$.
	If the displacement $\xi$ is sufficiently small in an appropriate Sobolev space, then the flow mapping $\eta$ is a
	diffeomorphism from $\Omega$ to $\Omega(t)$, which makes us to switch back and forth from Lagrangian
	to Eulerian coordinates.

	Next, we give some useful equations which we often use in what follows.
	
	Since $\mathcal{A}(D\eta)^T=I$, differentiating it with respect to $t$ and $x$ once yields
	\begin{equation}\label{identity-Lagrangian-1}
		\begin{split}
			&\partial_t \mathcal{A}_{i}^j=-\mathcal{A}_{k}^j\mathcal{A}_{i}^{m} \partial_{m}v^k,\,\partial_{s} \mathcal{A}_{i}^j=-\mathcal{A}_{k}^j\mathcal{A}_{i}^{m} \partial_{m}\partial_{s}\xi^k,
		\end{split}
	\end{equation}
	where we used the fact $\partial_t\eta=v$ in the first equation in \eqref{identity-Lagrangian-1}.
	Whence differentiate the Jacobian determinant $J$, we get
	\begin{equation}\label{identity-deri-J}
		\begin{split}
			&\partial_t J =J \mathcal{A}_{i}^j \partial_j v^i, \quad \partial_{k} J =J \mathcal{A}_{i}^j \partial_j\partial_{k} \xi^i.
		\end{split}
	\end{equation}
	Moreover, we may verify the following Piola identity:
	\begin{equation}\label{identity-Piola}
		\begin{split}
			&\partial_j (J \mathcal{A}_{i}^j) =0 \quad \forall \,i = 1, 2, 3.
		\end{split}
	\end{equation}
	Here and in what follows, the subscript notation for vectors and tensors as well as the Einstein summation convention has been adopted unless otherwise specified.

	Under Lagrangian coordinates, we may introduce the differential operators with
their actions given by $(\nabla_{\mathcal{A}}f)_i=\mathcal{A}_i^j  \partial_jf$, $ \mathbb{D}_{\mathcal{A}} (v)=\nabla_{\mathcal{A}} v+(\nabla_{\mathcal{A}} v)^T$, $\Delta_{\mathcal{A}} f=\nabla_{\mathcal{A}}\cdot \nabla_{\mathcal{A}} f$, so the Lagrangian version of the system \eqref{VFS-eqns-12} can be written on the fixed reference domain $\Omega$ as
	\begin{equation}\label{EB-multi-1.1}
		\begin{cases}
			\partial_{t}\xi = v &\quad \text{ in } \Omega,\\
			\partial_t v + \nabla_{\mathcal{A}}q -\nu \nabla_{\mathcal{A}}\cdot \mathbb{D}_{\mathcal{A}}(v) = 0 &\quad \text{ in } \Omega,\\
			\nabla_{\mathcal{A}} \cdot v=0 &\quad \text{ in } \Omega,\\
			(q-g\,\xi^1)\mathcal{N} - \nu \mathbb{D}_{\mathcal{A}}(v)\mathcal{N} =0; &\quad \text{ on } \Sigma_0,\\
			v=0, & \quad \text{ on } \Sigma_b,\\
			\xi|_{t=0}=\xi_0,v|_{t=0} = v_0,
		\end{cases}
	\end{equation}
	where $n_0=e_1 = (1,0,0)^T$ is the outward-pointing unit normal vector on the interface $\Sigma_0$, $\mathcal{N}:= J\mathcal{A} n_0$ stands for the outward-pointing normal vector on the moving interface $\Sigma_F(t)$.

\subsection{Main result and ideas}\label{subsec-our-result}

The aim of this paper is to establish the local well-posedness of he system \eqref{EB-multi-1.1} with low regularity data in anisotropic Sobolev spaces.
We state the main result as follows, which proof will be presented in Section \ref{sect-4}.
	\begin{thm}\label{thm-main-1}
		$($Local well-posedness$)$ Let $s > 2$. Assume $(v_0, \xi_0)$ satisfies $\Lambda_h^{s-1}v_0 \in { _{0}{H}^1}(\Omega)$,  $\nabla_{J_0\mathcal{A}_0} \cdot v_0=0$, and $\Lambda_h^{s-1}\nabla\xi_0 \in H^1(\Omega)$, $\xi_0^1 \in H^{s}(\Sigma_0)$. There exists a positive constant $\epsilon_0$ and $0<T \leq\min\{1,\delta_0(\|\xi_0^1\|_{H^{s}{(\Sigma_0)}}^2 + \|\Lambda_h^{s-1} v_0\|_{H^1(\Omega)}^2  )^{-1}\}$ for some $\delta_0>0$ such that the system \eqref{EB-multi-1.1} has a unique solution $(v, \xi, q)$ depending continuously on the initial data $(v_0, \xi_0)$ which satisfies
		\begin{equation}\label{solution-space-1}
		\begin{split}
			&\Lambda_h^{s-1}\nabla\xi \in \mathcal{C}([0,T]; H^1(\Omega)),\quad \xi^1 \in \mathcal{C}([0,T];H^{s}(\Sigma_0)), \\
			&\Lambda_h^{s-1}v \in \mathcal{C}([0,T]; {_0}{H}^1(\Omega)) \cap L^2([0,T]; H^2(\Omega)),\quad \Lambda_h^{s-1} q \in L^2([0,T]; H^1(\Omega))
		\end{split}
		\end{equation}
		provided
		\[
		\|\Lambda_h^{s-1}\nabla\xi_0\|_{H^1(\Omega)}^2   \leq \epsilon_0.
		\]
		Moreover, for any $t\in (0,T]$, the solution satisfies the estimate
	\begin{equation}\label{est-main}
		\begin{split}
			\sup_{0\leq \tau\leq t}(\|\Lambda_h^{s-1}(\nabla\xi,\,v) \|_{H^1(\Omega)}^2 &+ \|\xi^1 \|_{H^s{(\Sigma_0)}}^2)  +   \|\Lambda_h^{s-1} (\nabla v ,q)\|^2_{L_t^2(H^1(\Omega))} \\
			&\leq C (\|\Lambda_h^{s-1}(\nabla\xi_0,\,v_0)\|_{H^1(\Omega)}^2 + \|\xi_0^1\|_{H^{s}{(\Sigma_0)}}^2).
		\end{split}
\end{equation}
	\end{thm}
\begin{rmk}\label{rmk-thm1-1}
We may readily check that, if the maximal time $T^{\ast} $of the existence of the solution $(\xi,\,v,\,q)$ in Theorem \ref{thm-main-1} is finite: $T^{\ast} <+\infty $, then
\begin{equation}\label{loc-blow-crit}
\begin{split}
\lim_{t  \nearrow T^{\ast}}(\|\Lambda_h^{s-1}(\nabla\xi(t),\,v(t))\|_{H^1(\Omega)}+\|\xi^1(t) \|_{H^{s}(\Sigma_0)}+\|J^{-1}\|_{L^\infty(\Omega)})=+\infty.
\end{split}
\end{equation}
Moreover, if in addition $\dot{\Lambda}^{-\lambda}_h(\xi_0^1,\,v_0) \in L^2(\Sigma_0)\times H^1(\Omega)$ for some $\lambda\in (0, 1)$, then $\dot{\Lambda}^{-\lambda}_h(\xi,\,v) \in \mathcal{C}([0, T_0]; L^2(\Sigma_0)\times H^1(\Omega))$.
\end{rmk}
\begin{rmk}\label{rmk-thm1-2}
We say the strong solution $(\xi, v)$ obtained in Theorem \ref{thm-main-1} low-regular, which means that the surface function $\xi^1 \in H^s(\Sigma_0)$ with $s>2$ has the minimal regularity guaranteeing  the regular condition $\nabla_h\xi^1 \in L^{\infty}(\Sigma_0)$.
 \end{rmk}
 
 \begin{rmk}
      The classical parabolic theory of non-homogeneous boundary conditions \cite{LionsMagenes1972} is not enough to support us to complete the proof of Theorem \ref{thm-linear-problim-11} (see Sect. \ref{sect-3}) about the well-posedness of the linear system \eqref{EB-multi-3.1} which is crucial in our construction of local solutions. Therefore, we can't employ the approach in \cite{Beale-1981} to construct the solutions. In addition, due to the lack of compatibility conditions, the construction of local solutions is also different from \cite{Guo-Tice-1,Wu2014,Ren-X-Zhang-2019}. And another difference from \cite{Guo-Tice-1,Wu2014,Ren-X-Zhang-2019} is that our construction of local solutions is based on the Lagrange framework.
    \end{rmk}
    \begin{rmk}
      Motivated by \cite{Danchin2020}, we establish the well-posedness of the linear system \eqref{EB-multi-3.1} to construct low regularity strong solutions. But there are two main difficulties to overcome. The first one is,  for the linear system \eqref{EB-multi-3.1} in the finite depth case , we cannot deal with the non-homogeneous boundary conditions of \eqref{EB-multi-3.1}
      as directly as for the infinite case in \cite{Danchin2020}. Therefore, we need to construct another suitable divergence-free vector field to correct. On the other hand, our results allow the presence of gravity and a small perturbation in the initial domain, so we need to establish a new iteration scheme on a known equilibrium domain.
    \end{rmk}
 
 To prove Theorem \ref{thm-main-1}, we will first linearize the surface wave equations \eqref{EB-multi-1.1} in Lagrangian coordinates to the system \eqref{fixed-point} or \eqref{EB-multi-3.1}. To solve \eqref{EB-multi-3.1} with non-homogeneous boundary conditions, we will first remove the divergence of the velocity in Lagrangian coordinates and homogenize free boundary conditions in \eqref{EB-multi-3.1}, and then solve the result system with homogeneous boundary conditions in Proposition \ref{linear-system}, from which, we may obtain Theorem \ref{thm-linear-problim-11}. With Theorem \ref{thm-linear-problim-11} in hand, we establish a new iteration scheme on a known equilibrium domain to get the low-regularity strong solutions (see Sect. \ref{sect-4}).
 
 \subsection{Plan of the paper}

The rest of the paper is organized as follows. Section \ref{sect-2} introduces some basic estimates, which will be heavily used in this paper. In Section \ref{sect-3}, we  prove Theorem \ref{thm-linear-problim-11} for the linear system  \eqref{EB-multi-3.1} in four steps, which is crucial in the proof of Theorem \ref{thm-main-1}. Then, we use the fixed point method to prove Theorem \ref{thm-main-1} in Section \ref{sect-4}. Section \ref{sect-appendix-1} is an appendix which gives the expressions of $B$ forms associated with the system \eqref{fixed-point}.

\subsection{Notations}

\medbreak  Let us end this introduction by some notations that will be used in all that follows.

For operators $A,B,$ we denote $[A, B]=AB-BA$ to be the  commutator of $A$ and $B.$ For~$a\lesssim b$, we mean that there is a uniform constant $C,$ which may be different on different lines, such that $a\leq Cb$.  The notation $a\thicksim b$ means both $a\lesssim b$ and $b\lesssim a$. Throughout the paper, the subscript notation for vectors and tensors as well as the Einstein summation convention has been adopted unless otherwise specified, Einstein's summation convention means that repeated Latin indices $i,\, j,\, k$, etc., are summed from $1$ to $3$, and repeated Greek indices $\alpha, \,\beta, \,\gamma$, etc., are summed from $2$ to $3$. In the vector $v=(v^1, v^2, v^3)^T$, we denote the vertical component by $v^1$, and its horizontal component by $v^h=(v^2, v^3)^T$.

\renewcommand{\theequation}{\thesection.\arabic{equation}}
\setcounter{equation}{0}

\section{ Preliminary estimates}\label{sect-2}

Let's first recall some basic estimates, which will be heavily used in the rest of the paper.

\begin{lem}[\cite{Alinhac-1986}, Theorem 2.61 in \cite{BCD}]\label{lem-composition-1}
Let $f$ be a smooth function on $\mathbb{R}$ vanishing at $0$, $s_1$, $s_2$ be two positive
real number, $s_1\in (0, 1)$, $s_2>0$. If $u$ belongs to $\dot{H}^{s_1}(\mathbb{R}^2) \cap \dot{H}^{s_2}(\mathbb{R}^2)\cap L^{\infty}(\mathbb{R}^2)$, then so does $f\circ u$,
and we have
\begin{equation*}
\|f\circ u\|_{\dot{H}^{s_i}} \leq C(f', \|u\|_{L^{\infty}})\|u\|_{\dot{H}^{s_i}}\quad \text{for} \quad i=1, 2.
\end{equation*}
\end{lem}

\begin{lem}(\cite{Alinhac-1986, BCD})\label{prop-composition-2}
Let $f$ be a smooth function such that $f'(0)=0$ . Let $s > 0$. For any couple $(u, v)$ of functions in $B^{s}_{p, r}\cap L^{\infty}$,
the function $f\circ v-f\circ u$ then belongs to $\dot{H}^{s}\cap L^{\infty}$ and
\begin{equation*}\begin{split}
&\|f\circ v-f\circ u\|_{\dot{H}^{s}} \leq C_{f''}(\|u\|_{L^{\infty}},  \|v\|_{L^{\infty}})\\
&\qquad \times (\|u-v\|_{\dot{H}^{s}}\sup_{\tau\in [0, 1]}\|v+\tau(u-v)\|_{L^{\infty}}+\|u-v\|_{L^{\infty}}\sup_{\tau\in [0, 1]}\|v+\tau(u-v)\|_{\dot{H}^{s}}),
\end{split}
\end{equation*}
where $C_{f''}(\cdot, \cdot)$ is a uniformly continuous function on $[0, \infty) \times [0, \infty)$ depending only on $f''$.
\end{lem}

\begin{lem}[Classical product laws in Sobolev spaces \cite{BCD}]\label{lem-product-law-1}
Let $s_0>2$, $p_1,\,p_2 \in [2, +\infty)$, $q_1,\, q_2 \in (2, +\infty]$, $\frac{1}{p_1}+\frac{1}{q_1}=\frac{1}{p_2}+\frac{1}{q_2}=\frac{1}{2}$, there hold
\begin{equation*}\label{product-law-1}
\begin{split}
&\|f\,g\|_{\dot{H}^{s_1+s_2-1}(\mathbb{R}^2)} \lesssim  \|f\|_{\dot{H}^{s_1}(\mathbb{R}^2)} \|g\|_{\dot{H}^{s_2}(\mathbb{R}^2)} \quad \forall \,\, s_1,\,s_2 \in (-1, 1),\, s_1+s_2>0,\\
&\|f\,g\|_{\dot{H}^{s_0-1}(\mathbb{R}^2)} \lesssim  \|f\|_{\dot{H}^{s_0-1}(\mathbb{R}^2)} \|g\|_{\dot{H}^{s_0-1}(\mathbb{R}^2)},\\
&\|\dot{\Lambda}_h^{\sigma}(f\,g)\|_{L^2(\mathbb{R}^2_h)} \lesssim
\|\dot{\Lambda}_h^{\sigma}f\|_{L^{p_1}(\mathbb{R}^2_h)}\|g\|_{L^{q_1}(\mathbb{R}^2_h)}
+\|\dot{\Lambda}_h^{\sigma}g\|_{L^{p_2}(\mathbb{R}^2_h)}\|f\|_{L^{q_2}(\mathbb{R}^2_h)} \quad \forall \,\, \sigma>0,\\
&\|\dot{\Lambda}_h^{s_1}(f\,g)\|_{L^2(\mathbb{R}^2_h)} \lesssim
\|\dot{\Lambda}_h^{s_1}f\|_{L^2(\mathbb{R}^2_h)}
\|\dot{\Lambda}_h^{s_0-1, s_2} g\|_{L^2(\mathbb{R}^2_h)} \, \forall\,\, s_1 \in (-1, 1],\,s_2\in (0, 1).
\end{split}
\end{equation*}
\end{lem}

	\begin{lem}[Korn-Poincar\'{e}'s inequality, Lemma 2.7 in \cite{Beale-1981}]\label{Korn's_inequality}
		Let $\Omega$ be defined by \eqref{def-domain-1}, then there exists a positive constant $C_{korn}$, independent of $u$, such that
		\begin{equation*}
			\|u\|_{H^1(\Omega)}\leq C_{korn}\|\mathbb{D}(u)\|_{L^2(\Omega)}
		\end{equation*}
		for all $u\in H^1(\Omega)$ with $u=0$ on $\Sigma_b$.
	\end{lem}
	
	\begin{lem}[Poincar\'{e}'s inequality, Lemma A.10 in \cite{Guo-Tice-1}]\label{lem-korn-2}
		Let $\Omega$ be defined by \eqref{def-domain-1}, then there holds
		\begin{equation*}\label{Poincare-1-1}
			\begin{split}
				&\|f\|_{L^2(\Omega)} \lesssim\,\| f\|_{L^2(\Sigma_0)}+\| \partial_1f\|_{L^2(\Omega)},\quad \|f\|_{L^{\infty}(\Omega)} \lesssim\,\| f\|_{L^{\infty}(\Sigma_0)}+\| \partial_1f\|_{L^{\infty}(\Omega)}.
			\end{split}
		\end{equation*}
	\end{lem}
In order to extend the interface boundary forms of $\partial_1v$ to the interior domain of the fluid, let us first introduce $\mathcal{H}(f)$ as the harmonic extension of $f|_{\Sigma_0}$ into $\Omega$:
\begin{equation*}\label{harmonic-ext-xi-1-1}
\begin{cases}
   &\Delta \mathcal{H}(f)=0 \quad\mbox{in} \quad \Omega,\\
   &\mathcal{H}(f)|_{\Sigma_0}=f|_{\Sigma_0},\quad \mathcal{H}(f)|_{\Sigma_b}=0.
\end{cases}
\end{equation*}

	\begin{lem}[Lemma 4.5 in \cite{Gui2020}]\label{harmonic-extension}
    Let $s\in \mathbb{R}$, $r\geq 2$, and $\Omega$ be defined by \eqref{def-domain-1}. For the harmonic extension $\mathcal{H}(f)$ of $f=f(t, x_1, x_h)$ (with $t \in\mathbb{R}^+$, $( x_1, x_h) \in \Omega$), there hold
\begin{equation*}\label{est-harmonic-ext-GL-1}
\begin{split}
 &\|\mathcal{H}(f)\|_{H^1(\Omega)} \lesssim \|f\|_{H^{\frac{1}{2}}(\Sigma_0)}, \quad \|\dot{\Lambda}_h^{s}\mathcal{H}(f)\|_{H^1(\Omega)} \lesssim \|\dot{\Lambda}_h^{s}f\|_{H^{\frac{1}{2}}(\Sigma_0)}, \\
 &\|\dot{\Lambda}_h^{s}\mathcal{H}(f)\|_{H^{2}(\Omega)} \lesssim \|\dot{\Lambda}_h^{s}f\|_{H^{\frac{3}{2}}(\Sigma_0)},\quad\|\mathcal{H}(f)\|_{H^{r}(\Omega)} \lesssim \|f\|_{H^{r-\frac{1}{2}}(\Sigma_0)},\\
 &\|\partial_t\mathcal{H}(f)\|_{H^1(\Omega)} \lesssim \|\partial_tf\|_{H^{\frac{1}{2}}(\Sigma_0)}.
\end{split}
\end{equation*}
    \end{lem}
	\begin{lem}[Lemma 2.8 in \cite{Beale-1981}]\label{Beale-lemma-2.8}
		Let $\Omega$ be defined by \eqref{def-domain-1}, and for $f_1\in H^{r-2}(\Omega)$, $f_2\in H^{r-\frac{3}{2}}(\Sigma_0)$, $2\leq r\leq 5$, there is a unique solution $u\in H^r(\Omega)$ of
		\begin{equation*}
			\begin{cases}
				\Delta u=f_1  &\mbox{in} \quad \Omega,\\
				\partial_1 u=f_2 &\mbox{on}  \quad \Sigma_0,\\
				u=0 &\mbox{on} \quad \Sigma_b,\\
			\end{cases}
		\end{equation*}
		and the solution satisfies the estimate
		\begin{equation*}
			\|u\|_{H^r(\Omega)}\leq C(\|f_1\|_{H^{r-2}(\Omega)}+\|f_2\|_{H^{r-\frac{3}{2}}(\Sigma_0)}).
		\end{equation*}
	\end{lem}
	\begin{lem}[Chapter 6 in \cite{Evans}]\label{Stokes-lemma-1}
		Let $\Omega$ be defined by \eqref{def-domain-1}, and for $f\in H^{r-2}(\Omega)$,  $r\geq 2 $, there is a unique solution $u\in H^r(\Omega)$ of
		\begin{equation*}
			\begin{cases}
				\Delta u=f  &\text{in} \quad\Omega,\\
				u=0 &\text{on}  \quad\partial\Omega,\\
			\end{cases}
		\end{equation*}
		and the solution satisfies the estimate
		\begin{equation*}
			\|u\|_{H^r(\Omega)}\leq C\|f\|_{H^{r-2}(\Omega)}.
		\end{equation*}
	\end{lem}

	\begin{lem}[\cite{Agmon-D-N-1964}]\label{lem-3-5}
		Let $\Omega$ be defined by \eqref{def-domain-1}, $r\geq2$. Suppose $v,q$ solve
		\[
		\begin{cases}
			-\nu \nabla \cdot \mathbb{D} (v) + \nabla q = f_1\in H^{r-2}(\Omega) \\
			\nabla \cdot v=f_2\in H^{r-1}(\Omega),\\
			q n_0 - \nu \mathbb{D} (u)n_0 =f_3 \in H^{r-\frac{3}{2}}(\Sigma_0),\\
			u|_{\Sigma_b}=0  ,\\
		\end{cases}
		\]
		Then there holds
		\[
		\|\nabla v\|_{H^{r-1}(\Omega)}^2+\|\nabla q\|_{H^{r-2}(\Omega)}^2\lesssim\|f_1\|_{H^{r-2}(\Omega)}^2+\|f_2\|_{H^{r-1}(\Omega)}^2+\|f_3\|_{H^{r-\frac{3}{2}}(\Sigma_0)}^2.
		\]
	\end{lem}
	
	\begin{lem}[Lions-Magenes Lemma, Lemma 1.2 on page 176 of \cite{Lions-M}]\label{Lions-M-lem}
		Let $X_0$, $X$ and $X_1$ be three Hilbert spaces with $X_0\subset X\subset X_1$. Suppose that $X_0$ is continuously embedded in $X$ and that $X$ is continuously embedded in $X_1$, and $X_1$ is the dual space of $X_0$. Suppose for some $T>0$ that $u\in L^2([0,T];X_0)$ is such that $u_t\in L^2([0,T];X_1)$. Then $u$ is almost everywhere equal to a function continuous from $[0,T]$ into $X$, and moreover the following equality holds in the sense of scalar distribution on $(0,T)$:
		\begin{equation*}
			\frac{1}{2}\frac{\mathrm{d}}{\mathrm{d}t}\|u\|_{X}^2=\langle u_t, u\rangle.
		\end{equation*}
	\end{lem}
	
\renewcommand{\theequation}{\thesection.\arabic{equation}}
\setcounter{equation}{0}
	\section{A linearied system}\label{sect-3}

	We are concerned with the time-dependent linear problem
	\begin{equation}\label{EB-multi-3.1}
		\begin{cases}
			\partial_t\xi^1-u^1=0&\quad \text{ on } [0,T]\times\Sigma_0,\\
			\partial_t u -\nu \nabla \cdot \mathbb{D} (u) + \nabla p = F_1 & \quad \text{ in } [0,T]\times\Omega,\\
			\nabla \cdot u=F_2 &\quad \text{ in } [0,T]\times\Omega,\\
			(p -g\,\xi^1)\,n_0 - \nu \mathbb{D} (u)n_0 =F_3 &\quad \text{ on } [0,T]\times\Sigma_0,\\
			u=0 & \quad \text{ on } [0,T]\times\Sigma_b,\\
			u|_{t=0} = u_0 & \quad \text{ in } \Omega,\\
			\xi^1|_{t=0}=\xi_0^1 & \quad \text{ on } \Sigma_0
		\end{cases}
	\end{equation}
with $F_2:= B^{1}:\nabla  K^{(1)} $, and $F_3:= B^{2}:\nabla_h K^{(2)} $.

For this linear system, we can get the following theorem.
\begin{thm}\label{thm-linear-problim-11}
For given $s>2$, $T>0$, assume that $F_1, \, B^{i}, \, K^{(i)}$ (with $i=1, 2$) satisfy that $\Lambda_h^{s-1}(F_1,\,K^{(i)}_t) \in L^2([0,T]; L^2(\Omega))$, $\Lambda_h^{s-1}B^{i} \in \mathcal{C}([0,T]; H^1(\Omega))$, $\Lambda_h^{s-1}K^{(i)} \in \mathcal{C}([0,T];  {_0}H^1(\Omega))$, and $\Lambda_h^{s-1}(B^{i}_t,\,\nabla\,K^{(i)}) \in L^2([0,T]; H^1(\Omega) )$, and $(\xi^1_0,\,u_0)$ satisfies that $\xi^1_0\in H^s(\Sigma_0)$, $\Lambda_h^{s-1}u_0\in {_0}H^1(\Omega)$, and $\nabla\cdot u_0={F_2}|_{t=0}$, then there exists a unique solution $(\xi^1, u, p)$ of \eqref{EB-multi-3.1} satisfying
\[
\begin{split}
			&\xi^1 \in \mathcal{C}([0,T]; H^{s}(\Sigma_0)), \quad \Lambda_h^{s-1}u  \in \mathcal{C}([0,T]; {H}^1(\Omega)) \cap L^2([0,T]; H^2(\Omega)),\\
			&\Lambda_h^{s-1}u_t  \in  L^2([0,T]; L^2(\Omega)), \quad \Lambda_h^{s-1}p  \in L^2([0,T]; H^1(\Omega)).
\end{split}
\]
		Moreover,  the solution satisfies the estimate
	\begin{equation*}\label{estimate-all}
		\begin{split}
			& \sup_{t \in [0, T]}\{\|\xi^1 \|^2_{H^s{(\Sigma_0)}}+\|\Lambda_h^{s-1} u \|^2_{H^1(\Omega)} \}  + \left(\|\Lambda_h^{s-1} (\nabla u ,p)\|^2_{L^2_T(H^1(\Omega))}+ \|\Lambda_h^{s-1} u_t\|^2_{L^2_T( L^2(\Omega)) }\right)\\
			&\leq C_0\bigg(  \|\xi_0^1\|^2_{H^{s}{(\Sigma_0)}}  + \|\Lambda_h^{s-1} u_0\|^2_{H^1(\Omega)}+\|\Lambda_h^{s-1} F_1\|^2_{L^2_T( L^2(\Omega))}\\
			&  +  \|\Lambda_h^{s-1}(B^{1}, B^{2})\|^2_{L^\infty_T( H^1(\Omega))}(\|\Lambda_h^{s-1}(\nabla K^{(1)}, \nabla K^{(2)} )\|^2_{L^2_T( H^1(\Omega))}
  +\|\Lambda_h^{s-1}(K^{(1)}_t, K^{(2)}_t)\|^2_{L^2_T( L^2(\Omega))})\\
			&+\|\Lambda_h^{s-1} (K^{(1)}, K^{(2)})\|^2_{L^\infty_T( H^1(\Omega))}(\|\Lambda_h^{s-1}(B^{1}, B^{2})\|^2_{L^\infty_T( H^1(\Omega))}+\|\Lambda_h^{s-1} (B^{1}_t, B^{2}_t)\|^2_{L^2_T(  L^2(\Omega))})\bigg) .
		\end{split}
\end{equation*}
	\end{thm}

	\begin{proof} The proof of Theorem \ref{thm-linear-problim-11} is divided into four steps.

	{\bf Step 1: Removing the divergence of the velocity}\par

	In order to remove the divergence of the velocity in the third equation of \eqref{EB-multi-3.1}, we look for $U=(U^1,U^2,U^3)^T$ under the form $U = \nabla  \Psi$ solving
	\begin{equation*}\label{U_cd}
		\nabla \cdot U=F_2,
	\end{equation*}
so $\Psi$ satisfies
	\begin{equation}\label{Psi}
		\Delta \Psi=F_2 \quad \text{in} \quad \Omega.
	\end{equation}
	We impose $\Psi$ by the boundary condition
	\begin{equation}\label{Psi_bc}
		\Psi|_{\Sigma_0}=0,\quad \mbox{and} \quad \partial_1\Psi|_{\Sigma_b}=0
	\end{equation}
	to make sure that $U^1=0$ on $\Sigma_b$.
	
Applying Lemma \ref{Beale-lemma-2.8} to the system \eqref{Psi}-\eqref{Psi_bc} yields
	\begin{equation*}
\|\Lambda_h^{s-1}\nabla\Psi\|_{H^1(\Omega)}\lesssim\|\Lambda_h^{s-1}F_2\|_{L^2(\Omega)},\quad
			\|\Lambda_h^{s-1}\nabla^2\Psi\|_{H^1(\Omega)}\lesssim\|\Lambda_h^{s-1}F_2\|_{H^1(\Omega)},
	\end{equation*}
and then
	\begin{equation}\label{U-est-1}
\|\Lambda_h^{s-1}U\|_{H^1(\Omega)}\lesssim\|\Lambda_h^{s-1}F_2\|_{L^2(\Omega)},\quad
			\|\Lambda_h^{s-1}U\|_{H^2(\Omega)}\lesssim\|\Lambda_h^{s-1}F_2\|_{H^1(\Omega)}.
	\end{equation}
Thanks to Lemma \ref{lem-product-law-1}, we have
	\begin{equation*}\label{NABLAPSI}
		\begin{split}
		\|\Lambda_h^{s-1}F_2\|_{L^2(\Omega)}&\lesssim \|\Lambda_h^{s-1}B^{1}\|_{L^{\infty}_{x_1}L^2_h}\|\Lambda_h^{s-1}\nabla K^{(1)}\|_{L^2}\lesssim \|\Lambda_h^{s-1}B^{1}\|_{H^1}\|\Lambda_h^{s-1}\nabla K^{(1)}\|_{L^2},\\
			\|\Lambda_h^{s-1}F_2\|_{H^1(\Omega)}&\lesssim \|\Lambda_h^{s-1}F_2\|_{L^2(\Omega)}+\|\Lambda_h^{s-1}(\nabla\,B^{1}:\nabla \,K^{(1)},\, B^{1}:\nabla^2 \,K^{(1)})\|_{L^2}\\
&\lesssim \|\Lambda_h^{s-1}B^{1}\|_{H^1}\|\Lambda_h^{s-1}\nabla K^{(1)}\|_{H^1},
		\end{split}
	\end{equation*}
which along with \eqref{U-est-1} implies
	\begin{equation}\label{U-est-2}
		\begin{split}
		&\|\Lambda_h^{s-1}U\|_{L^{\infty}_T(H^1(\Omega))}+\|\Lambda_h^{s-1}U\|_{L^{2}_T(H^2(\Omega))}\\
&\lesssim \|\Lambda_h^{s-1}B^{1}\|_{L^{\infty}_T(H^1(\Omega))}(\|\Lambda_h^{s-1}\nabla K^{(1)}\|_{L^{\infty}_T(L^2(\Omega))}+\|\Lambda_h^{s-1}\nabla K^{(1)}\|_{L^{2}_T(H^1(\Omega))}).
		\end{split}
	\end{equation}
On the other hand, since
	\begin{equation*}
		\begin{split}			
&\|\Lambda_h^{s-1}\nabla\Psi_t\|^2_{L^2(\Omega)}=-\langle\Lambda_h^{s-1}\Delta\Psi_t,\Lambda_h^{s-1}\Psi_t\rangle\\
			&=(\Lambda_h^{s-1}(B^{1} K^{(1)}_t),\Lambda_h^{s-1}\nabla\Psi_t)_{L^2(\Omega)}
-(\Lambda_h^{s-1}(B^{1}_t:\nabla K^{(1)}-(\nabla\cdot B^{1})\cdot K^{(1)}_t),\Lambda_h^{s-1}\Psi_t)_{L^2(\Omega)},
		\end{split}
	\end{equation*}
one has
	\begin{equation*}
		\begin{split}
			&\|\Lambda_h^{s-1}\nabla\Psi_t\|_{L^2(\Omega)}\leq
			\|\Lambda_h^{s-1} B^{1}_t\|_{L^2(\Omega)}\|\Lambda_h^{s-1} K^{(1)}\|_{H^1(\Omega)}+
			\|\Lambda_h^{s-1}  B^{1}\|_{H^1(\Omega)}\|\Lambda_h^{s-1} K^{(1)}_t\|_{L^2(\Omega)},
		\end{split}
	\end{equation*}
which follows that
	\begin{equation}\label{U-est-3}
		\begin{split}
			\|\Lambda_h^{s-1}U_t\|_{L^{2}([0, T]; L^2(\Omega))}\lesssim &
			\|\Lambda_h^{s-1} B^1_t\|_{L^{2}([0, T]; L^2(\Omega))}\|\Lambda_h^{s-1} K^{(1)}\|_{L^{\infty}([0, T]; H^1(\Omega))}\\
&+\|\Lambda_h^{s-1}   B^1\|_{L^{\infty}([0, T]; H^1(\Omega))}\|\Lambda_h^{s-1} K^{(1)}_t\|_{L^{2}([0, T]; L^2(\Omega))}.
		\end{split}
	\end{equation}
Hence, combining \eqref{U-est-2} and \eqref{U-est-3}, we have
	\begin{equation}\label{U-est-4}
		\begin{split}
&\|\Lambda_h^{s-1}U\|^2_{L^{\infty}_T(H^1(\Omega))}+\|\Lambda_h^{s-1}U\|^2_{L^{2}_T(H^2(\Omega))}+\|\Lambda_h^{s-1}U_t\|^2_{L^{2}_T( L^2(\Omega))}\\
&\lesssim (\|\Lambda_h^{s-1}\nabla K^{(1)}\|^2_{L^{\infty}_T(L^2(\Omega))}+\|\Lambda_h^{s-1}\nabla K^{(1)}\|^2_{L^{2}_T(H^1(\Omega))}+\|\Lambda_h^{s-1} K^{(1)}_t\|^2_{L^{2}_T(L^2(\Omega))})\\
&\qquad \quad  \times \|\Lambda_h^{s-1}B^{1}\|^2_{L^{\infty}_T(H^1(\Omega))}+\|\Lambda_h^{s-1} B^{1}_t\|^2_{L^{2}_T(L^2(\Omega))}\|\Lambda_h^{s-1} K^{(1)}\|^2_{L^{\infty}_T(H^1(\Omega))}.
		\end{split}
	\end{equation}
Similarly, for any $t_1,\,t_2\in[0,T]$, we may prove
	\begin{equation*}
		\begin{split}
			\|\Lambda_h^{s-1}(U(t_1)-U(t_2))\|_{H^1}&\lesssim\|\Lambda_h^{s-1}(B^{1}(t_1):\nabla K^{(1)}(t_1)-B^{1}(t_2):\nabla K^{(1)}(t_2))\|_{L^2}\\
			&\lesssim \|\Lambda_h^{s-1}(B^{1}(t_1)-B^{1}(t_2))\|_{H^1}\|\Lambda_h^{s-1}\nabla K^{(1)}(t_1)\|_{L^2}\\
&\qquad +\|\Lambda_h^{s-1}B^{1}(t_2)\|_{H^1}\|\Lambda_h^{s-1}\nabla( K^{(1)}(t_1)- K^{(1)}(t_2))\|_{L^2},
		\end{split}
	\end{equation*}
	which leads to
$\Lambda_h^{s-1} U\in \mathcal{C}([0,T]; H^1(\Omega)).$

With $U$ above in hand, we set
	\begin{equation*}
		v := u-U, \quad \widetilde{F}_1 := F_1-U_t+\nu\nabla \cdot \mathbb{D}(U),\quad \text{and}\quad \widetilde{F}_3 := F_3+\nu\mathbb{D}(U)n_0,
	\end{equation*}
then the system \eqref{EB-multi-3.1} can be reduced to finding $(\xi^1, v,  P)$ such that
	\begin{equation}
		\begin{cases}
			\partial_t \xi^1-v^1=U^1 &\quad \text{ in } [0,T]\times\Sigma_0,\\
			\partial_t v -\nu \nabla \cdot \mathbb{D} (v) + \nabla P = \widetilde{F}_1 &\quad \text{ in } [0,T]\times\Omega,\\
			\nabla \cdot v=0 &\quad \text{ in } [0,T]\times\Omega,\\
			(p -g\,\xi^1)\,n_0 -\nu \mathbb{D} (v)n_0 =\widetilde{F}_3 &\quad \text{ on } [0,T]\times\Sigma_0,\\
			v^1 =0,\, v^2=-U^2,\, v^3=-U^3  & \quad \text{ on } [0,T]\times\Sigma_b,\\
			v|_{t=0} = u_0-U_0  & \quad \text{ in } \Omega,\\
			\xi^1|_{t=0} = \xi^1_0  & \quad \text{ on } \Sigma_0.
		\end{cases}
	\end{equation}
	
	{\bf Step 2: Homogenizing the upper boundary conditions}\par

	We denote $\Xi\,A(t)\,(\forall\,t \in \mathbb{R})$ the extension of a function $A(t)\, (\forall\,t \in [0,T])$ by
	\begin{equation*}
		\Xi A(t)\eqdefa
		\begin{cases}
			(3A(-t)-2A(-2t))\chi_0(t)&\text{ if }t<0,\\
			A(t)&\text{ if }0\leq t\leq T, \\
			(3A(2T-t)-2A(3T-2t))\chi_T(t)&\text{ if }t>T, \\
		\end{cases}
	\end{equation*}
	where $\chi_0(t)\in C_c^{\infty}((-\infty,0])$ and $\chi_T(t)\in C_c^{\infty}([T,\infty))$ are defined by
	\begin{equation*}
		\chi_0(t)\eqdefa
		\begin{cases}
			1&\text{ if }\,t\in[-\frac{1}{4}T, 0],\\
             \in [0, 1]& 	\text{ if }	\, t\in[-\frac{1}{2}T, -\frac{1}{4}T],	\\
             0&\text{ if }t\in (-\infty, -\frac{1}{2}T],
		\end{cases}
		\text{ and }\,
		\chi_T(t)\eqdefa
		\begin{cases}
			1&\text{ if }t\in[T, \frac{5}{4}T],\\
           \in [0, 1]& \text{ if }	\, t\in[\frac{5}{4}T, \frac{3}{2}T],\\
			0&\text{ if }t\in [\frac{3}{2}T, +\infty)
		\end{cases}
	\end{equation*}
	respectively. We apply the extension operator $\Xi$ to $B^{1},K^{(1)},B^2,K^{(2)},\Psi,U$, and keep the same notation for the extension of $\widetilde{F}_3(t)=B^2:\nabla_h K^{(2)}+\nu\mathbb{D}(U)n_0$ ($\forall \, t \in \mathbb{R}$).
Denote two smooth cut-off function by
	\begin{equation}\label{cutfunction-B}
		\chi_{b}(x_1)\eqdefa
		\begin{cases}
			1, &\text{ if } x_1 \in [-\underline{b}, -\frac{2}{3}\underline{b});\\
			 \in [0, 1],  &\text{ if }x \in (-\frac{2}{3}\underline{b}, -\frac{1}{3}\underline{b}];\\
			0, &\text{ if } x_1 \in (-\frac{1}{3}\underline{b}, 0],
		\end{cases}
       \quad
       \chi_{f}(x)\eqdefa
		\begin{cases}
			 1, &\text{ if }x \in (-\frac{2}{3}\underline{b}, 0];\\
			   \in [0, 1], &\text{ if }x \in (-\frac{5}{6}\underline{b}, -\frac{2}{3}\underline{b}];\\
			 0, &\text{ if } x \in [-\underline{b}, -\frac{5}{6}\underline{b}),
		\end{cases}
	\end{equation}
	with $|\frac{d^n}{d x_1^n}(\chi_{b}(x_1),\chi_{f}(x_1))| \lesssim \underline{b}^{-n}$ for any $n \in \mathbb{N}$.
	Notice that $\chi_{f}(\text{Supp}{\chi'_{b}})\equiv 1$.
	We also set $\Omega_{b}:=\Omega\,\chi_{b}$ with the measure $\chi_{b}\,dx$ and $\Omega_{f}:=\Omega\,\chi_{f}$ with the measure $\chi_{f}\,dx$, and define $\varphi_{1,b} := \varphi_1\chi_{b}$ and $\varphi_{1,f}:=\varphi_1\chi_f$.\par

Our next goal is to construct $(V,P_1)$ satisfying
	\begin{equation*}\label{EB-multi-3.4aa}
		\begin{cases}
			\nabla \cdot V=0 &\quad \text{ in } \mathbb{R}\times\Omega,\\
			P_1n_0 - \nu \mathbb{D} (V)n_0 =\widetilde{F}_3 &\quad \text{ on } \mathbb{R}\times\Sigma_0,\\
			V^1 =0,V^2=-U^2,V^3=-U^3 & \quad \text{ on } \mathbb{R}\times\Sigma_b,
		\end{cases}
	\end{equation*}
which is equivalent to
	\begin{equation}\label{EB-multi-3.4}
		\begin{cases}
			\nabla \cdot V=0 &\quad \text{ in } \mathbb{R}\times\Omega,\\
			-\nu (\partial_1V^{\beta}+\partial_{\beta}V^1) =\widetilde{F}_3^{\beta} \,(\mbox{for}\,\beta=2,3) &\quad \text{ on } \mathbb{R}\times\Sigma_0,\\
P_1 =\widetilde{F}_3^1+ 2\nu \partial_1 V^1 &\quad \text{ on } \mathbb{R}\times\Sigma_0,\\
			V^1 =0,V^2=-U^2,V^3=-U^3 & \quad \text{ on } \mathbb{R}\times\Sigma_b.
		\end{cases}
	\end{equation}
And our construction and estimates of $V(=V_1+V_2)$ is divided into three parts, the first two  parts show the construction $V_1$ and $V_2$ and the estimates of $\|\Lambda_h^{s-1}V_t\|_{L^2_T(L^2(\Omega))}$ and $\|\Lambda_h^{s-1}V\|_{L^2_T(H^2(\Omega))}$, and the third part shows the estimates of $\|\Lambda_h^{s-1}V\|_{L^{\infty}_T(H^1(\Omega))}$.\\
{\bf Step\,2.1:The construction of $V_1$ and the estimates of $\|\Lambda_h^{s-1}(\partial_tV_1,\nabla V_1,\nabla^2 V_1)\|_{L^2_T(L^2(\Omega))}$}

Motivated by \cite{Danchin2020}, we first look for $V_1=(V_1^1,V_1^2,V_1^3)^T$ under the form $V_1 = \nabla \times \varphi$ with $\varphi = (0,\varphi^2,\varphi^3)^T$ satisfying
	\begin{equation}\label{EB-multi-3.7}
\begin{cases}
		\partial_t \varphi-\Delta \varphi=0 &\quad \text{in }\,\mathbb{R}\times \Omega,\\
			 \partial_1\partial_1\varphi^3-\partial_2\partial_2\varphi^2+\partial_2\partial_3\varphi^2= \frac{1}{\nu}\widetilde{F}_3^2& \quad\text{on } \, \mathbb{R}\times\Sigma_0,\\
		\partial_1\partial_1\varphi^2-\partial_3\partial_3\varphi^2+\partial_3\partial_2\varphi^3=-\frac{1}{\nu}\widetilde{F}_3^3& \quad\text{on }  \, \mathbb{R}\times\Sigma_0.
		\end{cases}
	\end{equation}
	Denoting the trace of $\varphi$ on $\mathbb{R}\times  \Sigma_0$ by $\Phi=(0,\Phi^2, \Phi^3)^T$ and combining with \eqref{EB-multi-3.7}, we have
	\begin{equation}\label{EB-multi-3.9}
		\begin{cases}
			\partial_t\Phi^2-\Delta_h\Phi^2-\partial_3\partial_3\Phi^2+\partial_3\partial_2\Phi^3=\frac{1}{\nu} \widetilde{F}_3^3  &  \text{on }\,\mathbb{R}\times\Sigma_0,\\
			\partial_t\Phi^3-\Delta_h\Phi^3-\partial_2\partial_2\Phi^3+\partial_2\partial_3\Phi^2= -\frac{1}{\nu}\widetilde{F}_3^2  &  \text{on }\, \mathbb{R}\times\Sigma_0.
		\end{cases}
	\end{equation}
	Set $\Phi_h:=(\Phi^2, \Phi^3)^T$. According to the Helmholtz-Hodge decomposition in $H^{s}(\Sigma_0)$, we have
	\begin{equation*}\label{Helmholtz-Hodge}
		\Phi_h=\mathbb{P}_0\Phi_h+\nabla_h \pi \quad \text{on }\,\mathbb{R}\times\Sigma_0,
	\end{equation*}
	and then
\begin{equation}\label{Helmholtz-Hodge-2aa}
		\nabla_h\cdot\Phi_h=\Delta_h \pi \quad \text{on }\,\mathbb{R}\times\Sigma_0,
	\end{equation}
where $\mathbb{P}_0$ is Leray's projection operator in $L^2(\Sigma_0)$.

Applying $-\nabla_h(-\Delta_h)^{-1}\nabla_h\cdot$ to both sides of the linear system \eqref{EB-multi-3.9}, we find
	\begin{equation}\label{nablaPhi}
		\partial_t(\nabla_h\pi)-\Delta_h(\nabla_h\pi)=-\frac{1}{\nu}\nabla_h(-\Delta_h)^{-1} \nabla_h\cdot\widetilde{F}_3^{\perp} \,\text{on }\,\mathbb{R}\times\Sigma_0,
	\end{equation}
where  $\widetilde{F}_3^{\perp}:=(\widetilde{F}_3^3,-\widetilde{F}_3^2)^T$.

Thanks to \eqref{Helmholtz-Hodge-2aa}, one finds
	\begin{equation*}\label{varpi_3}
		\begin{cases}
			\partial_2(\partial_2\Phi^2+\partial_3\Phi^3)=\partial_2\Delta_h\pi\,\,\text{on}\,\,\mathbb{R}\times\Sigma_0,\\
			\partial_3(\partial_2\Phi^2+\partial_3\Phi^3)=\partial_3\Delta_h\pi\,\,\text{on}\,\,\mathbb{R}\times\Sigma_0,\\
		\end{cases}
	\end{equation*}
	which along with \eqref{EB-multi-3.9} implies
	\begin{equation}\label{varpi_3}
			\partial_t\Phi_h-2\Delta_h\Phi_h = \frac{1}{\nu} \widetilde{F}_3^{\perp} -\Delta_h\nabla_h\pi \quad \text{on }\quad\mathbb{R}\times\Sigma_0.
	\end{equation}
	The classical energy estimate of \eqref{varpi_3} shows that
	\begin{equation}\label{EB-multi-3.10-1}
		\|(\nabla_h\Phi,\nabla_h^2\Phi)\|_{L^2(\mathbb{R}; H^{s-\frac{1}{2}}(\Sigma_0))}^2+\|\Phi_t\|_{L^2(\mathbb{R}; H^{s-\frac{1}{2}}(\Sigma_0))}^2 \lesssim \|(\widetilde{F}_3^{\perp}, \Delta_h\nabla_h\pi)\|_{L^2(\mathbb{R}; H^{s-\frac{1}{2}}(\Sigma_0))}^2.
	\end{equation}
While the energy estimate of \eqref{nablaPhi} gives
	\begin{equation*}\label{EB-multi-3.10-2}
		\|\Delta_h\nabla_h\pi\|_{L^2(\mathbb{R}; H^{s-\frac{1}{2}}(\Sigma_0))}^2\lesssim \|\widetilde{F}_3^{\perp}\|_{L^2(\mathbb{R}; H^{s-\frac{1}{2}}(\Sigma_0))}^2,
	\end{equation*}
which along with \eqref{EB-multi-3.10-1} gives rise to
	\begin{equation}\label{EB-multi-3.10-3}
		\|(\nabla_h\Phi,\nabla_h^2\Phi)\|_{L^2(\mathbb{R}; H^{s -\frac{1}{2} }(\Sigma_0))}^2+\|\Phi_t\|_{L^2(\mathbb{R}; H^{s-\frac{1}{2}}(\Sigma_0))}^2 \lesssim \|\widetilde{F}_3\|_{L^2(\mathbb{R}; H^{s-\frac{1}{2}}(\Sigma_0))}^2.
	\end{equation}
Substituting $V_1 = \nabla \times \varphi$ with $\varphi = (0,\varphi^2,\varphi^3)^T$ into the forth equation in \eqref{EB-multi-3.4} gives rise to
	\begin{equation*}
		\partial_1\varphi^2 =-U^3, \quad \partial_1\varphi^3 =U^2 \quad \text{on}\quad\mathbb{R}\times \Sigma_b.
	\end{equation*}
	Hence, the function $\varphi$ obeys the system
	\begin{equation}\label{EB-multi-3.10}
		\begin{cases}
			\partial_t \varphi-\Delta \varphi=0 \,&  \text{on}\quad\mathbb{R}\times \Omega,\\
			\varphi = \Phi\,& \text{on}\quad\mathbb{R}\times \Sigma_0,\\
			\partial_1\varphi =-n\times U,\,& \text{on}\quad\mathbb{R}\times \Sigma_b.
		\end{cases}
	\end{equation}

Denote the Fourier transform in terms of $(t, x_h)$ by\, $\widehat{}$\, or $\mathcal{F}=\mathcal{F}_{(t, x_h)\rightarrow (\tau, \varsigma_h)}$, and then apply the Fourier transform $\mathcal{F}$ to the problem \eqref{EB-multi-3.10}, so we get the second-order ODE for $\widehat{\varphi}$:
	\begin{equation}\label{ODE-varphi-1}
	\begin{cases}
		-\frac{d^2}{d{x}_1^2}\widehat{\varphi}+(i\,\tau+|\varsigma_h|^2)\widehat{\varphi}=0\quad \text{ in } (-b,0),\\
		\widehat{\varphi}|_{x_1=0} = \widehat{\Phi},\\
		\frac{d}{d{x}_1}\widehat{\varphi}|_{x_1=-b} =-n\times \widehat{U}|_{x_1=-b}.
	\end{cases}
	\end{equation}
	Solving the ODEs \eqref{ODE-varphi-1}, we get the explicit expression of the solution $\widehat{\varphi}(\tau, x_1, \varsigma_h)$ (with $\tau\in \mathbb{R}$, $\varsigma_h \in \mathbb{R}^2$) that
	\begin{equation}\label{EB-multi-3.11}
			\widehat{\varphi}(\tau,x_1,\varsigma_h)=\widehat{\varphi}_1 +\widehat{\varphi}_2
	\end{equation}
with
\begin{equation}\label{form-of-varphi-1}
      \begin{split}
        &\widehat{\varphi}_1 :=\frac{\mathrm{e}^{-rx_1}+\mathrm{e}^{r(2b+x_1)}}{\mathrm{e}^{2rb}+1}(0,\widehat{\Phi}^2,\widehat{\Phi}^3)^T,\\
       & \widehat{\varphi}_2 :=\frac{\mathrm{e}^{rx_1}-\mathrm{e}^{-rx_1}}{r(\mathrm{e}^{rb}+\mathrm{e}^{-rb})}
        (0,-\mathrm{i}\varsigma_3\widehat{\Psi}|_{x_1=-b},\mathrm{i}\varsigma_2\widehat{\Psi}|_{x_1=-b})^T,
      \end{split}
    \end{equation}
where $r$ is defined by the equation
$r^2= i \tau+|\varsigma_h|^2 $ so that $\arg r\in [-\frac{\pi}{4},\frac{\pi}{4}]$, which means
\begin{equation}\label{r-ex}
r=(\frac{\sqrt{\tau^2+|\varsigma_h|^4}
+|\varsigma_h|^2}{2})^{\frac{1}{2}}+i\,\mathrm{sgn}(\tau)(\frac{\sqrt{\tau^2+|\varsigma_h|^4}-|\varsigma_h|^2}{2})^{\frac{1}{2}}.
\end{equation}

Let's now get estimates of $\varphi$ by using the form \eqref{EB-multi-3.11} of $\varphi$. We first deal with the estimate of $\|\Lambda_h^{s-1}( \nabla^2\varphi, \nabla^2\nabla_h\varphi, \partial_t\varphi, \partial_t\nabla_h\varphi,\partial_t\nabla\varphi_2,\nabla^3\varphi_2)\|_{L^2(\mathbb{R}\times\Omega)}^2$.

Thanks to Plancherel's identity, one has
	\begin{equation}\label{Re-1+2}
		\begin{split}
			&\|\Lambda_h^{s-1}(\partial_t\nabla_h\varphi,\nabla^2\nabla_h\varphi)\|_{L^2(\mathbb{R}\times\Omega)}\\
&\lesssim\big\|(\chi_{\{\mathrm{Re}\,r\leq 1\}}+1-\chi_{\{\mathrm{Re}\,r\leq 1\}})|r|^2|\varsigma_h|\langle\varsigma_h\rangle^{s-1} (\widehat{\varphi}_1+ \widehat{\varphi}_2)\big\|_{L^2(\mathbb{R}\times\Omega)}.	
            \end{split}
	\end{equation}
On the one hand, due to $\| \chi_{\{\mathrm{Re}\,r\leq 1\}} (|\varsigma_h|^{\frac{1}{2}}|\mathrm{e}^{-rx_1}+\mathrm{e}^{r(2b+x_1)} | ) |\mathrm{e}^{2rb}+1|^{-1}\|_{L^{\infty}(\mathbb{R}\times\Omega)}\lesssim 1$, we get
    \begin{equation}\label{Re-1}
		\begin{split}
     \big\| \chi_{\{\mathrm{Re}\,r\leq 1\}} |r|^2|\varsigma_h|\langle\varsigma_h\rangle^{s-1} \widehat{\varphi}_1\big\|_{L^2(\mathbb{R}\times\Omega)}
\lesssim b^{\frac{1}{2}} \||r|^2\langle\varsigma_h\rangle^{s-\frac{1}{2}}\hat{\Phi}\|_{L_{\tau}^2(L^2_h)}.
       \end{split}
	\end{equation}
    Similarly, one obtains
    \begin{equation}\label{Re-1111111111}
		\begin{split}
     \big\| ( 1-\chi_{\{\mathrm{Re}\,r\leq 1\}}) |r|^2|\varsigma_h|\langle\varsigma_h\rangle^{s-1} \widehat{\varphi}_2\big\|_{L^2(\mathbb{R}\times\Omega)}
\lesssim b^{\frac{1}{2}} \||r|^2\langle\varsigma_h\rangle^{s-\frac{1}{2}}\hat{\Psi}|_{x_1=-b}\|_{L_{\tau}^2(L^2_h)},
       \end{split}
	\end{equation}
and
    \begin{equation}\label{Re-1111111112}
		\begin{split}
     \big\| ( 1-\chi_{\{\mathrm{Re}\,r\leq 1\}}) |r|^3\langle\varsigma_h\rangle^{s-1} \widehat{\varphi}_2\big\|_{L^2(\mathbb{R}\times\Omega)}
\lesssim b^{\frac{1}{2}} \||r|^2\langle\varsigma_h\rangle^{s+\frac{1}{2}}\hat{\Psi}|_{x_1=-b}\|_{L_{\tau}^2(L^2_h)}.
       \end{split}
	\end{equation}
	On the other hand, according to \eqref{form-of-varphi-1}, we have
    \begin{equation}\label{Re-22}
		\begin{split}
			&\big\|( 1-\chi_{\{\mathrm{Re}\,r\leq 1\}}) |r|^2|\varsigma_h|\langle\varsigma_h\rangle^{s-1} \widehat{\varphi}_1\big\|_{L^2(\mathbb{R}\times\Omega)}\\	
&\lesssim \|( 1-\chi_{\{\mathrm{Re}\,r\leq 1\}})
|\varsigma_h|^{\frac{1}{2}}\,|\mathrm{e}^{2rb}+1|^{-1}\|\mathrm{e}^{-rx_1}+\mathrm{e}^{r(2b+x_1)}\|_{L_{x_1}^2}\|_{L^{\infty}_{\tau, h}}\||r|^2\langle\varsigma_h\rangle^{s-\frac{1}{2}}\hat{\Phi}\|_{L_{\tau, h}^2}\\
&\lesssim  \||r|^2\langle\varsigma_h\rangle^{s-\frac{1}{2}}\hat{\Phi}\|_{L_{\tau}^2(L^2_h)}.
       \end{split}
	\end{equation}
Along the same lines, we may prove that
    \begin{equation}\label{est-Re-22-1}
		\begin{split}
    & \big\| \chi_{\{\mathrm{Re}\,r\leq 1\}} |r|^2|\varsigma_h|\langle\varsigma_h\rangle^{s-1} \widehat{\varphi}_2\big\|_{L^2(\mathbb{R}\times\Omega)}
\lesssim  \||r|^2\langle\varsigma_h\rangle^{s-\frac{1}{2}} \hat{\Psi}|_{x_1=-b}\|_{L_{\tau}^2(L^2_h)},\\
     &\big\| \chi_{\{\mathrm{Re}\,r\leq 1\}} |r|^3\langle\varsigma_h\rangle^{s-1} \widehat{\varphi}_2\big\|_{L^2(\mathbb{R}\times\Omega)}
\lesssim  \||r|^2\langle\varsigma_h\rangle^{s-\frac{1}{2}} \hat{\Psi}|_{x_1=-b}\|_{L_{\tau}^2(L^2_h)}.
       \end{split}
	\end{equation}
 Plugging \eqref{Re-1},\eqref{Re-1111111111},\eqref{Re-22},\eqref{est-Re-22-1} into \eqref{Re-1+2}  implies
	\begin{equation*}
		\begin{split}
			\|\Lambda_h^{s-1}(\partial_t\nabla_h\varphi,\nabla^2\nabla_h\varphi)\|^2_{L^2(\mathbb{R}\times\Omega)}
			&\lesssim \|\nabla_h^2(\Phi,\Psi)\|_{L^2(\mathbb{R};H^{s-\frac{1}{2}}(\Sigma_0))}^2+\|(\Phi_t,\Psi_t)\|_{L^2(\mathbb{R};H^{s-\frac{1}{2}}(\Sigma_0))}^2.
		\end{split}
	\end{equation*}
    Combining \eqref{Re-1111111112} and \eqref{est-Re-22-1}, we obtain
    \begin{equation*}
		\begin{split}
			\|\Lambda_h^{s-1}(\partial_t\nabla\varphi_2,\nabla^3\varphi_2)\|^2_{L^2(\mathbb{R}\times\Omega)}
			&\lesssim \|\nabla_h^2(\Phi,\Psi)\|_{L^2(\mathbb{R};H^{s-\frac{1}{2}}(\Sigma_0))}^2+\|(\Phi_t,\Psi_t)\|_{L^2(\mathbb{R};H^{s-\frac{1}{2}}(\Sigma_0))}^2.
		\end{split}
	\end{equation*}
    Similarly, we have
    \begin{equation*}
		\begin{split}
			\|\Lambda_h^{s-1}(\partial_t \varphi,\nabla^2 \varphi)\|^2_{L^2(\mathbb{R}\times\Omega)}
\lesssim\|\nabla_h^2(\Phi,\Psi)\|_{L^2(\mathbb{R};H^{s-\frac{1}{2}}(\Sigma_0))}^2+\|(\Phi_t,\Psi_t)\|_{L^2(\mathbb{R};H^{s-\frac{1}{2}}(\Sigma_0))}^2.
		\end{split}
	\end{equation*}
Hence, thanks to \eqref{EB-multi-3.10-3}, \eqref{EB-multi-3.11}, and Plancherel's identity, we get
\begin{equation}\label{Low1andh}
     \begin{split}
		&\|\Lambda_h^{s-1}( \nabla^2\varphi, \nabla^2\nabla_h\varphi, \partial_t\varphi, \partial_t\nabla_h\varphi,\partial_t\nabla\varphi_2,\nabla^3\varphi_2)\|_{L^2(\mathbb{R}\times\Omega)}^2\\
&
 \lesssim\|\Lambda^{s-1}_h(\widetilde{F}_3,F_2)\|_{L^2_T(H^1(\Omega))}^2
 +\|\Lambda_h^{s-1}B^{(1)}_t\|^2_{L^2_T(L^2(\Omega))}\|\Lambda_h^{s-1}K^{(1)}\|^2_{L^{\infty}_T(H^1(\Omega))}\\
 &\qquad\qquad\qquad\qquad\qquad\qquad
+	\|\Lambda_h^{s-1}B^{(1)}\|^2_{L^{\infty}_T(H^1(\Omega))}
\|\Lambda_h^{s-1} K^{(1)}_t\|^2_{L^2_T(L^2(\Omega))}.
     \end{split}
	\end{equation}
In order to get the full regularity of $\varphi$, we need to get the estimates of $\Lambda_h^{s-1}\partial_t\partial_1\varphi_1$ and $\Lambda_h^{s-1}\partial_1\partial_1\nabla\varphi_1$. Due to $\partial_1\partial_1\nabla\varphi_1=\partial_t\nabla\varphi_1-\Delta_h\nabla\varphi_1$, we just need to recover the estimate of $\Lambda_h^{s-1}\partial_t\partial_1\varphi_1$.

Due to \eqref{EB-multi-3.11}, we can get the estimate of $\partial_t\partial_1\varphi_1$ near the bottom as follows
	\begin{equation*}
    \begin{split}
    \|\Lambda_h^{s-1}\partial_t\partial_1\varphi_{1,b}\|_{L^2(\mathbb{R}\times\Omega)}^2
    \lesssim &\|\chi_{b}\Lambda_h^{s-1}\partial_t\partial_1 \varphi_1 \|_{L^2(\mathbb{R}\times\Omega)}^2+\|\chi'_{b}\Lambda_h^{s-1}\partial_t \varphi_1 \|_{L^2(\mathbb{R}\times\Omega)}^2\\
    \lesssim &\|\widetilde{F}_3^{\perp}\|_{L^2_T(H^{s-\frac{1}{2}}(\Sigma_0))}^2,
    \end{split}
	\end{equation*}
	which implies
	\begin{equation}\label{EB-multi-3.14-11}
		 \begin{split} \|\Lambda_h^{s-1}\partial_t\partial_1\varphi_1\|_{L^2(\mathbb{R}\times\Omega_b)}^2\lesssim \|\widetilde{F}_3^{\perp}\|_{L^2_T(H^{s-\frac{1}{2}}(\Sigma_0))}^2.
         \end{split}
	\end{equation}
We are now in a position to estimate $\partial_t\partial_1\varphi_1$ near the upper boundary. Let's first consider the function $\psi=(0,\psi^2,\psi^3)^T$ with
\begin{equation}\label{EB-multi-3.14-1}
\begin{cases}
\psi^2 :=\partial_t\varphi_1^2-\Delta_h\varphi_1^2-\partial_3\partial_3\varphi_1^2+\partial_3\partial_2\varphi_1^3-\frac{1}{\nu}\widetilde{F}_3^3, \\
\psi^3 :=\partial_t\varphi^3-\Delta_h\varphi_1^3-\partial_2\partial_2\varphi_1^3+\partial_2\partial_3\varphi_1^2+\frac{1}{\nu}\widetilde{F}_3^2. \end{cases}
\end{equation}
	which implies that $\psi_f:=\psi\chi_f$ satisfies
\begin{equation}\label{ext-heat-eqns-1}
	\begin{cases}
		(\partial_t-\Delta)\psi_f=\mathfrak{F}_1+\nabla\cdot\mathfrak{F}_2\quad (\forall\,\,(t, x)\in \mathbb{R}\times \Omega),\\
		\psi_{f}|_{x_1=0}=0,\\
		\psi_{f}|_{x_1=-b}=0,
	\end{cases}
	\end{equation}
where
$$\mathfrak{F}_1=(0,   \Re^2 -\frac{1}{\nu}(B^{2, 3}_t:\nabla_h K^{(2)}- \partial_{\alpha} B^{2, 3}_{\alpha,i}(K^{(2)})^{i}_t)\chi_f,\Re^3+\frac{1}{\nu}(B^{2, 2}_t:\nabla_h K^{(2)}-\partial_{\alpha} B^{2, 2}_{\alpha,i}(K^{(2)})^{i}_t)\chi_f)^T$$ and $\mathfrak{F}_2$ is a $3\times 3$ matrix defined by
    \[
	 \mathfrak{F}_2=\frac{1}{\nu}\left(
     \begin{array}{cccc}
           0 &  0 & 0 \\
            - \partial_1(\widetilde{F}_3^3\chi_f)  &  -(\partial_2\widetilde{F}_3^3+ (B^{2, 3}_{2,i}(K^{(2)})_t^i)\chi_f&
            -  (2\nu U^1_t+\partial_3\widetilde{F}_3^3 + B^{2, 3}_{3,i}(K^{(2)})_t^i)\chi_f\\
               \partial_1(\widetilde{F}_3^2\chi_f)  &
             (2\nu U^1_t+\partial_2\widetilde{F}_3^2+ B^{2, 2}_{2,i}(K^{(2)})_t^i)\chi_f&
 (\partial_3\widetilde{F}_3^2 + (B^{2, 2}_{3,i}(K^{(2)})_t^i)\chi_f
          \end{array}
     \right),
	\]
	and $\Re=(0, \Re^2, \Re^3)^T$ with
\begin{equation*}
	\begin{split}			
\Re^2:=-(2\partial_1\varphi_{1,t}^2\chi'_f+\varphi_{1,t}^2\chi''_f)
-(\partial_2^2+2\partial_3^2)(2\partial_1\varphi_1^2\chi_f'+\varphi_1^2\chi_f'')
+\partial_3\partial_2(2\partial_1\varphi_1^3\chi_f'+\varphi_1^3\chi_f''),\\	
\Re^3:=-(2\partial_1\varphi_{1,t}^3\chi'_f+\varphi_{1,t}^3\chi''_f)-( \partial_3^2+2\partial_2^2)(2\partial_1\varphi_1^3\chi_f'+\varphi_1^3\chi_f'')
+\partial_2\partial_3(2\partial_1\varphi_1^2\chi_f'+\varphi_1^2\chi_f'').
		\end{split}
	\end{equation*}
By  acting  the Fourier transform $\mathcal{F}_{t\to \tau}$ with respect to time variable $t$ to the system \eqref{ext-heat-eqns-1} , we obtain
		\begin{equation}\label{EB-multi-77.4}
		-\mathrm{i}\tau\widehat{\psi_f}-\Delta\widehat{\psi_f}=\widehat{\mathfrak{F}_1}+\nabla\cdot\widehat{\mathfrak{F}_2}.
		\end{equation}
		Then, we multiply the $i$th component of the equations of \eqref{EB-multi-77.4} by the conjugate of $\widehat{\psi_f}^i$, sum over $i$, and integrate over $\Omega$ to find
\begin{equation*}
 \begin{split}		
 -\int_{\Omega}\mathrm{i}\tau|\widehat{\psi_f}|^2\,dx+\int_{\Omega}|\nabla\widehat{\psi_f}|^2\,dx
 =(\widehat{\mathfrak{F}_1},\widehat{\psi_f})_{L^2}+(\widehat{\mathfrak{F}_2},\nabla\widehat{\psi_f})_{L^2}
\end{split}
\end{equation*}
		Keeping the real part and integrating both sides with respect to $\tau$ over $\mathbb{R}$, we have
	\begin{equation*}
		\begin{split}
		\|\nabla\psi_f\|_{L^2(\mathbb{R}; L^2(\Omega))}^2
			&=\int_{\mathbb{R}}\mathrm{Re}(\widehat{\mathfrak{F}_1},\widehat{\psi_f})_{L^2}\,d\tau
			+\int_{\mathbb{R}}\mathrm{Re}(\widehat{\mathfrak{F}_2},\nabla\widehat{\psi_f})_{L^2}\,d\tau\\
			&\lesssim
			\|\widehat{\mathfrak{F}_1}\|_{L^2(\mathbb{R}; L^2(\Omega))}\|\widehat{\psi_f}\|_{L^2(\mathbb{R}; L^2(\Omega))}
			+\|\widehat{\mathfrak{F}_2}\|_{L^2(\mathbb{R}; L^2(\Omega))}\|\nabla\widehat{\psi_f}\|_{L^2(\mathbb{R}; L^2(\Omega))}\\
			&\lesssim
			\|\mathfrak{F}_1\|_{L^2(\mathbb{R}; L^2(\Omega))}\|\psi_f\|_{L^2(\mathbb{R}; L^2(\Omega))}
			+\|\mathfrak{F}_2\|_{L^2(\mathbb{R}; L^2(\Omega))}\|\nabla\psi_f\|_{L^2(\mathbb{R}; L^2(\Omega))},
		\end{split}
		\end{equation*}
		which along with Poincar\'e inequality leads to
		\begin{equation*}
		\|\nabla\psi_f\|^2_{L^2(\mathbb{R}; L^2(\Omega))}
		\lesssim
		\|\mathfrak{F}_1\|^2_{L^2(\mathbb{R}, L^2(\Omega))}+\|\mathfrak{F}_2\|^2_{L^2(\mathbb{R}; L^2(\Omega))}.
		\end{equation*}
Along the same lines, we obtain
	\begin{equation*}\label{ext-heat-est-2}
    \begin{split}
	&	\|\Lambda_h^{s-1}\nabla\psi_f\|_{L^2(\mathbb{R}, L^2(\Omega))}^2\lesssim \|\Lambda_h^{s-1}\mathfrak{F}_1\|_{L^2(\mathbb{R}; L^2(\Omega))}^2+\|\Lambda_h^{s-1}\mathfrak{F}_2\|_{L^2(\mathbb{R}; L^2(\Omega))}^2,
      \end{split}
	\end{equation*}
   where
   \begin{equation*}
    \begin{split}
    &\|\Lambda_h^{s-1}\mathfrak{F}_1\|_{L^2(\mathbb{R}; L^2(\Omega))}^2+\|\Lambda_h^{s-1}\mathfrak{F}_2\|_{L^2(\mathbb{R}; L^2(\Omega))}^2 \\
    &\lesssim\|\Lambda_h^{s-1} \widetilde{F}_3 \|_{L^2(\mathbb{R}; H^1(\Omega))}^2+\|\Lambda_h^{s-1} U_t \|_{L^2(\mathbb{R}; L^2(\Omega))}^2+\|\Lambda_h^{s-1} B^{2}_t\|_{L^2(\mathbb{R}; L^2(\Omega))}^2\|\Lambda_h^{s-1}\nabla_h K^{(2)}\|_{L^\infty(\mathbb{R}; L^2(\Omega))}^2\\
      &+\| \Lambda_h^{s-1} B^{2}\|_{L^\infty(\mathbb{R}; H^1(\Omega))}^2\|\Lambda_h^{s-1} K^{(2)}_t\|_{L^2(\mathbb{R}; L^2(\Omega))}^2.
      \end{split}
	\end{equation*}
	Combining \eqref{Low1andh} with \eqref{EB-multi-3.14-1} yields
	\begin{equation*}
		 \begin{split}
&\|\Lambda_h^{s-1}\partial_t\partial_1\varphi_1\|_{L^2(\mathbb{R}; L^2(\Omega_f))}^2
 \lesssim\|\Lambda_h^{s-1} (\widetilde{F}_3,F_2) \|_{L^2(\mathbb{R}; H^1(\Omega))}^2\\
 &\quad+\|\Lambda_h^{s-1} B^{2}_t\|_{L^2_T(L^2(\Omega))}^2\|\Lambda_h^{s-1}\nabla_h K^{(2)}\|_{L^\infty_T(L^2(\Omega))}^2
       + \| \Lambda_h^{s-1} B^{2}\|_{L^\infty_T( H^1(\Omega))}^2\|\Lambda_h^{s-1} K^{(2)}_t \|_{L^2_T(L^2(\Omega))}^2\\
      &\quad+\|\Lambda_h^{s-1} B^{1}_t\|^2_{L^2_T(L^2(\Omega))}\|\Lambda_h^{s-1} K^{(1)}\|^2_{L^{\infty}_T(H^1(\Omega))}
+	\|\Lambda_h^{s-1}B^{1}\|^2_{L^{\infty}_T(H^1(\Omega))}
\|\Lambda_h^{s-1} K^{(1)}_t\|^2_{L^2_T(L^2(\Omega))}.
      \end{split}
	\end{equation*}
	which along with \eqref{EB-multi-3.14-11} leads to
	\begin{equation}\label{phi-es-part2}
     \begin{split}
		&\|\Lambda_h^{s-1}\partial_t\partial_1\varphi_1 \|_{L^2(\mathbb{R}; L^2(\Omega))}^2
  +\|\Lambda_h^{s-1}\partial_1\partial_1\nabla\varphi_1\|_{L^2(\mathbb{R}; L^2(\Omega))}^2\\
&\lesssim \|\Lambda_h^{s-1} (\widetilde{F}_3,F_2) \|_{L^2(\mathbb{R}; H^1(\Omega))}^2+\|\Lambda_h^{s-1} (B^{1}_t, B^{2}_t)\|^2_{L^2_T(L^2(\Omega))}\|\Lambda_h^{s-1} (K^{(1)}, K^{(2)})\|^2_{L^{\infty}_T(H^1(\Omega))}\\
      &\quad\quad+\|\Lambda_h^{s-1}  (B^{1}, B^{2})\|^2_{L^{\infty}_T(H^1(\Omega))}
\|\Lambda_h^{s-1} (K^{(1)}_t, K^{(2)}_t)\|^2_{L^2_T(L^2(\Omega))}.
      \end{split}
	\end{equation}
So far, we get the sufficient regularity of $\varphi$,
and then, by the construction of $V_1$, we have
	\begin{equation}\label{EB-multi-3.12}
		\begin{split}
			\mathcal{F}(V_1)(\tau,x_1,\varsigma_h)=\mathcal{F}(\nabla\times\varphi)
=(\mathrm{i}\varsigma_2\hat{\varphi}^3-\mathrm{i}\varsigma_3\hat{\varphi}^2,-\partial_1\hat{\varphi}^3,\partial_1\hat{\varphi}^2)^T.
		\end{split}
	\end{equation}
 which along with \eqref{Low1andh} and \eqref{phi-es-part2} implies
	\begin{equation*}
		\begin{split}
	&\|\Lambda_h^{s-1}\partial_tV_1\|_{L^2(\mathbb{R}; L^2(\Omega))}^2+\|\Lambda_h^{s-1}\nabla V_1\|_{L^2(\mathbb{R}; H^1(\Omega))}^2\\
	&\lesssim \|\Lambda_h^{s-1} (\widetilde{F}_3,F_2) \|_{L^2(\mathbb{R}; H^1(\Omega))}^2+\|\Lambda_h^{s-1} (B^{1}_t, B^{2}_t)\|^2_{L^2_T(L^2(\Omega))}\|\Lambda_h^{s-1} (K^{(1)}, K^{(2)})\|^2_{L^{\infty}_T(H^1(\Omega))}\\
      &\quad\quad+\|\Lambda_h^{s-1}  (B^{1}, B^{2})\|^2_{L^{\infty}_T(H^1(\Omega))}
\|\Lambda_h^{s-1} (K^{(1)}_t, K^{(2)}_t)\|^2_{L^2_T(L^2(\Omega))}.
		\end{split}
	\end{equation*}
{\bf Step\,2.2:The construction of $V_2$ and the estimate of $\|\Lambda_h^{s-1}(\partial_tV_2,\nabla V_2,\nabla^2 V_2)\|_{L^2_T(L^2(\Omega))}$  }

	However, we find that $V_1$ might not satisfy
	$V_1^1=0$ on $\mathbb{R} \times\Sigma_b$,
	so we need construct a suitable divergence-free vector field $V_2=(V_2^1,V_2^2,V_2^3)^T$ satisfying
	\begin{equation}\label{EB-multi-3.13}
		\begin{cases}
\nabla\cdot\,V_2=0&  \text{in}\, \mathbb{R}\times\Omega,\\
\partial_1V_2^{\beta}+\partial_{\beta}V_2^1 =0&  \text{on}\, \mathbb{R}\times\Sigma_0 \quad (\beta=2,3),\\
	V_2^{\beta}=0 &  \text{on}\,\mathbb{R}\times\Sigma_b,\\
			V_2^1 = -V_1^1  &  \text{on}\, \mathbb{R}\times\Sigma_b.
		\end{cases}
	\end{equation}
Take the Fourier transformation \,$\widehat{\cdot}$ \,in terms of $(t, x_h)$ to \eqref{EB-multi-3.13}, and define
	\begin{equation*}\label{EB-multi-3.14aa}
	\begin{split}
\widehat{V}^1_2 \eqdefa  -\chi_{b}\widehat{V}_1^1,
    \end{split}
	\end{equation*}
   we have
   \begin{equation}\label{V-23-b}
		\begin{cases}
 i\varsigma_2\widehat{V}_2^2+i\varsigma_3\widehat{V}_2^3=-\partial_1\widehat{V}^1_2&  \text{in}\, \mathbb{R}\times\Omega,\\
\partial_1\widehat{V}_2^{\beta}=-i\varsigma_{\beta}\widehat{V}_2^1=0 &  \text{on}\, \mathbb{R}\times\Sigma_0 \quad (\beta=2,3),\\
	V_2^{\beta}=0 &  \text{on}\,\mathbb{R}\times\Sigma_b .
		\end{cases}
	\end{equation}
    Let $V_2^{\beta}\eqdefa \partial_{\beta}\phi-(-\Delta_h)^{-1}\partial_{\beta}\partial_1^2\phi$, that is,
    \begin{equation}\label{VV22333444}
	\begin{split}
\widehat{V}^{\beta}_2 =  i\varsigma_{\beta}\widehat{\phi}-\frac{i\varsigma_{\beta}}{|\varsigma_h|^2}\partial_1^2\widehat{\phi},\quad (\beta=2,3),
    \end{split}
	\end{equation}
    where $\phi$ satisfy
    \begin{equation}\label{phi-V-2}
		\begin{cases}
 \frac{\mathrm{d}^2}{\mathrm{d}x_1^2}\widehat{\phi}-|\varsigma_h|^2\widehat{\phi}=-\partial_1\widehat{V}^1_2&  \text{in}\, \mathbb{R}\times\Omega,\\
\partial_1\widehat{\phi}=0 &  \text{on}\, \mathbb{R}\times\Sigma_0,\\
	\widehat{\phi}=0 &  \text{on}\,\mathbb{R}\times\Sigma_b .
		\end{cases}
	\end{equation}
     By solving \eqref{phi-V-2}, we can get
     \begin{equation*}
	\begin{split}
     \widehat{\phi}=&\frac{\mathrm{e}^{|\varsigma_h|(x_1+b)}-\mathrm{e}^{|\varsigma_h|(-x_1+b) } }{2|\varsigma_h|(\mathrm{e}^{|\varsigma_h| b} +\mathrm{e}^{-|\varsigma_h| b}  )}
     \int_{-b}^{0}\partial_1\widehat{V}_2^1(\mathrm{e}^{|\varsigma_h| \bar{x}_1} +\mathrm{e}^{-|\varsigma_h| \bar{x}_1})\mathrm{d}\bar{x}_1\\
     &+\frac{1}{2|\varsigma_h|}\int_{-b}^{x_1}\partial_1\widehat{V}_2^1
     (\mathrm{e}^{|\varsigma_h| (\bar{x}_1-x_1)} +\mathrm{e}^{-|\varsigma_h| (\bar{x}_1-x_1)})\mathrm{d}\bar{x}_1
    \end{split}
	\end{equation*}
     then we can verify that $\widehat{V}^{\beta}$ $(\beta=2,3)$ defined by \eqref{VV22333444} satisfies \eqref{V-23-b}, and combining with
     \begin{equation*}
	\begin{split}
     \widehat{V}_2^1=-\chi_b\widehat{V}_1^1= -\chi_b\bigg(\frac{\mathrm{e}^{-rx_1}+\mathrm{e}^{r(2b+x_1)}}{\mathrm{e}^{2rb}+1} (i\varsigma_2\widehat{\Phi}^3-i\varsigma_3\widehat{\Phi}^2)
      -\frac{|\varsigma_h|^2(\mathrm{e}^{rx_1}-\mathrm{e}^{-rx_1})}{r(\mathrm{e}^{rb}+\mathrm{e}^{-rb})}\widehat{\Psi}|_{x_1=-b}\bigg),
     \end{split}
	\end{equation*}
    we can get
    \begin{equation*}
		\begin{split}
			&\|\Lambda_h^{s-1}\partial_tV_2\|_{L^2_T(L^2(\Omega))}^2+\|\Lambda_h^{s-1}\nabla V_2\|_{L^2_T(H^1(\Omega))}^2 \\
			&\lesssim
		\|\nabla_h(\Phi,\Psi)\|_{L^2(\mathbb{R};H^{s+\frac{1}{2}}(\Sigma_0))}^2+\|(\Phi_t,\Psi_t)\|_{L^2(\mathbb{R};H^{s-\frac{1}{2}}(\Sigma_0))}^2.
		\end{split}
	\end{equation*}
    Moreover, we can also obtain
	\begin{equation*}
		 \begin{split}
		&\|\Lambda_{h}^{s-1} V_2 (0)\|^2_{H^{1}(\Omega)}+\|\Lambda_{h}^{s-1} \nabla V_2\|^2_{ L^{2}(\Omega;L^{\infty}_t)}\\
		& \lesssim \|\nabla_h(\Phi,\Psi)\|_{L^2(\mathbb{R};H^{s+\frac{1}{2}}(\Sigma_0))}^2+\|(\Phi_t,\Psi_t)\|_{L^2(\mathbb{R};H^{s-\frac{1}{2}}(\Sigma_0))}^2 ,
		 \end{split}
	\end{equation*}
	and according to Lebesgue dominated convergence theorem, we have $\|\Lambda_h^{s-1}\nabla V_2(t)\|_{H^1(\Omega)}$ is continuous on $[0,T]$.
	Taking $V=V_1+V_2$ and $P_1=2\nu\partial_1 V^1+\widetilde{F}_3^1$, we can see that $(V, P_1)$ satisfies \eqref{EB-multi-3.4} and
	\begin{equation*}
		\begin{split}
			&\|\Lambda_h^{s-1}\partial_tV\|_{L^2_T(L^2(\Omega))}^2+\|\Lambda_h^{s-1}\nabla V\|_{L^2_T(H^1(\Omega))}^2+\|\Lambda_h^{s-1}P_1\|_{L^2_T(H^1(\Omega))}^2\\
			&\lesssim
			\|\Lambda_h^{s-1}(\widetilde{F}_3,U)\|_{L^2_T(H^1(\Omega))}^2+\|\Lambda_h^{s-1} B^{2}_t\|_{L^2(L^2(\Omega))}^2\|\Lambda_h^{s-1}\nabla_h K^{(2)}\|_{L^{\infty}_T(L^2(\Omega))}^2\\
      &\quad+\| \Lambda_h^{s-1} B^{2}\|_{L^{\infty}_T(H^1(\Omega))}^2 \|\Lambda_h^{s-1} K^{(2)}_t \|_{L^2_T(L^2(\Omega))}^2,
		\end{split}
	\end{equation*}
	which along with Sobolev's embedding theorem lead to $\Lambda_h^{s-1} V\in\mathcal{C}([0,T]; L^2(\Omega))$.\\
   {\bf Step\,2.3:The estimates of $\|\Lambda_h^{s-1}V\|_{L^{\infty}_T(H^1(\Omega))}$}

	Next, we want to consider the continuity of $\|\Lambda_h^{s-1}\nabla V(t)\|_{L^2(\Omega)}$ in $[0,T]$.  However, by using \eqref{EB-multi-3.11} and \eqref{EB-multi-3.12} we can just get the  estimate near the bottom as follow
	\begin{equation*}
     \begin{split}
&\|\Lambda_{h}^{s-1} \nabla V_1(0)\|_{L^{2}(\Omega_{\frac{1}{2}b})}+\|\Lambda_{h}^{s-1} \nabla V_1\|_{L^{2}(\Omega_{\frac{1}{2}b} ;{L}^{\infty}_t)} \\
		&\lesssim  \bigg\||\varsigma_h|^2\langle\varsigma_h\rangle^{s-1} \frac{\mathrm{e}^{-rx_1}+\mathrm{e}^{r(2b+x_1)}}{\mathrm{e}^{2rb}+1}
\widehat{\Phi}\chi_{\{-b\leq x_1\leq -\frac{1}{2}b\}}\bigg\|_{L^2_h(L^2_{x_1}(L^1_{\tau}))} \\
    &\lesssim \bigg\|   \frac{ \mathrm{e}^{\mathrm{Re}\,r(2b+x_1)}}{\mathrm{e}^{2\mathrm{Re}\,rb}+1}\chi_{\{-b\leq x_1\leq -\frac{1}{2}b\}}\bigg\|_{L^{\infty}_h(L^2_{x_1}(L^2_{\tau}))} \||\varsigma_h|^2\langle\varsigma_h\rangle^{s-1} \hat{\Phi}\|_{L^2_h(L^2_{\tau}) }
    \lesssim  \|\widetilde{F}_3^{\perp}\|_{L_{\tau}^2(H^{s-\frac{1}{2}}(\Sigma_0))} ,
     \end{split}
	\end{equation*}
	where  $\Omega_{\frac{1}{2}b}:=\{x\in\Omega|-b\leq x_1\leq-\frac{1}{2}b\}$.

Let's now find the information about $\Lambda_h^{s-1}\nabla V_1(t)$ near the free upper bound.
	On the one hand, we find that $V_1$ satisfies
	\begin{equation*}\label{EB-multi-3.22-1}
       \begin{split}
		\frac{\mathrm{d}}{\mathrm{d}t}\|\Lambda_h^{s-1}\nabla_hV_1\|_{L^2(\Omega_f)}^2
		&=2\langle\Lambda_h^{s-1}\nabla_h\partial_tV_1,\Lambda_h^{s-1}\nabla_hV_1\chi_f\rangle\\
		&=-2(\Lambda_h^{s-1}\partial_tV_1,\Lambda_h^{s-1}\Delta_hV_1\chi_f)_{L^2(\Omega)}
       \end{split}
	\end{equation*}
	and
	\begin{equation*}\label{EB-multi-3.23-1}
		\Lambda_h^{s-1}\partial_1V_1^1=-\Lambda_h^{s-1}\nabla_h\cdot V_{1,h}.
	\end{equation*}
	According to Lemma \ref{Lions-M-lem}, we know that $\|\Lambda_h^{s-1}(\partial_1V_1^1,\nabla_hV_1)(t)\|_{L^2(\Omega_f)}$ is continuous on $[0,T]$.
	Moreover, we have
	\begin{equation*}
		\sup_{-\infty<t\leq T}\|\Lambda_h^{s-1}(\partial_1V_1^1,\nabla_hV_1)\|^2_{L^2(\Omega_f)}\leq   2\|\Lambda_h^{s-1}\partial_tV_1\|_{L^2(-\infty,T;L^2(\Omega_f))}\|\Lambda_{h}^{s-1} \nabla V_1\|_{L^2(-\infty,T;L^2(\Omega_f))}
	\end{equation*}
	
	On the other hand, inspired by the "good unknowns" in \cite{Gui2020}, we define $\mathcal{G}^{\beta}$ as
	\begin{equation*}
		\mathcal{G}^{\beta}:=\partial_1V_{1,f}^{\beta}
		+\partial_{\beta}V_{1,f}^1-\frac{1}{\nu}\widetilde{F}_{3,f}^{\beta}\,(\mbox{with} \,\beta=2, 3)
	\end{equation*}
	which along with the second equation in \eqref{EB-multi-3.4} implies $\mathcal{G}^{\beta}|_{x_1=0}=0$. Then,
	we find that $ V^{\beta}$$(\beta=2,3)$ satisfies
	\begin{equation*}
		\begin{split}
			&\int_{\Omega}\Lambda_h^{s-1}\partial_tV_{1,f}^{\beta}  \cdot \Lambda_h^{s-1}\partial_1\partial_1V_{1,f}^{\beta}\mathrm{d}x
			=\int_{\Omega}\Lambda_h^{s-1}\partial_tV_{1,f}^{\beta} \cdot \Lambda_h^{s-1}\partial_1(\mathcal{G}^{\beta}-\partial_{\beta}V_{1,f}^1+\frac{1}{\nu}\widetilde{F_3}^{\beta})\mathrm{d}x\\
			&= -\langle \Lambda_h^{s-1}\partial_t\partial_1V^{\beta}_{1,f},\Lambda_h^{s-1}\mathcal{G}^{\beta}\rangle
			-\int_{\Omega}\Lambda_h^{s-1}\partial_tV_{1,f}^{\beta} \cdot \Lambda_h^{s-1}(\partial_1\partial_{\beta}V_{1,f}^1-\frac{1}{\nu}\partial_1\widetilde{F}_{3,f}^{\beta})\mathrm{d}x\\	&=-\frac{1}{2}\frac{\mathrm{d}}{\mathrm{d}t}\|\Lambda_h^{s-1}\partial_1V_{1,f}^{\beta}\|_{L^2(\Omega)}^2
			-\langle\Lambda_h^{s-1}\partial_t\partial_1V^{\beta}_{1,f},\Lambda_h^{s-1}\partial_{\beta}V^1\rangle
			+\langle\Lambda_h^{s-1}\partial_t\partial_1V^{\beta}_{1,f},\frac{1}{\nu}\Lambda_h^{s-1}\widetilde{F}_{3,f}^{\beta}\rangle\\
			&\quad-\int_{\Omega} \Lambda_h^{s-1}\partial_tV_{1,f}^{\beta} \cdot \Lambda_h^{s-1}(\partial_1\partial_{\beta}V_{1,f}^1-\frac{1}{\nu}\partial_1\widetilde{F}_{3,f}^{\beta})\mathrm{d}x,
		\end{split}	
	\end{equation*}
	which lead to
	\begin{equation*}
		\begin{split}
			&\frac{1}{2}\frac{\mathrm{d}}{\mathrm{d}t}\left(\|\Lambda_h^{s-1}\partial_1V_{1,f}^{\beta}\|_{L^2(\Omega)}^2
			+\int_{\Omega}\Lambda_h^{s-1}\partial_1V_{1,f}^{\beta} \cdot \Lambda_h^{s-1}\partial_{\beta}V_{1,f}^1 \mathrm{d}x
			-\int_{\Omega}\Lambda_h^{s-1}\partial_1V_{1,f}^{\beta}\cdot \frac{1}{\nu}\Lambda_h^{s-1}\widetilde{F}_{3,f}^{\beta}\mathrm{d}x\right) \\
			&\quad\quad =- \int_{\Omega} \Lambda_h^{s-1}\partial_1\partial_1V_{1,f}^{\beta} \cdot \Lambda_h^{s-1}\partial_tV_{1,f}^{\beta}\mathrm{d}x
			- \int_{\Omega} \Lambda_h^{s-1}\partial_1\partial_{\beta}V_{1,f}^{\beta} \cdot \Lambda_h^{s-1}\partial_tV^1_{1,f}\mathrm{d}x\\
			& \quad\quad\quad +\langle\frac{1}{\nu}\Lambda_h^{s-1}(\partial_tF_{3,f}^{\beta}
         +\nu\partial_{\beta}\partial_tU^1_f+\nu\partial_1U_t^{\beta}\chi_f),
             \Lambda_h^{s-1}\partial_1V_{1,f}^{\beta}\rangle
			\\
&\quad\quad\quad -\int_{\Omega} \Lambda_h^{s-1}\partial_tV_{1,f}^{\beta}\cdot \Lambda_h^{s-1}(\partial_1\partial_{\beta}V_{1,f}^1-\frac{1}{\nu}\partial_1\widetilde{F}_{3,f}^{\beta})\mathrm{d}x,
		\end{split}	
	\end{equation*}
	where
	\begin{equation*}
		\begin{split}
	&\langle\frac{1}{\nu}\Lambda_h^{s-1}(\partial_tF_{3,f}^{\beta}
         +\nu\partial_{\beta}\partial_tU^1_f),
             \Lambda_h^{s-1}\partial_1V_{1,f}^{\beta}\rangle	\\	
			=&-\int_{\Omega} \frac{1}{\nu}\Lambda_h^{s-1}(B^2K^{(2)}_t)\chi_f \cdot \Lambda_h^{s-1}\partial_1\nabla_hV_{1,f}^{\beta}\mathrm{d}x
			+\int_{\Omega}\frac{1}{\nu}\Lambda_h^{s-1}((\nabla_h\cdot B^2)K_t^{(2)})\chi_f\cdot \Lambda_h^{s-1}\partial_1V_{1,f}^{\beta}\mathrm{d}x\\
			&+\int_{\Omega}\frac{1}{\nu}\Lambda_h^{s-1}(B_t^2:\nabla_h K^{(2)})\chi_f\cdot \Lambda_h^{s-1}\partial_1V_{1,f}^{\beta}\mathrm{d}x
			-\int_{\Omega}\Lambda_h^{s-1}\partial_tU^1_f
\cdot\Lambda_h^{s-1}\partial_{\beta}\partial_1V_{1,f}^{\beta}\mathrm{d}x,
		\end{split}	
	\end{equation*}
	and
	\begin{equation*}
		\begin{split}
			&\langle\Lambda_h^{s-1}\partial_1\partial_tU^{\beta}\chi_f,\Lambda_h^{s-1}\partial_1V_{1,f}^{\beta}\rangle
			=\langle\Lambda_h^{s-1}\partial_1\partial_{\beta}\partial_t\Psi\chi_f
      ,\Lambda_h^{s-1}\partial_1V_{1,f}^{\beta}\rangle\\
   =&\langle\Lambda_h^{s-1}\partial_{\beta}U^1_t\chi_f,\Lambda_h^{s-1}\partial_1V_{1,f}^{\beta}\rangle
			=-\int_{\Omega}\Lambda_h^{s-1}U^1_t \chi_f\cdot \Lambda_h^{s-1}\partial_1\partial_{\beta}V_{1,f}^{\beta}\mathrm{d}x.
		\end{split}	
	\end{equation*}
	Denoting
	\begin{equation*}
		\mathcal{E}_1^{\beta}(t):=\frac{1}{2}\|\Lambda_h^{s-1}\partial_1V_{1,f}^{\beta}\|_{L^2(\Omega)}^2
		-4\|\Lambda_h^{s-1}\partial_{\beta}V_{1,f}^1\|^2_{L^2}
		-4\|\frac{1}{\nu}\Lambda_h^{s -1}\widetilde{F}_{3,f}^{\beta}\|^2_{L^2},
	\end{equation*}
	we have
	\begin{equation*}
    \begin{split}
		\sup_{-\infty<t\leq T}\mathcal{E}_1^{\beta}(t)
\lesssim & \|\Lambda_h^{s-1}(\widetilde{F}_{3,f}^{\beta},V_t,U_t,\nabla V,\nabla^2 V)\|^2_{L^2(-\infty,T; L^2(\Omega))}\\
+ \|\Lambda_h^{s-1}& B^{2}_t\|_{L^2_T(L^2 )}^2\|\Lambda_h^{s-1}\nabla_h K^{(2)}\|_{L^{\infty}_T(L^2 )}^2
      +\| \Lambda_h^{s-1} B^{2}\|_{L^{\infty}_T(H^1 )}^2 \|\Lambda_h^{s-1} K^{(2)}_t\|_{L^2_T(L^2)}^2.
    \end{split}
	\end{equation*}
	Since $\|\Lambda_h^{s-1}(\partial_{\beta}V_{1,f}^1(t),\widetilde{F}_{3,f}^{\beta}(t))\|_{L^2}\in \mathcal{C}([0,T])$, so $\|\Lambda_h^{s-1}\partial_1V_{1,f}^{\beta}(t)\|_{L^2}$ continuous on $[0,T]$ due to Lemma \ref{Lions-M-lem}. Moreover, there is
	\begin{equation*}
		\sup_{0\leq t\leq T}\|\Lambda_h^{s-1}\partial_1V_{1,f}^{\beta}(t)\|^2_{L^2(\Omega)}\leq 2\sup_{0\leq t\leq T}\left(\mathcal{E}_1^{\beta}(t)+4\|\Lambda_h^{s-1}\partial_{\beta}V^1(t)\|^2_{L^2}
		+4\|\frac{1}{\nu}\Lambda_h^{s-1}\widetilde{F}_3^{\beta}(t)\|^2_{L^2}\right).
	\end{equation*}
	As a whole, we know that $ \Lambda_h^{s-1}V\in\mathcal{C}([0,T];H^1(\Omega))$, and \begin{equation*}
		\begin{split}
			 &\|\Lambda_h^{s-1} V \|^2_{L_T^{\infty}(H^1(\Omega))}      \lesssim \|\Lambda_h^{s-1} F_1\|^2_{L^2_T( L^2(\Omega))}
    \\
			&\quad +  \|\Lambda_h^{s-1}(B^{1}, B^{2})\|^2_{L^\infty_T( H^1(\Omega))}(\|\Lambda_h^{s-1}\nabla (K^{(1)}, K^{(2)})\|^2_{L^2_T( H^1(\Omega))}
  +\|\Lambda_h^{s-1}(K^{(1)}_t, K_t^{(2)})\|^2_{L^2_T( L^2(\Omega))})\\
			&\quad +(\|\Lambda_h^{s-1}(B^{1}, B^{2})\|^2_{L^\infty_T( H^1(\Omega))}+\|\Lambda_h^{s-1} (B_t^{1}, B_t^{2})\|^2_{L^2_T(  L^2(\Omega))})\|\Lambda_h^{s-1} (K^{(1)}, K^{(2)})\|^2_{L^\infty_T( H^1(\Omega))} .
		\end{split}
\end{equation*}
	
	Next, setting
	\[
	W := u-U-V, \quad \widetilde{f}_0:=U^1+V^1, \quad \widetilde{f}_1:= \widetilde{F}_1 -V_t+\nabla \cdot \mathbb{D}(V)-\nabla P_1,
	\]
	the initial problem reduces to solving $(W,Q)$ that satisfy
	\begin{equation}\label{EB-multi-3.16}
		\begin{cases}
			\partial_t \xi^1-W^1=\widetilde{f}_0&\quad \text{ on } [0,T]\times\Sigma_0,\\
			\partial_t W -\nu \nabla \cdot \mathbb{D} (W) + \nabla Q = \widetilde{f}_1 &\quad \text{ in } [0,T]\times\Omega,\\
			\nabla \cdot W=0 &\quad \text{ in } [0,T]\times\Omega,\\
			(Q  -g\xi^1)n_0 - \nu \mathbb{D} (W)n_0 =0 &\quad \text{ on } [0,T]\times\Sigma_0,\\
			W=0 & \quad \text{ on } [0,T]\times\Sigma_b,\\
			W|_{t=0} =W_0& \quad \text{ in } \Omega,\\
			\xi^1|_{t=0} = \xi_0^1& \quad \text{ on } \Sigma_0
		\end{cases}
	\end{equation}
with $W_0:= u_0-U_0-V_0$.
	
	{\bf Step 3: Solving the homogeneous boundary equations \eqref{EB-multi-3.16}}
	
    Consider the homogeneous boundary linear system \eqref{EB-multi-3.16}, we can get the following proposition, which will be proved at the end of this section.
	\begin{prop}\label{linear-system}
		Assume that $s>2$, $\widetilde{f}_0 \in L^2(0,T;H^s(\Sigma_0))$, $\Lambda_h^{s-1}\widetilde{f}_1\in L^2(0,T; L^2(\Omega))$, $\xi^1_0 \in  H^s(\Sigma_0)$, and $\Lambda_h^{s-1} W_0\in {_0}H^1(\Omega)$ with $\nabla\cdot W_0=0$, then there exists a unique solution $(\xi^1,W,Q)$ of \eqref{EB-multi-3.16} satisfying
		\begin{equation*}
			\begin{split}
				\xi^1\in\mathcal{C}([0,T];H^s(\Sigma_0)),\, \Lambda_h^{s-1} Q\in L^2(0,T; H^1(\Omega)),\, \Lambda_h^{s-1}W\in \mathcal{C}([0,T]; {_0}H^1(\Omega))\cap L^2(0,T;  H^2(\Omega)).
			\end{split}
		\end{equation*}
		Moreover, there holds
		\begin{equation}\label{unif-est-linear-1}
			\begin{split}
				&\sup_{0\leq t\leq T}\left(\|\xi^1\|^2_{H^s(\Sigma_0)}+\|\Lambda_h^{s-1}W\|^2_{H^1(\Omega)}\right)
				+\|\Lambda_h^{s-1}(\nabla W,Q)\|^2_{L^2(0,T;H^1(\Omega))} \\
				&\leq C (\|\xi^1_0\|^2_{H^s(\Sigma_0)}+\|\Lambda_h^{s-1}W_0\|^2_{H^1(\Omega)}+\|\widetilde{f}_0\|^2_{L^2(0,T;H^s(\Sigma_0))}+\|\Lambda_h^{s-1}\widetilde{f}_1\|^2_{L^2(0,T;L^2(\Omega))}).
			\end{split}
		\end{equation}
	\end{prop}

	{\bf Step 4: End of the proof of Theorem \ref{thm-linear-problim-11}}

   Set $u:=U+V+W$, $P=P_1+Q$, then putting all the previous steps together, we can get \eqref{estimate-all} and then complete the proof of Theorem \ref{thm-linear-problim-11}.
   \end{proof}

	In order to prove Proposition \ref{linear-system}, thanks to \cite{Hataya2011}, we get the Helmholtz decomposition in $\Omega$ as follows
		\begin{equation*}\label{Hel-decomposition}
			L^2(\Omega)=\mathbb{P}L^2(\Omega)\oplus\{\nabla \vartheta\in L^2|\vartheta\in H^1(\Omega), \vartheta=0\text{ on } \Sigma_0 \},
		\end{equation*}
		where $\mathbb{P}L^2(\Omega)\eqdefa  \overline{\{u\in L^2\cap\mathcal{C}^{\infty}(\overline{\Omega}):\nabla\cdot u=0\text{ in } \Omega, u\cdot n_0=0\text{ on }\Sigma_b\}}^{L^2(\Omega)} $. From this, we may decompose the pressure term $\nabla Q$ as $\nabla Q=\mathbb{P}(\nabla Q)+\nabla q$, where $\mathbb{P}(\nabla Q)=\nabla \pi_1+\nabla \pi_2$, and $(\pi_1,\,\pi_2)$ solves the system
		\begin{equation}\label{EB-multi-3.187}
			\begin{cases}
				\Delta \pi_i=0, (i=1,2) &\text{ in }\Omega,\\
				\pi_1=g\xi^1,\, \pi_2=-2\nu(\partial_2 W^2+\partial_3W^3) &\text{ on } \Sigma_0,\\
				\partial_1\pi_1=\partial_1\pi_2=0 &\text{ on } \Sigma_b.\\
			\end{cases}
		\end{equation}
Hence, the system \eqref{EB-multi-3.16} can be reformulated to the following equivalent system
		\begin{equation}\label{EB-multi-3.188}
			\begin{cases}
				\partial_t \xi^1-W^1=\widetilde{f}_0&\quad \text{ on } [0,T]\times\Sigma_0,\\
				\partial_t W -\nu \mathbb{P}\nabla \cdot \mathbb{D} (W) + \nabla \pi_1+\nabla \pi_2 = \mathbb{P}\widetilde{f}_1 &\quad \text{ in } [0,T]\times\Omega,\\
				\nabla \cdot W=0 &\quad \text{ in } [0,T]\times\Omega,\\
				\nu (\partial_1W^2+\partial_2 W^1) =0 &\quad \text{ on } [0,T]\times\Sigma_0,\\
				\nu (\partial_1W^3+\partial_2 W^1) =0 &\quad \text{ on } [0,T]\times\Sigma_0,\\
				W=0 & \quad \text{ on } [0,T]\times\Sigma_b,\\
				W|_{t=0} = W_0 & \quad \text{ in } \Omega,\\
				\xi^1|_{t=0} = \xi_0^1& \quad \text{ on } \Sigma_0.
			\end{cases}
		\end{equation}
		where $\pi_1$ and $\pi_2$ are determined by \eqref{EB-multi-3.187}.

Set $G\eqdefa\left(
			\begin{array}{cc}
				0 & 1\\
				A_1 &A_2
			\end{array}
			\right)$, $D(G)\eqdefa \{(\xi^1, W)^T:\,\,\nabla \cdot W=0 \,\,(\mbox{in}\,\,\Omega),\,\,(\partial_1W^{\alpha}+\partial_{\alpha} W^1)|_{\Sigma_0} =0\, (\mbox{with}\,\,\alpha=2, 3),\,W|_{\Sigma_b} =0\}$,
		where $A_1\xi^1=-\nabla \pi_1$ and $A_2 W=\nu \mathbb{P}\nabla \cdot \mathbb{D} (W)-\nabla \pi_2$. Then, the problem \eqref{EB-multi-3.188} is reduced to solving the following evolution equations
		\begin{equation*}
			\begin{cases}
				\partial_t
				\left(
				\begin{array}{c}
					\xi^{1}\\
					W
				\end{array}\right)
				-G
				\left(
				\begin{array}{c}
					\xi^{1}\\
					W
				\end{array}\right)
				=\left(
				\begin{array}{c}
					\widetilde{f}_0\\
					\widetilde{f}_1
				\end{array}
				\right)\\
				\left.\left(
				\begin{array}{c}
					\xi^{1}\\
					W
				\end{array}\right)\right|_{t=0}=
				\left(
				\begin{array}{c}
					\xi_0^{1}\\
					W_0
				\end{array}\right).
			\end{cases}
		\end{equation*}
For the operator $G$, we have the following lemma.
\begin{lem}\label{lem-dissi-G-oper-1}
   $G$ is a dissipative operator in $D(G)$ $\to$ $L^2(\Sigma_0)\times L_{\sigma}^2(\Omega)$.
\end{lem}
\begin{proof} For any $\lambda>0$, we have
\begin{equation*}
\begin{split}
    &\|(\lambda- G)(\xi^1,W)^T\|^2_{L^2(\Sigma_0)\times L^2(\Omega)}
    =((\lambda- G) (\xi^1,W)^T,(\lambda- G)(\xi^1,W)^T)_{L^2(\Sigma_0 \times  \Omega)}\\
   & =\lambda^2\| (\xi^1,W)^T\|^2_{L^2(\Sigma_0)\times L^2(\Omega)}
       +\|  G(\xi^1,W)^T\|^2_{L^2(\Sigma_0)\times L^2(\Omega)}
       -( \lambda (\xi^1,W)^T, G (\xi^1,W)^T)_{L^2(\Sigma_0 \times  \Omega)}.
\end{split}
\end{equation*}
Notice that
\begin{equation*}
\begin{split}
-( \lambda (\xi^1,W)^T, G (\xi^1,W)^T&)_{L^2(\Sigma_0 \times  \Omega)}
     = -\lambda(\xi^1,W^1)_{L^2(\Sigma_0)}-\lambda(A_1\xi^1+A_2 W,W)_{L^2(\Omega)}\\
     =&-\lambda(\xi^1,W^1)_{L^2(\Sigma_0)}-\lambda(-\nabla\pi_1+\nu\mathbb{P}\nabla\cdot\mathbb{D}(W)
     -\nabla\pi_2,W)_{L^2(\Omega)}\\
     =& -\lambda(\nu \nabla\cdot\mathbb{D}(W)
     -\nabla\pi_2,W)_{L^2(\Omega)} =\frac{\lambda}{2}\|\mathbb{D}W\|^2_{L^2(\Omega)},
\end{split}
\end{equation*}
then we have
\begin{equation*}
\begin{split}
     \|(\lambda- G)(\xi^1,W)^T\|^2_{L^2(\Sigma_0)\times L^2(\Omega)}
    \geq \lambda^2\| (\xi^1,W)^T\|^2_{L^2(\Sigma_0)\times L^2(\Omega)},
\end{split}
\end{equation*}
which finishes the proof of Lemma \ref{lem-dissi-G-oper-1}.
\end{proof}
Let's now prove Proposition \ref{linear-system}.
   \begin{proof}[Proof of Proposition \ref{linear-system}]
Thanks to Lemma \ref{lem-dissi-G-oper-1}, we know that $G$ is a dissipative operator, and the right half plane belongs to the resolvent set, and $G$ has a closed extension, still denoted by $G$, which generates a contraction semigroup $e^{tG}$. Hence, similar to  Lemma 2.2.1 in Chapter V.2 of \cite{Sohr2001}, we obtain there is a weak solution $(\xi^1,W)$ of \eqref{EB-multi-3.188} in $L^2_{loc}((0,T); L^2(\Sigma_0))\times (L^{\infty}_{loc}((0,T); L_{\sigma}^2(\Omega))\cap L^{2}_{loc}((0,T); { _{0}{H}}_{\sigma}^1(\Omega)))$. Therefore, in order to complete the proof of Proposition \ref{linear-system}, we need to give necessary {\it a priori} energy estimate \eqref{unif-est-linear-1}.
		
	In fact, by using the energy method, we can deduce from \eqref{EB-multi-3.188} that
		\begin{equation*}
			\begin{split}
				&\frac{1}{2}\frac{\mathrm{d}}{\mathrm{d}t}\left( \|\Lambda_{h}^{s}W\|^2_{L^{2}(\Omega)}+g\|\xi^1\|^2_{H^s(\Sigma_0)} \right)+\nu\|\Lambda_h^{s}\mathbb{D}(W)\|^2_{L^2(\Omega)}\\
				&=(\Lambda_{h}^{s-1}\widetilde{f}_1,\Lambda_{h}^{s+1}W)_{L^2(\Omega)}+(\Lambda_{h}^{s}\widetilde{f}_0, g\Lambda_{h}^{s}\xi^1)_{L^2(\Sigma_0)}.
			\end{split}
		\end{equation*}
		Thanks to H$\ddot{\mathrm{o}}$lder's inequality  and Lemma \ref{Korn's_inequality}, we have
		\begin{equation}\label{Step-3-1}
			\begin{split}
				&\frac{\mathrm{d}}{\mathrm{d}t}\left( \|\Lambda_{h}^{s}W\|^2_{L^{2}(\Omega)}+g\|\xi^1\|^2_{H^s(\Sigma_0)} \right)+\frac{3\nu}{2}\|\Lambda_h^{s}\mathbb{D}(W)\|^2_{L^2(\Omega)}\\
				&\leq\frac{2}{\nu}\|\Lambda_{h}^{s-1}\widetilde{f}_1\|^2_{L^2(\Omega)}
				+g\|\widetilde{f}_0\|_{H^s(\Sigma_0)}\|\xi^1\|_{H^s(\Sigma_0)}.
			\end{split}
		\end{equation}

		On the other hand, in order to avoid the loss of regularity on the boundary, similar to the proof of Theorem 1.3 in \cite{Gui2020}, we introduce new unknowns as follows:
		\begin{equation}\label{good_unknows_1}
			\begin{cases}
				\mathcal{G}_Q:=Q-g\mathcal{H}(\xi^1)+2\nu\nabla_h\cdot W^h,\\
				\mathcal{G}_W:=
				\left(
				\begin{array}{c}
					0 \\
					\mathcal{G}_W^2\\
					\mathcal{G}_W^3
				\end{array}\right)
				:=
				\left(
				\begin{array}{c}
					0\\
					\partial_1W^2+\partial_2 W^1\\
					\partial_1W^3+\partial_3 W^1
				\end{array}\right).
			\end{cases}
		\end{equation}
		Then,
		one can deduce that
		\begin{equation}\label{Eng-partial_1}
			\|\Lambda_h^{s-1}W_t\|^2_{L^2(\Omega)} +I+II
			= (\Lambda_h^{s-1}\widetilde{f}_1,\Lambda_h^{s-1}W_t)_{L^2(\Omega)} .
		\end{equation}
		where $I:=\int_{\Omega}\Lambda_h^{s-1}W_t\cdot \Lambda_h^{s-1}\nabla Q\mathrm{d}x$, and
		$II:=\nu\int_{\Omega}\Lambda_h^{s-1}W_t\cdot(\nabla\cdot\mathbb{D}(\Lambda_h^{s-1}W))\mathrm{d}x$.
		Using \eqref{good_unknows_1}$_1$,  the integral $I$ can be split into
		three parts
		\begin{equation*}
			\begin{split}
				I=&\int_{\Omega}\Lambda_h^{s-1}W_t\cdot \Lambda_h^{s-1}\nabla \mathcal{G}_Q\mathrm{d}x
				+\int_{\Omega}\Lambda_h^{s-1}W_t\cdot g\Lambda_h^{s-1}\nabla\mathcal{H}(\xi^1)\mathrm{d}x\\
				&-2\nu\int_{\Omega}\Lambda_h^{s-1}W_t\cdot \Lambda_h^{s-1}
				\nabla(\nabla_h\cdot W^h)\mathrm{d}x.
			\end{split}
		\end{equation*}
		Due to $\nabla\cdot W=0$, one can show that
		\begin{equation*}
			\begin{split}
				I=& \frac{\mathrm{d}}{\mathrm{d}t}\int_{\Omega}\Lambda_h^{s-1}W\cdot g\Lambda_h^{s-1}\nabla\mathcal{H}(\xi^1)\mathrm{d}x
				-\int_{\Omega}\Lambda_h^{s-1}W\cdot g\Lambda_h^{s-1}\nabla\mathcal{H}(W^1+\tilde{\eta})\mathrm{d}x\\
				&+\nu \frac{\mathrm{d}}{\mathrm{d}t}\|\Lambda_h^{s-1} \partial_1W^1\|^2_{L^2(\Omega)}
				-2\nu\int_{\Omega}\Lambda_h^{s-1}W^1_t\cdot \Lambda_h^{s-1}
				\partial_1(\nabla_h\cdot W^h)\mathrm{d}x.
			\end{split}
		\end{equation*}
		For $II$,  thanks to the new unknowns $\mathcal{G}_W$ defined by \eqref{good_unknows_1}$_2$ and $\nabla\cdot W=0$, we obtain that
		\begin{equation*}
			\begin{split}
				II=&-\nu\int_{\Omega}\Lambda_h^{s-1}W_t\cdot \Lambda_h^{s-1}\partial_1^2W \mathrm{d}x
				+\frac{\nu}{2}\frac{\mathrm{d}}{\mathrm{d}t}\|\Lambda_h^{s-1}\nabla_hW\|^2_{L^2(\Omega)}\\
				=& - \nu\int_{\Omega}\Lambda_h^{s-1} W_t\cdot \Lambda_h^{s-1}\partial_1  \mathcal{G}_W  \mathrm{d}x
				+\nu\int_{\Omega}\Lambda_h^{s-1}W^1_t\cdot \Lambda_h^{s-1}\partial_1 (\nabla_h\cdot W^h) \mathrm{d}x\\
				&+\frac{\nu}{2}\frac{\mathrm{d}}{\mathrm{d}t}\left(
				\|\Lambda_h^{s-1}\partial_1W^1\|^2_{L^2(\Omega)}
				+ \|\Lambda_h^{s-1}\nabla_hW\|^2_{L^2(\Omega)}\right).
			\end{split}
		\end{equation*}
		In view of the boundary conditions $W|_{\Sigma_b}=0$ and $\mathcal{G}_W|_{\Sigma_0} = 0$, one can deduce that
		\begin{equation*}
			\begin{split}
				&-\nu\int_{\Omega}\Lambda_h^{s-1}   W_t\cdot \Lambda_h^{s-1} \partial_1\mathcal{G}_W  \mathrm{d}x
				=\frac{\nu}{2}\frac{\mathrm{d}}{\mathrm{d}t}\|\Lambda_h^{s-1} \partial_1 W^h\|^2_{L^2{\Omega}}\\
				&\quad\quad\quad\quad+\nu\frac{\mathrm{d}}{\mathrm{d}t}\int_{\Omega}\Lambda_h^{s-1} \partial_1 W^2 \cdot \Lambda_h^{s-1}\partial_2 W^1 \mathrm{d}x
				+\nu \int_{\Omega}\Lambda_h^{s-1} \partial_2\partial_1 W^2 \cdot \Lambda_h^{s-1} W_t^1 \mathrm{d}x\\
				&\quad\quad\quad\quad+ \nu\frac{\mathrm{d}}{\mathrm{d}t}\int_{\Omega}\Lambda_h^{s-1} \partial_1 W^3 \cdot \Lambda_h^{s-1}\partial_3 W^1 \mathrm{d}x
				+\nu \int_{\Omega}\Lambda_h^{s-1} \partial_3\partial_1 W^3 \cdot \Lambda_h^{s-1} W_t^1 \mathrm{d}x .
			\end{split}
		\end{equation*}
		Define
		\begin{equation}\label{Eng-partial_2}
			\begin{split}
				\tilde{\mathcal{E}}_{s-1}:=& \|\Lambda_h^{s-1}\nabla W\|^2_{L^2(\Omega)}
				+ 2 \|\Lambda_h^{s-1}\partial_1W^1\|^2_{L^2(\Omega)}
				+\frac{2}{\nu}\int_{\Omega}\Lambda_h^{s-1}W\cdot g\Lambda_h^{s-1}\nabla\mathcal{H}(\xi^1)\mathrm{d}x\\
				&+2\int_{\Omega}\Lambda_h^{s-1} \partial_1 W^h \cdot \Lambda_h^{s-1}\nabla_h W^1 \mathrm{d}x,
			\end{split}
		\end{equation}
		combining $I$, $II$ with \eqref{Eng-partial_1} leads to
		\begin{equation*}
			\begin{split}
				\frac{\nu}{2}\frac{\mathrm{d}}{\mathrm{d}t} \tilde{\mathcal{E}}_{s-1}+
				\|\Lambda_h^{s-1}W_t\|^2_{L^2(\Omega)}
				=&(\Lambda_h^{s-1}\widetilde{f}_1,\Lambda_h^{s-1}W_t)_{L^2(\Omega)}+ (g\Lambda_h^{s-1}\nabla\mathcal{H}(W^1+\widetilde{f}_0),\Lambda_h^{s-1}W )_{L^2(\Omega)}.
			\end{split}
		\end{equation*}
		Then we have
		\begin{equation}\label{Step-3-2}
			\begin{split}
				& \nu \frac{\mathrm{d}}{\mathrm{d}t} \tilde{\mathcal{E}}_{s-1}+
				\|\Lambda_h^{s-1}W_t\|^2_{L^2(\Omega)}\\
				&\leq \|\Lambda_h^{s-1}\widetilde{f}_1\|^2_{L^2(\Omega)}+ 2g\left(\|\Lambda_h^{s-1}\nabla W\|_{L^2(\Omega)} +\|\widetilde{f}_0\|_{H^{s-\frac{1}{2}}(\Sigma_0)}\right)\|\Lambda_h^{s-1}W \|_{L^2(\Omega)},
			\end{split}
		\end{equation}
		due to  H\"{o}lder's inequality and Lemma \ref{Korn's_inequality}.
		
		Now, let's consider the estimates of $\|\Lambda_h^{s-1}\partial_1^2 W\|_{L^2(L^2(\Omega))}$. Since there are different boundary conditions on the upper boundary $\Sigma_0$ and the bottom boundary $\Sigma_b$, we need to deal with these two different situations separately.
		
		Define $W_b:=W\chi_b$, $Q_b:=Q\chi_b$ where $\chi_b$ defined by \eqref{cutfunction-B}, then according to \eqref{EB-multi-3.16}, $(W_b, Q_b)$ solves
		\begin{equation*}
			\begin{cases}
				-\nu\nabla\cdot\mathbb{D}(W_b)+\nabla Q_b=-\chi_b W_t+\widetilde{f}_b,\\
				\nabla \cdot W_b=\chi'_bW^1_b,\\
				(Q_b\mathbb{I}-\nu\mathbb{D}(W_b))n_0|_{\Sigma_0}=0,\\
				W_b|_{\Sigma_b}=0,
			\end{cases}
		\end{equation*}
		where
		\begin{equation*}
			\widetilde{f}_b=\widetilde{f}_1\chi_b-\nu[\nabla\cdot\mathbb{D};\chi_b]W+[\nabla;\chi_b]Q.
		\end{equation*}
		According to Lemma \eqref{lem-3-5},  we obtain the estimates near the bottom boundary as follow,
		\begin{equation}\label{Step-3-3.1}
			\begin{split}
				\|\Lambda_h^{s-1}\nabla W\|^2_{H^1(\Omega_b)}+ \|\Lambda_h^{s-1} \nabla Q\|^2_{L^2(\Omega_b)}
				\leq& \|\Lambda_h^{s-1}\widetilde{f}_1\|^2_{L^2(\Omega_b)}
				+\|\Lambda_h^{s-1}W_t\|^2_{L^2(\Omega_b)}\\
				&+\|\Lambda_h^{s-1}W\|_{H^1(\Omega)}
				+\|\Lambda_h^{s-1}Q\|^2_{L^2(\Omega_f)}.
			\end{split}
		\end{equation}
		Next, we need to consider the estimates near the upper boundary. Consider the equation of $\partial_1^2 W^h=(\partial_1^2 W^2,\partial_1^2 W^3)^T$
		\begin{equation*}
			-\nu\partial_1^2W^h+\nabla_hQ=-W^h_t+\nu\Delta_h W^h+\widetilde{f}_1^h.
		\end{equation*}
		By using the energy method, one can deduce that
		\begin{equation*}\label{Stokes-eeee}
			\begin{split}
				\nu\|\Lambda_h^{s-1}\partial_1^2W^h\|^2_{L^2(\Omega_f)}+IV=
				\int_{\Omega}\Lambda_h^{s-1} ( W^h_t-\nu\Delta_h W^h-\widetilde{f}_1^h )\cdot\Lambda_h^{s-1}\partial_1^2W^h \chi_f\mathrm{d}x,
			\end{split}
		\end{equation*}
		where
		\begin{equation*}
			\begin{split}
				IV=&-\int_{\Omega}\Lambda_h^{s-1}\nabla_hQ \cdot \Lambda_h^{s-1}\partial_1^2W^h \chi_f\mathrm{d}x\\
				=&-\int_{\Omega}\Lambda_h^{s-1}\partial_1Q \cdot \Lambda_h^{s-1}\partial_1 (\nabla_h\cdot W^h) \chi_f\mathrm{d}x
				-\int_{\Omega}\Lambda_h^{s-1}\partial_1Q \cdot \Lambda_h^{s-1}\partial_1 (\nabla_h\cdot W^h) \chi'_f\mathrm{d}x\\
				&-\int_{\Sigma_0}\Lambda_h^{s-1} \nabla_hQ \cdot \Lambda_h^{s-1}\partial_1 W^h  \mathrm{d}x
			\end{split}
		\end{equation*}
		Thanks to the boundary conditions of $\nabla_hW^1$ and $Q$ on $\Sigma_0$, we obtain
		\begin{equation*}
			\begin{split}
				-\int_{\Sigma_0}\Lambda_h^{s-1} \nabla_hQ \cdot &\Lambda_h^{s-1}\partial_1 W^h  \mathrm{d}x=\int_{\Sigma_0}\Lambda_h^{s-1} \nabla_h(g\xi^1-2\nu\nabla_h\cdot W^h)\cdot \Lambda_h^{s-1}\nabla_h W^1  \mathrm{d}x\\
				=&\frac{g}{2}\frac{\mathrm{d}}{\mathrm{d}t}\|\Lambda_h^{s-1}\nabla_h \xi^1\|^2_{L^{2}(\Sigma_0)}
				-\int_{\Sigma_0}\Lambda_h^{s-1} \nabla_h\xi^1\cdot \Lambda_h^{s-1}\nabla_h \tilde{\eta}  \mathrm{d}x\\
				&+2\nu\int_{\Sigma_0}\Lambda_h^{s-1} ( \nabla_h\cdot W^h)\cdot \Lambda_h^{s-1}\Delta_h W^1  \mathrm{d}x.
			\end{split}
		\end{equation*}
		Hence, we have
		\begin{equation*}
			\begin{split}
				&\nu\|\Lambda_h^{s-1}\partial_1^2W^h\|^2_{L^2(\Omega_f)}
				+\frac{g}{2}\frac{\mathrm{d}}{\mathrm{d}t}\|\Lambda_h^{s-1}\nabla_h \xi^1\|^2_{L^{2}(\Sigma_0)}\\
				&=\int_{\Omega}\Lambda_h^{s-1} ( W^h_t-\nu\Delta_h W^h-\widetilde{f}_1^h )\cdot\Lambda_h^{s-1}\partial_1^2W^h \chi_f\mathrm{d}x\\
				&\quad+\int_{\Omega}\Lambda_h^{s-1}\partial_1Q \cdot \Lambda_h^{s-1}\partial_1(\nabla_h\cdot W^h) \chi_f\mathrm{d}x
				+\int_{\Omega}\Lambda_h^{s-1}\partial_1Q \cdot \Lambda_h^{s-1}\partial_1(\nabla_h\cdot W^h) \chi'_f\mathrm{d}x\\
				&\quad+\int_{\Sigma_0}g\Lambda_h^{s-1} \nabla_h\xi^1\cdot \Lambda_h^{s-1}\nabla_h \widetilde{f}_0  \mathrm{d}x
				-2\nu\int_{\Sigma_0}\Lambda_h^{s-\frac{1}{2}} ( \nabla_h\cdot W^h)\cdot \Lambda_h^{s-\frac{3}{2}}\Delta_h W^1  \mathrm{d}x.
			\end{split}
		\end{equation*}
		Since $\partial_1Q=\tilde{f}^1+\Delta_hW^1-\partial_1(\nabla_h\cdot W^h)-W_t$, one has
		\begin{equation}\label{Step-3-3.2}
			\begin{split}
				&\frac{\nu}{2}\|\Lambda_h^{s-1}\partial_1^2W^h\|^2_{L^2(\Omega_f)}
				+\frac{g}{2}\frac{\mathrm{d}}{\mathrm{d}t}\|\Lambda_h^{s-1}\nabla_h \xi^1\|^2_{L^{2}(\Sigma_0)}\\
				&\leq C\left(\|\Lambda_h^{s-1}  W^h_t\|^2_{L^2(\Omega_f)}+
				\| \Lambda_h^{s}\nabla W^h \|_{L^2(\Omega)}
				+\|\Lambda_h^{s-1}\widetilde{f}_1\|^2_{L^2(\Omega_f)}  +\|\nabla_h\xi^1\|_{H^{s-1} }\| \nabla_h \widetilde{f}_0\|_{H^{s-1} }\right).
			\end{split}
		\end{equation}
		Moreover, we have the estimates of $\nabla Q$ as follow,
		\begin{equation}\label{Step-3-3.3}
			\begin{split}
				&\|\Lambda_h^{s-1}\partial_1Q\|^2_{L^2(\Omega)}\leq \|\Lambda_h^{s-1}  (W^h_t, \widetilde{f}_1)\|^2_{L^2(\Omega)}+
				\| \Lambda_h^{s}\nabla W^h \|^2_{L^2(\Omega)} , \\
				&\|\Lambda_h^{s-1}(\nabla_hQ-\nu \partial_1^2 W^h)\|^2_{L^2(\Omega_f)}\leq \|\Lambda_h^{s-1}  (W^h_t,\widetilde{f}_1)\|^2_{L^2(\Omega_f)}+
				\| \Lambda_h^{s+1} W^h \|^2_{L^2(\Omega_f)} , \\
				&\|\Lambda_h^{s-1}(\nabla_hQ-\nu\partial_1^2 W^h)\|^2_{L^2(\Omega )}\leq \|\Lambda_h^{s-1} ( W^h_t,\widetilde{f}_1)\|^2_{L^2(\Omega )}+
				\| \Lambda_h^{s+1} W^h \|^2_{L^2(\Omega )} .
			\end{split}
		\end{equation}
		Thanks to Lemma \ref{lem-korn-2}, we may get
		\begin{equation}\label{Step-3-3.4}
			\begin{split}
				\|\Lambda_h^{s-1}Q\|^2_{L^2(\Omega)}
				&\lesssim \|\Lambda_h^{s-1}Q\|^2_{L^2(\Sigma_0)}+\|\Lambda_h^{s-1}\partial_1Q\|^2_{L^2(\Omega)}  \\
				&\leq \| \xi^1\|^2_{H^s(\Sigma_0)}+
				\|\Lambda_h^{s-1}  (W^h_t, \widetilde{f}_1)\|^2_{L^2(\Omega)}
				+\|\Lambda_h^{s-1}\nabla_hW^h\|^2_{H^1(\Omega)}.
			\end{split}
		\end{equation}
		Then combining \eqref{Step-3-3.1}, \eqref{Step-3-3.3}, \eqref{Step-3-3.4} with \eqref{Step-3-3.2}, we may get
		\begin{equation}\label{Step-3-3}
			\begin{split}
				&c_0\|\Lambda_h^{s-1}(\partial_1^2W^h,Q,\nabla Q)\|^2_{L^2(\Omega)}
				+ \frac{\mathrm{d}}{\mathrm{d}t}\|\Lambda_h^{s-1}\nabla_h \xi^1\|^2_{L^{2}(\Sigma_0)}\\
				&\leq C_0\left(\|\Lambda_h^{s-1}  W^h_t\|^2_{L^2(\Omega )}+
				\| \Lambda_h^{s}\nabla W^h \|_{L^2(\Omega)}
				+\|\Lambda_h^{s-1}\widetilde{f}_1\|^2_{L^2(\Omega )}  +\| \xi^1\|_{H^{s} }+\|\widetilde{f}_0\|^2_{H^{s} }\right),
			\end{split}
		\end{equation}
		for some suitably positive constant $c_0$ and $C_0$.
		
		Finally, we will show the total energy estimates. For any given small positive constant $\kappa$, we set
		\begin{equation*}
			\begin{split}
				&\bar{\mathcal{E}}_{s-1}:=\|\Lambda_h^{s}W\|^2_{L^2(\Omega)}
				+ \|\xi^1\|_{H^s(\Sigma_0)}+\kappa\nu \tilde{\mathcal{E}}_{s-1}+\kappa^2 \|\nabla_h\xi^1\|_{H^{s -1}(\Sigma_0)},\\
				&\bar{\mathcal{D}}_{s-1}:=\frac{3\nu}{2}\|\Lambda_h^{s}\mathbb{D}(W)\|^2_{L^2(\Omega)}
				+\kappa  \|\Lambda_h^{s-1}W_t\|^2_{L^2(\Omega)}+\kappa^2c_0\|\Lambda_h^{s-1}(\partial_1^2W^h,Q,\nabla Q)\|^2_{L^2(\Omega)},
			\end{split}
		\end{equation*}
		From this, we find that there is a small enough $\kappa > 0$ such that the following inequalities hold
		\begin{equation}\label{CEEEEEE}
			\begin{split}
				&\bar{\mathcal{E}}_{s-1} \geq c_1(\|\Lambda_h^{s-1}W\|^2_{H^1(\Omega)}
				+ \|\xi^1\|_{H^s(\Sigma_0)}),\\
				&\bar{\mathcal{D}}_{s-1}\geq c_1( \|\Lambda_h^{s-1}\nabla W\|^2_{H^1(\Omega)}
				+\|\Lambda_h^{s-1}W_t\|^2_{L^2(\Omega)}+ \|\Lambda_h^{s-1} Q \|^2_{H^1(\Omega)}),
			\end{split}
		\end{equation}
		for some suitably small positive constant $c_1$.
		
		Combining \eqref{CEEEEEE} with \eqref{Step-3-1},\eqref{Step-3-2},\eqref{Step-3-3}, we obtain
		\begin{equation*}
			\begin{split}
				&c_1\frac{\mathrm{d}}{\mathrm{d}t}\left( \|\Lambda_{h}^{s-1}W\|^2_{H^{1}(\Omega)}+\|\xi^1\|^2_{H^s(\Sigma_0)} \right)+ c_1 \left(\|\Lambda_h^{s-1}(\nabla W,Q )\|^2_{H^1(\Omega)}
				+\|\Lambda_h^{s-1}W_t\|^2_{L^2(\Omega)} \right)\\
				&\leq\frac{2}{\nu}\|\Lambda_{h}^{s-1}\widetilde{f}_1\|^2_{L^2(\Omega)}
				+\|\widetilde{f}_0\|_{H^s(\Sigma_0)}\|\xi^1\|_{H^s(\Sigma_0)}\\
				&\quad+\kappa\left(\|\Lambda_h^{s-1}\widetilde{f}_1\|^2_{L^2(\Omega)}+ 4C_{korn}\|\Lambda_h^{s-1}\nabla W\|^2_{L^2(\Omega)} +\|\tilde{\eta}\|^2_{H^{s-\frac{1}{2}}(\Sigma_0)} \right)\\
				&\quad+\kappa^2C_0\left(\|\Lambda_h^{s-1}  W^h_t\|^2_{L^2(\Omega )}+
				\| \Lambda_h^{s}\nabla W^h \|_{L^2(\Omega)}^2
				+\|\Lambda_h^{s-1}\widetilde{f}_1\|^2_{L^2(\Omega )} +\| \xi^1\|_{H^{s} }^2+\| \widetilde{f}_0\|^2_{H^{s} }\right).
			\end{split}
		\end{equation*}
		Let $\kappa\leq\min\{1,\frac{c_0}{16C_{korn}},\sqrt{\frac{c_0}{4C_0}}\}$, then we have
		\begin{equation}\label{enr-est-diff-1}
			\begin{split}
				&c_1\frac{\mathrm{d}}{\mathrm{d}t}\left( \|\Lambda_{h}^{s-1}W\|^2_{H^{1}(\Omega)}+\|\xi^1\|^2_{H^s(\Sigma_0)} \right)+ \frac{c_1}{2} \left(\|\Lambda_h^{s-1}(\nabla W,Q )\|^2_{H^1(\Omega)}
				+\|\Lambda_h^{s-1}W_t\|^2_{L^2(\Omega)} \right)\\
				&\leq C  \|\Lambda_{h}^{s-1}\widetilde{f}_1\|^2_{L^2(\Omega)}
				+\|\widetilde{f}_0\|_{H^s(\Sigma_0)} \|\xi^1\|_{H^s(\Sigma_0)}
				+2\kappa \|\widetilde{f}_0\|^2_{H^{s}(\Sigma_0)}  +\kappa^2c'_0  \|\xi^1\|^2_{H^{s}(\Sigma_0)}.
			\end{split}
		\end{equation}
		Due to $\partial_t \xi^1-W^1=\widetilde{f}_0$ on $ [0,T]\times\Sigma_0$, one can see that $\forall\,\,t \in [0, T]$
		\begin{equation*}
			\begin{split}
				\sup_{\tau \in [0, t]}\| \xi^1(\tau)\|^2_{H^s(\Sigma_0)}
				\leq t\int_0^t\left(2C_{korn}\| \Lambda_h^{s-1}\nabla W\|^2_{L^2(\Omega)}+ \|\widetilde{f}_0\|^2_{H^s(\Sigma_0)}\right)\mathrm{d}\tau+ \|\xi^1_0\|_{H^s(\Sigma_0)}^2,
			\end{split}
		\end{equation*}
which along with \eqref{enr-est-diff-1} and Young's inequality implies
		\begin{equation*}
			\begin{split}
				c_1\frac{\mathrm{d}}{\mathrm{d}t}&\left( \|\Lambda_{h}^{s-1}W\|^2_{H^{1}(\Omega)}+\|\xi^1\|^2_{H^s(\Sigma_0)} \right)+ \frac{c_1}{2} \left(\|\Lambda_h^{s-1}(\nabla W,Q )\|^2_{H^1(\Omega)}
				+\|\Lambda_h^{s-1}W_t\|^2_{L^2(\Omega)} \right)\\
				\leq&C  \left(\|\Lambda_{h}^{s-1}\widetilde{f}_1\|^2_{L^2(\Omega)}
				+\|\widetilde{f}_0\|^2_{H^s(\Sigma_0)}\right)
				+(\frac{c_1}{8}+\kappa^2C_0)t\int_0^t\| \Lambda_h^{s-1}\nabla W\|^2_{L^2(\Omega)}\mathrm{d}\tau\\
				&+ (1+\kappa^2C_0)\left(t\int_0^t\|\widetilde{f}_0\|^2_{H^s(\Sigma_0)}\mathrm{d}\tau+ \|\xi^1_0\|_{H^s(\Sigma_0)}^2\right)
			\end{split}
		\end{equation*}
	Taking  $T=1$ and $0< \kappa\leq \sqrt{\frac{c_1}{8C_0}}$, we have
		\begin{equation*}
			\begin{split}
				c_1\frac{\mathrm{d}}{\mathrm{d}t}&\left( \|\Lambda_{h}^{s-1}W\|^2_{H^{1}(\Omega)}+\|\xi^1\|^2_{H^s(\Sigma_0)} \right)+ \frac{c_1}{2} \left(\|\Lambda_h^{s-1}(\nabla W,Q )\|^2_{H^1(\Omega)}
				+\|\Lambda_h^{s-1}W_t\|^2_{L^2(\Omega)} \right)\\
				\leq&C  \left(\|\Lambda_{h}^{s-1}\widetilde{f}_1\|^2_{L^2}
				+\|\widetilde{f}_0\|^2_{H^s}+\int_0^t\|\widetilde{f}_0\|^2_{H^s}\mathrm{d}\tau+ \|\xi^1_0\|_{H^s}^2\right)+ \frac{c_1}{4} \int_0^t\| \Lambda_h^{s-1}\nabla W\|^2_{L^2}\mathrm{d}\tau,
			\end{split}
		\end{equation*}
		which leads to \eqref{unif-est-linear-1}. 

We thus complete the proof of Theorem \ref{linear-system}.
	\end{proof}

\renewcommand{\theequation}{\thesection.\arabic{equation}}
\setcounter{equation}{0}

	\section{Proof of Theorem 1.1}\label{sect-4}
    Following the idea from \cite{Gui2020}, we rewrite the system \eqref{EB-multi-1.1} as the following linearized form
   \begin{equation}\label{fixed-point}
		\begin{cases}
			\partial_t\xi-v=0&\quad\text{in }\Omega,\\
			\partial_tv-\nu\nabla\cdot\mathbb{D}(v)+\nabla q= F_1(\nabla\xi,v,q) &\quad\text{in }\Omega,\\
			\nabla\cdot v=F_2(\nabla\xi,v)&\quad\text{in }\Omega,\\
			qn_0-g\xi^1n_0-\nu\mathbb{D}(v)n_0=F_3(\nabla\xi,v)&\quad\text{on }\Sigma_0,\\
			v=0&\quad\text{on }\Sigma_b,\\
			(\xi,v)|_{t=0} = (\xi_0,v_0)&\quad\text{in }\Omega,
		\end{cases}
	\end{equation}
	where
\begin{equation*}
\begin{split}
&F_1(\mathcal{A}(\nabla\xi),v,q):= -\nu(\nabla\cdot\mathbb{D}(v)-\nabla_{\mathcal{A}}\cdot\mathbb{D}_{\mathcal{A}}(v))+\nabla_{\mathbb{I}-\mathcal{A}} q,\\
&F_2(\nabla\xi,v):= {B}^{1}(\nabla\xi):\nabla v,  \quad F_3(\nabla\xi,v): = {B}^{2}(\nabla\xi):\nabla_h v.
\end{split}
\end{equation*}
     and $\mathcal{A}=\mathcal{A}(\nabla\xi)=(\nabla\xi+\mathbb{I})^{-1}$, ${B}^{1}= \mathbb{I}-a_{11}^{-1}J\mathcal{A}$  and $({B}^{2}:\nabla_h v)^j={B}^{2, j}_{\alpha,i}\partial_{\alpha} v^i$ $(i,j=1,2,3,\alpha=2,3)$.
   Here, ${B}^{2}={B}^{2}(\nabla\,\xi)$ is a $3\times 2\times 3$ tensor defined in Appendix (Sect.\ref{sect-appendix-1}).

     Notice that the matrix $\mathcal{A}-\mathbb{I}$ depends only on $\nabla \xi$ from the explicit form $\mathcal{A}=(\nabla\xi+\mathbb{I})^{-1}$, then we can get the following estimates about ${B}^{1}$ and ${B}^{2}$.
     \begin{lem}[Lemma 4.7 in \cite{Gui2020}]\label{estimate-B-B}
       Let  $s>2$ ,and assume $\nabla \xi$ and $\nabla \widetilde{\xi}$ satisfy
       \[
       \sup_{0\leq t\leq T}\big(\|\Lambda_h^{s-1}\nabla \xi\|_{H^1(\Omega)}+\|\Lambda_h^{s-1}\nabla   \widetilde{\xi}\|_{H^1(\Omega)}\big)\leq \epsilon_0,
       \]
        for  some suitably small positive constant $ \epsilon_0  \in (0, 1]$ , then there exists a positive constant $C$ such that
        \begin{equation}\label{prod-ab-est-one1}
        \begin{split}
        &\|\Lambda_h^{s-1}(\mathcal{A}-\mathbb{I}, {B}^{1}, {B}^{2})\|_{H^1(\Omega)}\leq C\|\Lambda_h^{s-1}\nabla \xi\|_{H^1(\Omega)} ,\\
        &\|\Lambda_h^{s-1} {B}^{1}_t \|_{L^2(\Omega)}\leq C(1+\|\Lambda_h^{s-1}\nabla \xi\|_{H^1(\Omega)})\|\Lambda_h^{s-1}\nabla \xi_t\|_{L^2(\Omega)} ,\\
        &\|\Lambda_h^{s-1} {B}^{2}_t \|_{L^2(\Omega)}\leq C\|\Lambda_h^{s-1}\nabla \xi\|_{H^1(\Omega)}\|\Lambda_h^{s-1}\nabla \xi_t\|_{L^2(\Omega)} ,
        \end{split}
      \end{equation}
        and
    \begin{equation}\label{prod-ab-est-one2}
        \begin{split}
        &\|\Lambda_h^{s-1}(\mathcal{A}-\widetilde{\mathcal{A}},B-\widetilde{B}, {B}^{2}-\widetilde{B}^2)\|_{H^1(\Omega)}\leq C \|\Lambda_h^{s-1}(\nabla \xi-\nabla\widetilde{\xi})\|_{H^1(\Omega)},\\
        &\|\Lambda_h^{s-1}(B^1_t-\widetilde{B}^1_t, {B}^{2}_t-\widetilde{B}^2_t)\|_{L^2(\Omega)}\leq C \big(\|\Lambda_h^{s-1}(\nabla \xi_t-\nabla\widetilde{\xi}_t)\|_{L^2(\Omega)}\\
        &\quad\quad\quad\quad\quad\quad\quad+\|\Lambda_h^{s-1}(\nabla \xi_t,\nabla\widetilde{\xi}_t)\|_{L^2(\Omega)}\|\Lambda_h^{s-1}(\nabla \xi-\nabla\widetilde{\xi})\|_{H^1(\Omega)}).
        \end{split}
    \end{equation}
     \end{lem}
     \begin{proof}
       Let $\epsilon_0$ be suitably small, we can find that $J$, $a_{11}$ and $|\overrightarrow{\rm{a}_1}|$ are near $1$. Then,  similar to the proof for Lemma 4.7 in \cite{Gui2020}, we may use the product estimates to get \eqref{prod-ab-est-one1} and \eqref{prod-ab-est-one2}.
     \end{proof}

      \subsection{Existence part of the proof of Theorem 1.1}
		Setting
		\begin{equation*}
			\begin{split}
				X_T:=\{(\xi^1,v,q)|&\xi^1 \in \mathcal{C}([0,T]; H^{s}(\Sigma_0)),
				\Lambda_h^{s-1}q \in L^2([0,T]; H^1(\Omega)),\\
				&\Lambda_h^{s-1}v \in \mathcal{C}([0,T]; {_0}{H}^1(\Omega)) \cap L^2([0,T]; H^2(\Omega))\}
			\end{split}
		\end{equation*}
		with the norm $\|\cdot\|_{X_T}$ defined by
		\begin{equation*}
			\begin{split}
				\|(\xi^1,v,q)\|_{X_T}:=&\|\xi^1 \|_{L^{\infty}([0, T]; H^s{(\Sigma_0)})}+\|\Lambda_h^{s-1}v\|_{L^{\infty}([0, T]; H^1(\Omega))}+\|\Lambda_h^{s-1}v\|_{L^2([0, T]; H^2(\Omega))}\\
				&+\|\Lambda_h^{s-1}v_t\|_{L^2([0, T]; L^2(\Omega))}+\|\Lambda_h^{s-1}q\|_{L^2([0, T]; H^1(\Omega))}.
			\end{split}
		\end{equation*}
		It's obvious that $(X_T,\|\cdot\|_{X_T})$ is a Banach space. Then, based on the fixed point method, we shall show the local existence of the system \eqref{fixed-point}.

{\bf Step 1: Construction of  the initial iteration $(\xi_1,v_1,q_1)$}

In order to get the first step of the iteration in the fixed point method, we will solve the linear system of $(\xi_1,v_1,q_1)$ as follows: 
		\begin{equation}\label{linear-equiv-form-one1}
			\begin{cases}
				\partial_t\xi_1  -v_1  =0&\quad\text{in } \Omega,\\
				\partial_tv_1-\nu\nabla\cdot\mathbb{D}(v_1)+\nabla q_1= F_1(\nabla\xi_0,v_1,q_1) &\quad\text{in } \Omega,\\
				\nabla\cdot v_1=F_2(\nabla\xi_0,v_1)&\quad\text{in } \Omega,\\
				(q_1-g\,\xi^1_1)n_0-\nu\mathbb{D}(v_1)n_0=F_3(\nabla\xi_0,v_1)&\quad\text{on } \Sigma_0,\\
				v_1=0&\quad\text{on } \Sigma_b,\\
				(\xi_1, v_1)|_{t=0} = (\xi_0, v_0)&\quad\text{in }\Omega,
			\end{cases}
		\end{equation}
        which is equivalent to
        \begin{equation}\label{linear-equiv-form-one2}
			\begin{cases}
				\partial_t\xi_1  -v_1  =0&\quad\text{in } \Omega,\\
				\partial_tv_1
-\nu\nabla_{\mathcal{A}_0}\cdot\mathbb{D}_{\mathcal{A}_0}(v_1)+\nabla_{\mathcal{A}_0} q_1= 0 &\quad\text{in } \Omega,\\
				\nabla_{\mathcal{A}_0}\cdot v_1=0&\quad\text{in } \Omega,\\
				(q_1-g\,\xi^1_1) \mathcal{N}_0
-\nu\mathbb{D}(v_1)\mathcal{N}_0=0&\quad\text{on } \Sigma_0,\\
				v_1=0&\quad\text{on } \Sigma_b,\\
				(\xi_1, v_1)|_{t=0} = (\xi_0, v_0)&\quad\text{in }\Omega.
			\end{cases}
		\end{equation}
        where $\mathcal{A}_0=\mathcal{A}(\nabla\xi_0)=(\mathbb{I}+\nabla\xi_0)^{-T}$ and $\mathcal{N}_0=J_0\mathcal{A}_0n_0$.

    Since $\xi^2$ and $\xi^3$ are decoupled with other unknowns on the boundaries $\Sigma_0$ and $\Sigma_b$ in \eqref{linear-equiv-form-one1}, we will first consider the linear system
              \begin{equation}\label{linear-equiv-form-one2-bdd}
			\begin{cases}
				\partial_t\xi_1^1  -v_1^1  =0&\quad\text{on } \Sigma_0,\\
				\partial_tv_1
-\nu\nabla_{\mathcal{A}_0}\cdot\mathbb{D}_{\mathcal{A}_0}(v_1)+\nabla_{\mathcal{A}_0} q_1= 0 &\quad\text{in } \Omega,\\
				\nabla_{\mathcal{A}_0}\cdot v_1=0&\quad\text{in } \Omega,\\
				(q_1-g\,\xi^1_1) \mathcal{N}_0
-\nu\mathbb{D}(v_1)\mathcal{N}_0=0&\quad\text{on } \Sigma_0,\\
				v_1=0&\quad\text{on } \Sigma_b,\\
            \xi_1^1|_{t=0}=\xi_0^1&\quad\text{ on } \Sigma_0,\\
				v_1|_{t=0} =v_0&\quad\text{in }\Omega.
			\end{cases}
		\end{equation}

	We will use the fixed point method to solve \eqref{linear-equiv-form-one2-bdd}. Indeed, given $(\xi^1_{1,(0)},v_{1,(0)},q_{1,(0)})\in X_T$ satisfying $\nabla_{J_0\mathcal{A}_0}\cdot v_{1,(0)}|_{t=0}=0$ and $\|(\xi^1_{1,(0)},v_{1,(0)},q_{1,(0)}  \|_{X_T}\leq 2C_0\left(   \|\xi_0^1\|_{H^{s}{(\Sigma_0)}} + \|\Lambda_h^{s-1} v_0\|_{H^1(\Omega)} \right)$
as the initial iteration data, we construct $(\xi^1_{1,(n)}, v_{1,(n)},q_{1,(n)})$ ($\forall\, n\in \mathbb{N}^+$) satisfying the following linear system
\begin{equation}\label{EB-4.10}
\begin{cases}
 \partial_t\xi^1_{1,(n)} - v^1_{1,(n)} =0 &\,\text{on } \Sigma_0,\\
\partial_tv_{1,(n)}-\nu\nabla\cdot\mathbb{D}(v_{1,(n)})+\nabla q_{1,(n)}= F_1(\nabla\xi_0,v_{1,(n-1)},q_{1,(n-1)}) &\,\text{in }\Omega,\\
\nabla\cdot v_{1,(n)}=F_2(\nabla\xi_0,v_{1,(n-1)})&\,\text{in } \Omega,\\
(q_{1,(n)} -g\,\xi^1_{1,(n)})n_0-\nu\mathbb{D}(v_{1,(n)})n_0=F_3(\nabla\xi_0,v_{1,(n-1)})&\,\text{on } \Sigma_0,\\
v_{1,(n)}=0&\,\text{on } \Sigma_b,\\
\xi^1_{1,(n)}|_{t=0} = \xi^1_0&\,\text{in }\Sigma_0,\\
v_{1,(n)}|_{t=0} = v_0&\,\text{in }\Omega.
			\end{cases}
		\end{equation}	
Denote this solution mapping by $\Upsilon_1$.
We may claim that, there exists $\varepsilon_1>0$ so small that if $\|\Lambda_h^{s-1}\nabla\xi_0\|_{H^1(\Omega)}\leq \varepsilon_1$, the sequence $\{(\xi^1_{1,(n)},v_{1,(n)},q_{1,(n)})\}_{n=1}^{\infty}$ satisfies the following uniform estimate
		\begin{equation}\label{xi-n-uni-est-1}
			\begin{split}
				\|(\xi^1_{1,(n)},v_{1,(n)},q_{1,(n)})  \|_{X_T}\leq 2C_0(   \|\xi_0^1\|_{H^{s}{(\Sigma_0)}} + \|\Lambda_h^{s-1} v_0\|_{H^1(\Omega)} )\quad (\forall\,\,n\in \mathbb{N}^+).
			\end{split}
		\end{equation}
In fact, applying Theorem \ref{thm-linear-problim-11} to \eqref{EB-4.10}, we prove that $(\xi^1_{1,(n)},v_{1,(n)},q_{1,(n)})\in X_T$ ($\forall\,\,n\in \mathbb{N}^+$) and
	\begin{equation*}
		\begin{split}
& \|(\xi^1_{1,(n)}, v_{1,(n)},q_{1,(n)})\|_{X_T}\\
&\leq C_0\bigg(   \|\xi_0^1\|_{H^{s}{(\Sigma_0)}}  + \|\Lambda_h^{s-1} u_0\|_{H^1(\Omega)}\bigg)\\
&\quad+C_0\bigg(\|\Lambda_h^{s-1} F_1\|_{L^2([0, T], L^2(\Omega))}  +  \|\Lambda_h^{s-1}(B^1(\nabla\xi_0), B^2(\nabla\xi_0))\|_{H^1(\Omega)}\|(\xi^1_{1,(n)},v_{1,(n)},q_{1,(n)})\|_{X_T}\bigg).
		\end{split}
\end{equation*}
So if $\|\Lambda_h^{s-1}\nabla\xi_0\|_{H^1(\Omega)}\leq \varepsilon_1$ with $\varepsilon_1:=\min\{\epsilon_0, \frac{1}{2C_0C_1}\}$, then we get by an induction argument that ($\forall\,\,n\in \mathbb{N}^+$)
		\begin{equation*}
			\begin{split}
				&  \|(\xi^1_{1,(n)},v_{1,(n)},q_{1,(n)}  \|_{X_T}\\
& \leq C_0(\|\xi_0^1\|_{H^{s}{(\Sigma_0)}}  + \|\Lambda_h^{s-1} u_0\|_{H^1(\Omega)})+C_0C_1\|\Lambda_h^{s-1} \nabla\xi_0\|_{H^1(\Omega)}\|(\xi^1_{1,(n)},v_{1,(n)},q_{1,(n)})\|_{X_T}\\
&\leq C_0(  \|\xi_0^1\|_{H^{s}{(\Sigma_0)}} + \|\Lambda_h^{s-1} v_0\|_{H^1(\Omega)})+2C_0^2C_1 \varepsilon_1\,(  \|\xi_0^1\|_{H^{s}{(\Sigma_0)}} + \|\Lambda_h^{s-1} v_0\|_{H^1(\Omega)})\\
&\leq 2C_0(  \|\xi_0^1\|_{H^{s}{(\Sigma_0)}} + \|\Lambda_h^{s-1} v_0\|_{H^1(\Omega)}),
			\end{split}
		\end{equation*}
which finishes the proof of \eqref{xi-n-uni-est-1}.

		Next, we will show that the sequence $\{(\xi^1_{1,(n)}, v_{1,(n)},p_{1,(n)})\}_{n=1}^{\infty} $ is a Cauchy sequence in $X_T$.
In fact, set  $\delta f_{(n)}:= f_{(n)}-f_{(n-1)}$ ($\forall\,n \in \mathbb{N}$), then $(\delta \xi_{1,(n+1)}^1,\delta  v_{1,(n+1)},\delta  q_{1,(n+1)}) \in X_T$ satisfies
		\begin{equation*}\label{EB-4.10-delta-v-1-n}
			\begin{cases}
				\partial_t\delta\xi^1_{1,(n+1)}  -\delta v^1_{1,(n+1)}  =0&\quad\text{on }\Sigma_0,\\
				\partial_t\delta v_{1,(n+1)}-\nu\nabla\cdot\mathbb{D}(\delta v_{1,(n+1)})+\nabla \delta q_{1,(n+1)}=  F_1(\nabla\xi_0,\delta v_{1,(n)},\delta q_{1,(n)}) &\quad\text{in }\Omega,\\
				\nabla\cdot \delta v_{1,(n+1)}= F_2(\nabla\xi_0,\delta v_{1,(n)})&\quad\text{in }\Omega,\\
				\delta q_{1,(n+1)}n_0-g\delta\xi^1_{1,(n+1)}-\nu\mathbb{D}(\delta v_{1,(n+1)})n_0= F_3(\nabla\xi_0,\delta v_{1,(n)})&\quad\text{on }\Sigma_0,\\
				\delta v_{1,(n+1)}=0&\quad\text{on }\Sigma_b,\\
				\delta\xi_{1,(n+1)}^1|_{t=0} =  0&\quad\text{on }\Sigma_0,\\
				\delta v_{1,(n+1)}|_{t=0} = 0&\quad\text{in }\Omega.
			\end{cases}
		\end{equation*}
Thanks to Theorem \ref{thm-linear-problim-11} again, we obtain
		\[
		\begin{split}
			&\|(\delta \xi_{1,(n+1)}^1,\delta  v_{1,(n+1)},\delta  q_{1,(n+1)})\|_{X_T}\\
			&\leq
			C_0\left(\|\Lambda_h^{s-1} ((\nabla\cdot\mathbb{D} -\nabla_{\mathcal{A}_0}\cdot\mathbb{D}_{\mathcal{A}_0})\delta v_{1,(n)},(\nabla-\nabla_{\mathcal{A}_0}) \delta q_{1,(n)})\|_{L^2(L^2(\Omega))} \right.\\
			&\quad\left.+\|\Lambda_h^{s-1}(B^1(\nabla\xi_0), B^2(\nabla\xi_0)\|_{H^1(\Omega)}\|(0,\delta v_{1,(n)},0)\|_{X_T}\right).
		\end{split}
		\]
		Let $\|\Lambda_h^{s-1}\nabla\xi_0\|_{H^1(\Omega)}\leq \min\{\frac{\epsilon_0}{2},\varepsilon_2\}$ and let $\varepsilon_2$ be small enough such that
		\begin{equation*}
			\begin{split}
				 \|\Lambda_h^{s-1} (\mathrm{I}-\mathcal{A}_0) \|_{L^{\infty}(H^1 )}(1+\|\Lambda_h^{s-1}  \mathcal{A}_0  \|_{L^{\infty}(H^1 )})
				\leq \frac{1}{8C_0} ,	
			\end{split}
		\end{equation*}
		and
		\begin{equation*}\label{BBxi}
 \|\Lambda_h^{s-1}(B^1(\nabla\xi_0), B^2(\nabla\xi_0))\|_{H^1(\Omega)}\leq \frac{1}{4C_0},
        \end{equation*}
        which implies that
        \[
		\begin{split}
			 \|(\delta \xi_{1,(n+1)}^1,\delta  v_{1,(n+1)},\delta  q_{1,(n+1)})\|_{X_T}
			\leq\frac{1}{2}\|(0,\delta v_{1,(n)},\delta  q_{1,(n )})\|_{X_T} .
		\end{split}
		\]
	From this, we know that the solution mapping $\Upsilon_1: X_T \rightarrow X_T$ in \eqref{EB-4.10} is a contraction mapping. Hence, the sequence $\{(\xi^1_{1,(n)},v_{1,(n)},q_{1,(n)})\}_{n=1}^{\infty}\subset X_T$ is a Cauchy sequence, and then there is  $(\xi^1_{1},v_1,q_1)\in X_T$, the fixed point of the mapping $\Upsilon_1$, so that $\{(\xi^1_{1,(n)},v_{1,(n)},q_{1,(n)})\}_{n=1}^{\infty}$ converges to $(\xi^1_{1},v_1,q_1)$ in $X_T$. Moreover, we may check that $(\xi^1_{1},v_1,q_1)$ is the solution to \eqref{linear-equiv-form-one2-bdd} in $X_T$, and there holds
		\begin{equation}\label{initial-data-induct-1}
		\|(\xi^1_1,v_1,q_1)\|_{X_T}
		\leq 2C_0
		\left(\|\Lambda_h^{s-1}\xi^1_0\|_{H^1(\Sigma_0)} +\|\Lambda_h^{s-1}v_0\|_{H^1(\Omega)}\right).
	\end{equation}
	With $v_1$ above in hand, let
		\begin{equation}\label{assume-T}
		T\leq \frac{\|\Lambda_h^{s-1}\nabla\xi_0\|^2_{H^1(\Omega)}}{ 8C_0^2
			(\|\Lambda_h^{s-1}\xi^1_0\|_{H^1(\Sigma_0)} +\|\Lambda_h^{s-1}v_0\|_{H^1(\Omega)} )^2} ,
		\end{equation}
       and we solve $\xi_{1}$ in $\Omega$ by the linear ODEs
		\begin{equation*}
			\begin{cases}
				\partial_t\xi_{1}    =v_{1}&\quad\text{in }(0, T]\times \Omega,\\
				\xi_{1}|_{t=0} = \xi_0 &\quad\text{in }\Omega,
			\end{cases}
		\end{equation*}
		to get $\Lambda_h^{s-1}\nabla\xi_{1} \in \mathcal{C}([0,T];  H^1(\Omega))$, 
		  \begin{equation*}
		\begin{split}
			\|\Lambda_h^{s-1}\nabla\xi_1\|_{L^{\infty}_T(H^1(\Omega))}
			&\leq  2\|\Lambda_h^{s-1}\nabla\xi_0\|_{H^1(\Omega)} ,\\
            \|\Lambda_h^{s-1}\partial_t\nabla\xi_1\|_{L^2_T(L^2(\Omega))}
&\leq     T^{\frac{1}{2}}  \|\Lambda_h^{s-1}\nabla v_1\|_{L^2_T(L^2(\Omega))} \leq2\|\Lambda_h^{s-1}\nabla\xi_0\|_{H^1(\Omega)} .
		\end{split}
		\end{equation*}
and then
		  \begin{equation}\label{nabla-xi-1}
		\begin{split}
			\|\Lambda_h^{s-1}\nabla\xi_1\|_{L^{\infty}_T(H^1(\Omega))}+\|\Lambda_h^{s-1}\partial_t\nabla\xi_1\|_{L^2_T(L^2(\Omega))}\leq 4\|\Lambda_h^{s-1}\nabla\xi_0\|_{H^1(\Omega)}.
		\end{split}
		\end{equation}

		{\bf Step 2: Construction of approximate sequence.}\par
		Taking $(\xi_1,v_1,q_1)$ obtained in Step 1 as the initial value of the iteration, and letting $(\xi_{n},v_{n},q_{n})$ be a solution of the following linear problem $(n\in \mathbb{N}^+\text{ and }n\geq2)$:
		\begin{equation}\label{EB-4.10-d-non-linear}
			\begin{cases}
				\partial_{t}\xi_{n}-v_{n}=0,&\quad\text{in }\Omega,\\
				\partial_tv_{n}-\nu\nabla\cdot\mathbb{D}(v_{n})+\nabla q_{n}= F_1(\nabla\xi_{n-1},v_{n},q_{n}), &\quad\text{in }\Omega,\\
				\nabla\cdot v_{n}=F_2(\nabla\xi_{n-1},v_{n}),&\quad\text{in }\Omega,\\
				q_{n}n_0-g\xi^1_{n}n_0-\nu\mathbb{D}(v_{n})n_0=F_3(\nabla\xi_{n-1},v_{n}),&\quad\text{on }\Sigma_0,\\
				v_{n}=0,&\quad\text{on }\Sigma_b,\\
				(\xi_{n},v_{n})|_{t=0} = (\xi_0 ,v_0),&\quad\text{in }\Omega ,
			\end{cases}
		\end{equation}
         which is equivalent to
        \begin{equation*}
			\begin{cases}
				\partial_t\xi_{n}  -v_{n}  =0&\quad\text{in }\Omega,\\
				\partial_tv_{n}
-\nu\nabla_{\mathcal{A}_{n-1}}\cdot\mathbb{D}_{\mathcal{A}_{n-1}}(v_{n})+\nabla_{\mathcal{A}_{n-1}} q_{n}= 0 &\quad\text{in }\Omega,\\
				\nabla_{\mathcal{A}_{n-1}}\cdot v_{n}=0&\quad\text{in }\Omega,\\
				(q_{n}-g\,\xi^1_{n}) \mathcal{N}_{n-1}
-\nu\mathbb{D}(v_{n})\mathcal{N}_{n-1}=0&\quad\text{on }\Sigma_0,\\
				v_{n}=0&\quad\text{on }\Sigma_b,\\
				(\xi_{n}, v_{n})|_{t=0} = (\xi_0, v_0)&\quad\text{in }\Omega,
			\end{cases}
		\end{equation*}
        where $\mathcal{A}_n=\mathcal{A}(\nabla\xi_n)=(\mathbb{I}+\nabla\xi_n)^{-T}$ and $\mathcal{N}_n=J_n\mathcal{A}_nn_0$.
        For fixed $n \in \mathbb{N}$, in order to solve \eqref{EB-4.10-d-non-linear}, we first consider the linear system
        \begin{equation}\label{EB-non-nm-bdd}
\begin{cases}
 \partial_t\xi^1_{n} - v^1_{n} =0 &\,\text{on }\Sigma_0,\\
\partial_tv_{n}-\nu\nabla\cdot\mathbb{D}(v_{n})+\nabla q_{n}= F_1(\nabla\xi_{n-1},v_{n},q_{n}) &\,\text{in }\Omega,\\
\nabla\cdot v_{n}=F_2(\nabla\xi_{n-1},v_{n})&\,\text{in }\Omega,\\
(q_{n} -g\,\xi^1_{n})n_0-\nu\mathbb{D}(v_{n})n_0=F_3(\nabla\xi_{n-1},v_{n})&\,\text{on }\Sigma_0,\\
v_{n}=0&\,\text{on }\Sigma_b,\\
\xi^1_{n} |_{t=0}=\xi^1_0 &\,\text{on }\Sigma_0,\\
v_{n}|_{t=0} = v_0&\,\text{in }\Omega.
			\end{cases}
		\end{equation}

Choose $(\xi^1_{n,(0)}, v_{n,(0)},q_{n,(0)}):=(\xi^1_1, v_{1},q_{1})$ and construct the sequence
$(\xi^1_{n,(m)}, v_{n,(m)},q_{n,(m)})$ ($\forall\, m\in \mathbb{N}^+$) of \eqref{EB-non-nm-bdd} solving
\begin{equation}\label{EB-non-nm}
\begin{cases}
 \partial_t\xi^1_{n,(m)} - v^1_{n,(m)} =0 &\,\text{on }\Sigma_0,\\
\partial_tv_{n,(m)}-\nu\nabla\cdot\mathbb{D}(v_{n,(m)})+\nabla q_{n,(m)}= F_1(\nabla\xi_{n-1},v_{n,(m-1)},q_{n,(m-1)}) &\,\text{in }\Omega,\\
\nabla\cdot v_{n,(m)}=F_2(\nabla\xi_{n-1},v_{n,(m-1)})&\,\text{in }\Omega,\\
(q_{n,(m)} -g\,\xi^1_{n,(m)})n_0-\nu\mathbb{D}(v_{n,(m)})n_0=F_3(\nabla\xi_{n-1},v_{n,(m-1)})&\,\text{on }\Sigma_0,\\
v_{n,(m)}=0&\,\text{on }\Sigma_b,\\
\xi^1_{n,(m)} |_{t=0}=\xi^1_0 &\,\text{on }\Sigma_0,\\
v_{n,(m)}|_{t=0} = v_0&\,\text{in }\Omega.
			\end{cases}
		\end{equation}	
Thanks to Theorem \ref{thm-linear-problim-11}, we know that the solution mapping $\Upsilon_n: X_T\rightarrow X_T$ of the system \eqref{EB-non-nm} is well-defined, and there holds
		 \begin{equation}\label{est-n-m-eng-1}
		\begin{split}
			&\|(\xi_{n,(m)}^1,v_{n,(m)},q_{n,(m)})\|_{X_T}\\
&\leq  C_0(   \|\xi_0^1\|_{H^{s}{(\Sigma_0)}}  + \|\Lambda_h^{s-1} v_0\|_{H^1(\Omega)}+\|\Lambda_h^{s-1} F_1(\nabla\xi_{n-1},v_{n,(m-1)},q_{n,(m-1)})\|_{L^2_T(L^2(\Omega))}\\
			&  +  \|\Lambda_h^{s-1}(B^1(\nabla\xi_{n-1}), B^2(\nabla\xi_{n-1}) )\|_{L^{\infty}(H^1)}\|(0,v_{n,(m-1)},0)\|_{X_T}\\
&+\|\Lambda_h^{s-1} (B^1_t(\nabla\xi_{n-1}), B^2_t(\nabla\xi_{n-1}))\|_{L^{2}_T(L^2)}\|\Lambda_h^{s-1} v_{n,(m-1)}\|_{L^{\infty}_T(H^1)}).
		\end{split}
\end{equation}

For $\|\Lambda_h^{s-1}\nabla\xi_0\|_{H^1(\Omega)}\leq \varepsilon_2 $ with $\varepsilon_2:= \min\{\frac{\epsilon_0}{4}, \frac{1}{16C_0C_2}\}$, we may employ lemma \ref{estimate-B-B} and an induction argument to obtain that, $\forall\,\, n\geq 2,\,\,m\in \mathbb{N}^+$,
       \begin{equation}\label{assume-xixit}
		\begin{split}
			&\|\Lambda_h^{s-1} \nabla\xi_{n-1}\|_{L^{\infty}_T(H^1(\Omega))}
+\|\Lambda_h^{s-1} \partial_t\nabla\xi_{n-1}\|_{L^2_T(L^2(\Omega))}
\leq  4\|\Lambda_h^{s-1}\nabla\xi_0\|_{H^1(\Omega)}
		\end{split}
		\end{equation}
and
       \begin{equation}\label{estim-v-nm}
		\begin{split}
			&\|(\xi_{n,(m)}^1,v_{n,(m)},q_{n,(m)})\|_{X_T} \leq C_0(\|\xi_0^1\|_{H^{s}{(\Sigma_0)}}  + \|\Lambda_h^{s-1} u_0\|_{H^1(\Omega)})\\
&\quad\quad
+ C_0C_2\|\Lambda_h^{s-1} \nabla\xi_{n-1}\|_{L^{\infty}_T(H^1(\Omega))}\|(\xi^1_{n,(m-1)},v_{1,(n)},q_{n,(m-1)})\|_{X_T}\\
&\quad\quad +C_0C_2\|\Lambda_h^{s-1}\nabla \partial_t\xi_{n-1}\|_{L^2_T(H^1(\Omega))}\|(\xi^1_{n,(m-1)},v_{1,(n)},q_{n,(m-1)})\|_{X_T} \\
&\quad\leq C_0(  \|\xi_0^1\|_{H^{s}{(\Sigma_0)}} + \|\Lambda_h^{s-1} v_0\|_{H^1(\Omega)})+16C_0^2C_2 \varepsilon_2\,(  \|\xi_0^1\|_{H^{s}{(\Sigma_0)}} + \|\Lambda_h^{s-1} v_0\|_{H^1(\Omega)})\\
&\quad\leq 2C_0(  \|\xi_0^1\|_{H^{s}{(\Sigma_0)}} + \|\Lambda_h^{s-1} v_0\|_{H^1(\Omega)}).
		\end{split}
		\end{equation}
       In order to show that the sequence $\{(\xi_{n,(m)}^1,v_{n,(m)},q_{n,(m)})\}_{m=1}^{\infty}$ is a Cauchy sequence, let us consider the following system
		\begin{equation*}\label{EB-mmmmmmmmmmmm}
			\begin{cases}
				\partial_t\delta\xi^1_{n,(m+1)}-\delta v^1_{n,(m+1)}=0 &\quad\text{on }\Sigma_0,\\
				\partial_t\delta v_{n,(m+1)}-\nu\nabla\cdot\mathbb{D}(\delta v_{n,(m+1)})+\nabla \delta q_{n,(m+1)}=  F_1(\nabla\xi_{n-1},\delta v_{n,(m)},\delta q_{n,(m)}) &\quad\text{in }\Omega,\\
				\nabla\cdot \delta v_{n,(m+1)}= F_2(\nabla\xi_{n-1},\delta v_{n,(m)})&\quad\text{in }\Omega,\\
				\delta q_{n,(m+1)}n_0-g\delta\xi^1_{n,(m+1)}n_0-\nu\mathbb{D}(\delta v_{n,(m+1)})n_0= F_3(\nabla\xi_{n-1},\delta v_{n,(m)})&\quad\text{on }\Sigma_0,\\
				\delta v_{n,(m+1)}=0&\quad\text{on }\Sigma_b,\\\end{cases}
		\end{equation*}
       with the initial data $\delta \xi^1_{n,(m+1)}|_{t=0} = 0$ on $\Sigma_0$ and $\delta v_{n,(m+1)}|_{t=0} = 0 $ in $\Omega$.
        Thanks to Theorem \ref{thm-linear-problim-11} again, we can obtain
		\[
		\begin{split}
			&\|(\delta \xi_{n,(m+1)}^1,\delta  v_{n,(m+1)},\delta  q_{n,(m+1)})\|_{X_T}
			\leq
			C_0\big(\| F_1(\nabla\xi_{n-1},\delta v_{n,(m+1)},\delta q_{n,(m+1)})\|_{L^2(L^2(\Omega))}  \\
			&\qquad\qquad\qquad+\|\Lambda_h^{s-1}(B^1(\nabla\xi_{n-1}), B^2(\nabla\xi_{n-1})\|_{L^{\infty}(H^1)}\|(0,\delta v_{n,(m+1)},0)\|_{X_T} \\
			&\qquad\qquad\qquad+\|\Lambda_h^{s-1}(B^1_t(\nabla\xi_{n-1}), B^2_t(\nabla\xi_{n-1})\|_{L^2(L^2)}\|(0,\delta v_{n,(m+1)},0)\|_{X_T} \big).\\
		\end{split}
		\]
      Then for $\|\Lambda_h^{s-1}\nabla\xi_0\|_{H^1(\Omega)}\leq \varepsilon_3 $ with $\varepsilon_3:= \min\{\frac{\epsilon_0}{4}, \frac{1}{2C_0C_3}\}$, we arrive at
     \[
		\begin{split}
\|(\delta \xi_{n,(m+1)}^1,\delta  v_{n,(m+1)},\delta  q_{n,(m+1)})\|_{X_T}
			\leq&
			C_0 C_3\varepsilon_3\|(\delta \xi_{n,(m)}^1,\delta  v_{n,(m)},\delta  q_{n,(m)})\|_{X_T}  \\
           \leq&
			\frac{1}{2}\|(\delta \xi_{n,(m)}^1,\delta  v_{n,(m)},\delta  q_{n,(m)})\|_{X_T},
		\end{split}
		\]
      which implies that $\{(\xi_{n,(m)}^1,v_{n,(m)},q_{n,(m)})\}_{m=1}^{\infty}$ is a Cauchy sequence in $X_T$, and denote its limit by $(\xi_{n}^1,v_{n},q_{n})$. Similarly, we may readily verify that $(\xi_{n}^1,v_{n},q_{n})$ solves \eqref{EB-non-nm-bdd}.

        Due to \eqref{assume-T} and \eqref{nabla-xi-1},
        we therefore get a sequence $\{(\xi_{n}^1,v_{n},q_{n})\}_{m=1}^{\infty}\subset X_T$  by an induction argument about the index $n(=2,3,4,\cdots)$ and we solve $\xi_{n}$ in $\Omega$ by the linear ODEs
		\begin{equation*}
			\begin{cases}
				\partial_t\xi_{n}    =v_{n}&\quad\text{in }(0, T]\times \Omega,\\
				\xi_{n}|_{t=0} = \xi_0 &\quad\text{in }\Omega
			\end{cases}
		\end{equation*}
		to obtain  $\Lambda_h^{s-1}\nabla\xi_{n} \in \mathcal{C}([0,T]; H^1(\Omega))$ and
		\begin{equation}\label{estima-unif-gradxi}
		\begin{split}
			&\|\Lambda_h^{s-1}\nabla\xi_{n}\|_{L^{\infty}_T(H^1(\Omega))}
			\leq  \|\Lambda_h^{s-1}\nabla\xi_0\|_{H^1(\Omega)}+ T^{\frac{1}{2}} \|\Lambda_h^{s-1}\nabla v_n\|_{L^2_T(H^1(\Omega))},\\
			&\|\Lambda_h^{s-1}\partial_t\nabla\xi_{n}\|_{L^2_T(H^1(\Omega))}
			\leq  T^{\frac{1}{2}}\|\Lambda_h^{s-1}\nabla v_n\|_{L^{\infty}_T(L^2(\Omega))}.
		\end{split}
		\end{equation}
Therefore, we get that  $(\xi_{n},v_{n},q_{n})$ is a solution of \eqref{EB-4.10-d-non-linear} for $n\in \mathbb{N}^+$ and $n\geq 2$, and obtain the uniform estimates of $\{(\xi_{n},v_{n},q_{n})\}_{n=1}^{\infty}$ from \eqref{estim-v-nm} and \eqref{estima-unif-gradxi}, that is ($n\in\mathbb{N}^+$)
		\begin{equation}\label{uniform-estimates-n}
			\begin{split}
				&\|(\xi_n^1,v_n,q_n)\|_{X_T}
				\leq 2C_0\big(   \|\xi_0^1\|_{H^{s}{(\Sigma_0)}}  + \|\Lambda_h^{s-1} v_0\|_{H^1(\Omega)}    \big),\\
				&\|\Lambda_h^{s-1}\nabla\xi_n\|_{L^{\infty}_T(H^1(\Omega))}
+\|\Lambda_h^{s-1}\partial_t\nabla\xi_n\|_{L^2_T(L^2(\Omega))}
				\leq 4 \|\Lambda_h^{s-1}\nabla\xi_0\|_{H^1(\Omega)} .
			\end{split}
		\end{equation}
		
		{\bf Step 3: Convergence of the approximate sequence.}
		
		With Step 1 and Step 2 in hand, we will show that the sequence $\{(\xi^1_{n},v_{n},q_{n})\}_{n=1}^{\infty}$  is a Cauchy sequence.
		In order to do so, we consider the system
		\begin{equation*}\label{EB-4.1111}
			\begin{cases}
				\partial_t\delta\xi^1_{n+1}-\delta v^1_{n+1}=0 &\quad\text{on }\Sigma_0,\\
				\partial_t\delta v_{n+1}-\nu\nabla\cdot\mathbb{D}(\delta v_{n+1})+\nabla \delta q_{n+1}= \delta F_1(\nabla\xi_{n},v_{n+1},q_{n+1}) &\quad\text{in }\Omega,\\
				\nabla\cdot \delta v_{n+1}=\delta F_2(\nabla\xi_{n},v_{n+1})&\quad\text{in }\Omega,\\
				\delta q_{n+1}n_0-g\delta\xi^1_{n+1}n_0-\nu\mathbb{D}(\delta v_{n+1})n_0=\delta F_3(\nabla\xi_{n},v_{n+1})&\quad\text{on }\Sigma_0,\\
				\delta v_{n+1}=0&\quad\text{on }\Sigma_b,\\
				\delta \xi^1_{n+1}|_{t=0} = 0&\quad\text{on }\Sigma_0, \\
             \delta v_{n+1})|_{t=0} = 0&\quad\text{in }\Omega.
			\end{cases}
		\end{equation*}
		Thanks to Theorem \ref{thm-linear-problim-11} and \eqref{uniform-estimates-n}, we obtain
		\[
		\begin{split}
			&\|(\delta \xi_{n+1}^1,\delta  v_{n+1},\delta  q_{n+1})\|_{X_T}
			\leq
			C_0 C_4\|\Lambda_h^{s-1}\nabla\xi_0\|_{H^1(\Omega)}\|(\delta \xi_{n,(m+1)}^1,\delta  v_{n,(m+1)},\delta  q_{n,(m+1)})\|_{X_T}  \\
			&\quad  +2C_0C_4\big(   \|\xi_0^1\|_{H^{s}{(\Sigma_0)}}  + \|\Lambda_h^{s-1} v_0\|_{H^1(\Omega)}    \big) \big(\|\Lambda_h^{s-1} \nabla \delta  \xi_n \|_{L^{\infty}(H^1)} +\|\Lambda_h^{s-1}  \nabla\partial_t\delta  \xi_n \|_{L^2(L^2)} \big).
		\end{split}
		\]
        Since
        \[
        \begin{split}
       & \|\Lambda_h^{s-1} \nabla \delta  \xi_n \|_{L^{\infty}_T(H^1)}
        \leq T^{\frac{1}{2}}\|\Lambda_h^{s-1} \nabla \delta  v_n \|_{L^{2}_T(H^1)}, \ \|\Lambda_h^{s-1}  \nabla\partial_t\delta  \xi_n \|_{L^2_T(L^2)}
        \leq  T^{\frac{1}{2}}\|\Lambda_h^{s-1} \nabla \delta  v_n \|_{L^{\infty}_T(L^2)},
        \end{split}
        \]
        we have
        \[
		\begin{split}
\|(\delta \xi_{n+1}^1,\delta  v_{n+1},\delta  q_{n+1})\|_{X_T}
			&\leq
			C_0 C_4\|\Lambda_h^{s-1}\nabla\xi_0\|_{H^1(\Omega)}\|(\delta \xi_{n,(m+1)}^1,\delta  v_{n,(m+1)},\delta  q_{n,(m+1)})\|_{X_T}  \\
			&+2C_0C_4\big(   \|\xi_0^1\|_{H^{s}{(\Sigma_0)}}  + \|\Lambda_h^{s-1} v_0\|_{H^1(\Omega)}    \big)T^{\frac{1}{2}}\|(\delta \xi_{n}^1,\delta  v_{n},\delta  q_{n})\|_{X_T}.
		\end{split}
		\]
		Let  $
		\|\Lambda_h^{s-1}\nabla\xi_0\|_{H^1(\Omega)}\leq \frac{1}{3C_0C_4}$
		and
		$ T\leq \frac{1}{8C_0C_4(   \|\xi_0^1\|_{H^{s}{(\Sigma_0)}}  + \|\Lambda_h^{s-1} v_0\|_{H^1(\Omega)})},$
		then  we obtain
		\[
		\begin{split}
			\|(\delta \xi_{n+1}^1,\delta  v_{n+1},\delta  q_{n+1})\|_{X_T}
			\leq\frac{1}{2}
			\|(\delta \xi_{n}^1,\delta v_{n},\delta q_{n})\|_{X_T},
		\end{split}
		\]
which follows that the sequence $\{(\xi_{n}^1,v_n,q_n)\}_{n=1}^{\infty}$ is a Cauchy sequence in $X_T$. Hence, there exists $(\xi^1, v,q)\in X_T$ so that
		\[
		(\xi^1_n,v_n,q_n)\to (\xi^1,v,q) \text{ in }X_T.
		\]
		With $v$ in hand, we can solve $\xi$ in $\Omega$ by
		\begin{equation*}
			\begin{cases}
				\partial_t\xi    =v&\quad\text{in } (0, T] \times \Omega,\\
				\xi|_{t=0} = \xi_0 &\quad\text{in }\Omega
			\end{cases}
		\end{equation*}
		to obtain $\Lambda_h^{s-1}\nabla\xi\in \mathcal{C}([0,T]; H^1(\Omega))$. Finally, we can readily verify that $(\xi,v,q)$ is a solution of \eqref{EB-multi-1.1} and satisfies \eqref{est-main}.
		
       This finishes the proof of existence part of Theorem \ref{thm-main-1}.
		\subsection{Uniqueness part of the proof of Theorem \ref{thm-main-1}}
	Let $({\xi}, {v}, {q})$ and $(\widetilde{\xi}, \widetilde{v},\widetilde{q})$  be two solutions to \eqref{fixed-point} with the initial data $({\xi}_0, {v}_0)$ and $( \widetilde{\xi}_0, \widetilde{v}_0)$ respectively , and satisfy
       \begin{equation*}
		\begin{split}
			\Lambda_h^{s-1}(\nabla{\xi},\,\nabla\widetilde{\xi} ) &\in \mathcal{C}([0,T]; H^1(\Omega)), \,({\xi}^1,\,\widetilde{\xi}^1) \in \mathcal{C}([0,T]; H^{s}(\Sigma_0)), \\
			\Lambda_h^{s-1}({v},\,\widetilde{v}) &\in \mathcal{C}([0,T];\, {_0}{H}^1(\Omega)) \cap L^2(0,T; H^2(\Omega)),\quad \Lambda_h^{s-1}({q},\,\widetilde{q})\in L^2(0,T; H^1(\Omega)).
		\end{split}
		\end{equation*}
Set $\delta \xi\eqdefa \xi-\widetilde{\xi}$, $\delta v=v-\widetilde{v}$, and  $\delta q=q-\widetilde{q}$.
Then we have
		\begin{equation*}\label{uniq-stab}
			\begin{cases}
				\partial_{t}\delta\xi -\delta v =0,&\quad\text{in }\Omega,\\
				\partial_t\delta v -\nu\nabla\cdot\mathbb{D}(\delta v )+\nabla \delta q = \delta F_1 , &\quad\text{in }\Omega,\\
				\nabla\cdot \delta v =\delta F_2 ,&\quad\text{in }\Omega,\\
				\delta q n_0-g\delta \xi^1 n_0-\nu\mathbb{D}(\delta v )n_0=\delta F_3 ,&\quad\text{on }\Sigma_0,\\
				\delta v =0,&\quad\text{on }\Sigma_b,\\
				(\delta \xi,\delta v )|_{t=0} = (\delta\xi_0,\delta v_0),&\quad\text{in }\Omega,
			\end{cases}
		\end{equation*}
		where $\delta F_1:= F_1(\nabla\xi,v,q) - F_1(\nabla \widetilde{\xi}, \widetilde{v},\widetilde{q})$, $\delta F_2:=  B^1(\nabla\xi) :\nabla v
    - B^1(\nabla\widetilde{\xi}) :\nabla \widetilde{v} $, $\delta F_3:=B^2(\nabla\xi) :\nabla_h v  - B^2(\nabla\widetilde{\xi}) :\nabla_h \widetilde{v}$.
     If $\delta  \xi_0=0$ and $\delta v_0=0$ in $\Omega$, then, according to Theorem \ref{thm-linear-problim-11}, we obtain
		\begin{equation*}\label{estim-u-s-11}
		\begin{split}
			  \|(\delta \xi^1,  \delta v  , \delta q)\|_{X_t}
\leq& C_0  \mathfrak{C}\,t^{\frac{1}{2}} \|(0,\delta v,  \delta q)\|_{X_t}
		\end{split}
		\end{equation*}
       for any $0\leq t\leq  T $, where the positive constant $\mathfrak{C}$ depends on $\|\Lambda_h^{s-1}(\nabla \xi,\nabla\tilde{\xi})\|_{L_T^{\infty}(H^1(\Omega))}$, $\|(0,v,q)\|_{X_T}$ and $\|(0,\tilde{v},\tilde{q})\|_{X_T}$.
       Let $0 \leq t\leq t_0$ with $t_0\leq\frac{1}{4\mathfrak{C}^2C^2_0}$, there holds
       \begin{equation}\label{unq-7/8}
		\begin{split}
			  \|(\delta \xi^1,  \delta v  , \delta q)\|_{X_t}
\leq&  \frac{1}{2}  \|(0,\delta v,  \delta q)\|_{X_t}.
		\end{split}
		\end{equation}
       which implies that
       $(\delta \xi^1,  \delta v  , \delta q)=0$ in $[0,t_0]$, and then $\nabla\delta \xi=0$ in $[0,t_0]$ from \eqref{fixed-point}$_1$. The uniqueness of such strong solutions on the time interval $[0, T]$ then follows by a bootstrap argument. At the same lines, we may get that the solution $({\xi}, {v}, {q})$ to \eqref{fixed-point} depends continuously on the initial data $({\xi}_0, {v}_0)$ in $[0,T]$.
      We then complete the proof of uniqueness. 

\renewcommand{\theequation}{\thesection.\arabic{equation}}
\setcounter{equation}{0}


\section{Appendix}\label{sect-appendix-1}

\subsection{The expressions of $B$ forms}

We denote $B^1:\nabla v=B^1_{i,j}\partial_i v^j$ $(i,j=1,2,3)$ and $(B^2:\nabla_h v)^j=B_{\alpha,i}^{2, j}\partial_{\alpha} v^i$ $(i,j=1,2,3,\alpha=2,3)$.
   Here, $B^1:=\mathbb{I}-a_{11}^{-1}J\,\mathcal{A}$
    and
    $B^2$ is a $3\times 2\times 3$ tensor defined by
    \begin{equation*}\label{prtv-interface-2}
  \begin{split}
  &B^{2, 2}:\nabla_h v \eqdefa -\nu a_{11}^{-1}|\overrightarrow{\rm{a}_1}|^{-4}\bigg[-a_{11}(a_{11}^2+a_{31}^2-|\overrightarrow{\rm{a}_1}|^{4})\partial_2v^1\\
  &\qquad
  +(a_{11}^2+a_{31}^2) (a_{21}\mathcal{B}_{1}-a_{11}\mathcal{B}_{2}):\nabla_hv +|\overrightarrow{\rm{a}_1}|^{2}a_{21}a_{11}^2(-\nabla_h\cdot\,v^h+B_{h}:\nabla_h v) \\
  &\qquad-a_{21}a_{31} (- a_{11} \partial_3 v^1+ (a_{31}\mathcal{B}_{1 } -a_{11}\mathcal{B}_{3 } ) :\nabla_hv  )\bigg],\\
\end{split}
\end{equation*}
  \begin{equation*}
  \begin{split}
&B^{2, 3}:\nabla_h v\eqdefa   -\nu a_{11}^{-1}|\overrightarrow{\rm{a}_1}|^{-4}\bigg[-a_{11}(a_{11}^2+a_{21}^2-|\overrightarrow{\rm{a}_1}|^{4})\partial_3v^1\\
&\qquad+(a_{11}^2+a_{21}^2) (a_{31}\mathcal{B}_{1}-a_{11}\mathcal{B}_{3}) :\nabla_hv+|\overrightarrow{\rm{a}_1}|^{2}a_{31}a_{11}^2(-\nabla_h\cdot\,v^h+B_{h}:\nabla_hv) \\
  &\qquad-a_{21}a_{31} (-a_{11}\partial_2v^1+ (a_{21}\mathcal{B}_{1} -a_{11}\mathcal{B}_{2 } ):\nabla_hv  )\bigg].
\end{split}
\end{equation*}
and
   \begin{equation*}\label{prtv-interface-3}
  \begin{split}
&B^{2, 1}:\nabla_hv\eqdefa -\nu a_{11}^{-1}|\overrightarrow{\rm{a}_1}|^{-2}J^{-1}
\bigg[-2a_{21} |\overrightarrow{\rm{a}_1}|^{2}J\partial_2 v^1
-2a_{31}|\overrightarrow{\rm{a}_1}|^{2}J\partial_3 v^1\\
&\qquad+2J( a_{21} |\overrightarrow{\rm{a}_1}|^{2}B^{2, 2}+ a_{31} |\overrightarrow{\rm{a}_1}|^{2}B^{2, 3}+a_{11}|\overrightarrow{\rm{a}_1}|^{2} B_{h} ):\nabla_h v ,\\
&\qquad +2 a_{11}|\overrightarrow{\rm{a}_1}|^{2}( (1-J)\nabla_h\cdot\,v^h
- \mathcal{B}_{4 } :\nabla_h v )
-  a_{11} (a_{\beta\,1}\partial_{\beta} v^1 + a_{j1}\mathcal{B}_{j}:\nabla_h v)\bigg].
\end{split}
\end{equation*}
where $\overrightarrow{\rm{a}_1}:=(a_{11}, a_{21}, a_{31})^T$,
 \begin{equation*}
       \begin{split}
			 \mathcal{B}_{1} :\nabla_h v:= \left(
          \begin{array}{cc}
            a_{i2}a_{i1}+a_{12}a_{11} &    a_{i3}a_{i1}+a_{13}a_{11} \\
            a_{12}a_{21}             &  a_{13}a_{21} \\
            a_{12}a_{31}           &  a_{13}a_{31}
          \end{array}
          \right):
          \left(
          \begin{array}{ccc}
           \partial_2v^1 &  \partial_3v^1 \\
           \partial_2v^2 &  \partial_3v^2 \\
           \partial_2v^3 &  \partial_3v^3
          \end{array}
          \right),
       \end{split}
	\end{equation*}

		 \begin{equation*}
       \begin{split}
			 \mathcal{B}_{2} :\nabla_h v:= \left(
          \begin{array}{cc}
            a_{22}a_{11}-1    &    a_{23}a_{11}  \\
            a_{i2}a_{i1}+a_{22}a_{21}  &  a_{i3}a_{i1}+a_{23}a_{31} \\
            a_{22}a_{31}  & a_{23}a_{31}
          \end{array}
          \right):
          \left(
          \begin{array}{ccc}
           \partial_2v^1 &  \partial_3v^1 \\
           \partial_2v^2 &  \partial_3v^2 \\
           \partial_2v^3 &  \partial_3v^3
          \end{array}
          \right),
       \end{split}
	\end{equation*}

   \begin{equation*}
       \begin{split}
			 \mathcal{B}_{3} :\nabla_h v := \left(
          \begin{array}{cc}
            a_{32}a_{11}    &    a_{33}a_{11}-1  \\
            a_{32}a_{21}  &     a_{33}a_{21} \\
            a_{i2}a_{i1}+a_{32}a_{31}  &  a_{i3}a_{i1}+a_{33}a_{31}
          \end{array}
          \right):
          \left(
          \begin{array}{ccc}
           \partial_2v^1 &  \partial_3v^1 \\
           \partial_2v^2 &  \partial_3v^2 \\
           \partial_2v^3 &  \partial_3v^3
          \end{array}
          \right),
       \end{split}
	\end{equation*}

   \begin{equation*}
       \begin{split}
			 \mathcal{B}_{4 } :\nabla_h v:= \left(
          \begin{array}{cc}
           -a_{12}        &   -a_{13}  \\
           -a_{22}+1      &   -a_{23}  \\
           -a_{32}        &   -a_{33}+1
          \end{array}
          \right):
          \left(
          \begin{array}{ccc}
           \partial_2v^1 &  \partial_3v^1 \\
           \partial_2v^2 &  \partial_3v^2 \\
           \partial_2v^3 &  \partial_3v^3
          \end{array}
          \right),
       \end{split}
	\end{equation*}
and
    \begin{equation*}
       \begin{split}
			 B_{h}:\nabla_h v := -a_{11}^{-1}\left(
          \begin{array}{cc}
            a_{12}            &     a_{13} \\
             a_{22}-a_{11}   &     a_{23} \\
             a_{32}          &    a_{33}-a_{11}
          \end{array}
          \right):
          \left(
          \begin{array}{ccc}
           \partial_2v^1 &  \partial_3v^1 \\
           \partial_2v^2 &  \partial_3v^2 \\
           \partial_2v^3 &  \partial_3v^3
          \end{array}
          \right).
       \end{split}
	\end{equation*}

	\bigbreak \noindent {\bf Acknowledgments.}
		G. Gui's research is supported in part by the National Natural Science Foundation of China under Grants 12371211 and 12126359. Y. Li is partially supported by the Innovation Project of Hunan Province under Grant CX20230602.

	\end{document}